\documentclass[11pt]{amsart}
\usepackage{airy-hodge}
\newcommand{\vide}[1]{}

\begin{document}
\title{Hodge properties of Airy moments}

\author[C.~Sabbah]{Claude Sabbah}
\address[C.~Sabbah]{CMLS, CNRS, École polytechnique, Institut Polytechnique de Paris, 91128 Palaiseau cedex, France}
\email{Claude.Sabbah@polytechnique.edu}
\urladdr{https://perso.pages.math.cnrs.fr/users/claude.sabbah}

\author[J.-D. Yu]{Jeng-Daw Yu}
\address[J.-D. Yu]{Department of Mathematics, National Taiwan University,
Taipei 10617, Taiwan}
\email{jdyu@ntu.edu.tw}
\urladdr{http://homepage.ntu.edu.tw/~jdyu/}

\thanks{This research was partly supported by the PICS project TWN 8094 from CNRS}

\begin{abstract}
We consider the complex analogues of symmetric power moments of cubic expo\-nen\-tial sums. These are symmetric powers of the classical Airy differential equation. We~show that their de~Rham cohomologies underlie an arithmetic Hodge structure in the sense of Anderson and we compute their Hodge numbers by means of the irregular Hodge filtration, which is indexed by rational numbers, on their realizations as exponential mixed Hodge structures. The main result is that all Hodge numbers are either zero or one.
\end{abstract}



\maketitle
\tableofcontents
\mainmatter

\section{Introduction}
Families of cubic exponential sums attached to the polynomial $x^3+zx$ in $x$ over finite fields and their symmetric power moments over $z$
(Airy moments) have been considered in analogy with moments of Kloosterman sums for example, and much is already known about the corresponding $L$ functions (\cf \cite{Haessig09,H-RL11} and the references therein). For the analogous moments of Kloosterman sums, a conjecture of Broadhurst and Roberts, concerning the precise functional equations satisfied
by the corresponding global $L$ functions, has been proved in \cite{F-S-Y18} (except the exact local information at prime $2$ when $k$ is even).
In such a case, a pure Nori motive~$\Motive_k$ defined over $\QQ$ is attached to the $k$-moments.
One of the main arguments is to prove that
the associated Galois representation of $\Motive_k$
is potentially automorphic, which relies on a theorem of Patrikis and Taylor~\cite{PT15}, owing to the property proved in \cite{F-S-Y18}, that all nonzero Hodge numbers of $\Motive_k$ are equal to one.

For Airy moments, a strictly similar strategy to understand the Galois theoretic properties of the corresponding motivic object $\Motive_k$ is not possible, basically because $\Motive_k$ is not a Nori motive,
but an \emph{ulterior motive} in the sense of Anderson \cite{Anderson86}. Nevertheless, Hodge theory can be developed in this context (arithmetic Hodge structure in the sense of Anderson), with the caveat that the Hodge numbers $h^{\sfp,\sfq}$ may occur for rational exponents $\sfp,\sfq$ subject to $\sfp+\sfq\in\ZZ$.

A similar notion of Hodge structure, that we call a \emph{finite monodromic mixed Hodge structure} as defined in Section \ref{sec:mhsmu},
has been considered by Scherk and Steenbrink \cite{S-S85} when analyzing the join of two mixed Hodge structures equipped with an automorphism of finite order. Already in \cite{Steenbrink77}, Steenbrink attached to the mixed Hodge structure on the nearby cycles and vanishing cycles of a holomorphic function with an isolated singularity, a Hodge polynomial in $\ZZ[t^{1/m},t^{-1/m}]$, where~$m$ is the order of the semi-simple part of the monodromy.
He conjectured a behaviour of the Hodge polynomial
under Thom-Sebastiani sums as modeled in~\cite{S-S85}, a conjecture which was proved by M.\,Saito \cite{MSaito91d} and later received a proof using motivic integration in \cite{G-L-M06}.

The complex analogue of the Airy $k$-moment is the $k$-th symmetric product $\Sym^k\Ai$ of the Airy differential equation $\Ai$ on the affine line $\Afu_z$, defined by the classical Airy operator $\partial_z^2-z$. We regard $\Ai$ as a rank-two vector bundle on $\Afu_z$ with a connection having $\infty$ as its only singularity, which is irregular of slope $3/2$ and has irregularity number equal to $3$. Therefore, $\Sym^k\Ai$ is regarded as a rank $(k+1)$ vector bundle with connection on $\Afu_z$ having a singularity at infinity only. Contrary to the case of the Kloosterman connection, the de~Rham cohomology $\coH^1_\dR(\Afu,\Sym^k\Ai)$ does not naturally underlie a mixed Hodge structure (again reflecting the fact that $\Motive_k$ is not classical).
However, it is the de~Rham fiber of an exponential mixed Hodge structure with finite monodromy, hence underlies a finite monodromic mixed Hodge structure
(see Section \ref{subsec:finitemonoexp}).
As indicated above, such structures provide an alternative approach to the arithmetic Hodge structures of Anderson, and they are more closely related to exponential mixed Hodge structures defined by Kontsevich and Soibelman \cite{K-S10}. As in \cite{F-S-Y18}, we recover the $\QQ$-indexed Hodge filtration of this finite monodromic mixed Hodge structure as the irregular Hodge filtration on the de~Rham fiber $\coH^1_\dR(\Afu,\Sym^k\Ai)$. The main result of this article is the computation of the corresponding irregular Hodge numbers, opening the way to applying the results of Patrikis and Taylor on the arithmetic of the underlying ulterior motive $\Motive_k$.

\begin{thm}\label{thm:Hodge_numbers}
Let $k$ be an integer $\geq2$ and set $k'=\flr{(k-1)/2}$.
Then the de~Rham cohomology $\coH^1_\dR(\Afu,\Sym^k\Ai)$ naturally underlies
a finite monodromic mixed Hodge structure and the Hodge numbers
$h^{\sfp\sfq} = \dim \gr_F^\sfp\gr^W_{\sfp+\sfq}\coH^1_\dR(\Afu,\Sym^k\Ai)$
are all equal to one for the following $(\sfp,\sfq)$
and vanish otherwise:
\begin{itemize}
\item
$k$ odd,
$\sfp+\sfq = k+1$
and $\sfp = \frac{1}{3}(k+2i)$, $1\leq i\leq k'+1$,
\item
$\ktwofour$, $\sfp+\sfq = k+1$
and $\min\{\sfp,\sfq\} = \frac{1}{3}(k+2i)$, $1\leq i \leq \sfrac{k'}{2}$,
\item
$\kfour$
and either $\sfp+\sfq=k+1$
with $\min\{\sfp,\sfq\} = \frac{1}{3}(k+2i)$, $1\leq i \leq \sfrac{(k'-1)}{2}$,
or $\sfp=\sfq=\sfrac{(k+2)}{2}$.
\end{itemize}
\end{thm}

We are mostly interested in the middle de~Rham cohomology, which is by definition the image $\image[\coH^1_{\dR,\rc}(\Afu,\Sym^k\Ai)\to\coH^1_\dR(\Afu,\Sym^k\Ai)]$. We show (Corollaries \ref{cor:midW} and \ref{cor:midWcl}) that, with respect to the finite monodromic mixed Hodge structure on $\coH^1_\dR(\Afu,\Sym^k\Ai)$, it is equal to the weight $(k+1)$ subspace:
\[
\coH^1_{\dR,\rmid}(\Afu,\Sym^k\Ai)=W_{k+1}\coH^1_\dR(\Afu,\Sym^k\Ai).
\]
As a consequence of Theorem \ref{thm:Hodge_numbers}, the middle de~Rham cohomology underlies a finite monodromic pure Hodge structure, and the corresponding Hodge numbers coincide with those of $\coH^1_\dR(\Afu,\Sym^k\Ai)$ if $\knotfour$, and are obtained by omitting the case $\sfp=\sfq=\sfrac{(k+2)}{2}$ if $\kfour$.

\subsubsection*{Organization of the paper}
The approach of the present article is much inspired from that of \cite{F-S-Y18} on moments of Kloosterman connections. Sections \ref{sec:mhsmu}--\ref{sec:fmEMHS} consist of a review of known properties of various tools that we use, together with complementary results in presence of an action of the group $\mf=\varprojlim_m \ZZ/m\ZZ$. In Section \ref{sec:fmEMHS},
we emphasize the decomposition of a finite monodromic exponential mixed Hodge structure
(that we call a $\mf$-exponential mixed Hodge structure for short)
into its classical and non-classical parts. While the former can be treated with the same methods as in \cite{F-S-Y18}, the latter shows new Hodge-theoretic phenomena.

The main example of such an object is that attached to a regular function on $\Afu_s\times V$ of the form $s^mg$ ($m\geq2$), where $g:V\to\Afu$ is a regular function on a smooth quasi-projective variety. The corresponding non-classical $\mf$-exponential mixed Hodge structure (Proposition~\ref{prop:HrUMf}) can be expressed by means of the mixed Hodge structure of the covering of $V\moins g^{-1}(0)$ which is cyclic of degree $m$ and ramified along $g^{-1}(0)$.
While this expression makes clear the ulterior motivic origin of this finite monodromic mixed Hodge structure, it does not help for computing Hodge numbers in Theorem~\ref{thm:Hodge_numbers}. In order to handle the case of even $k$ in the theorem, we~provide in Section \ref{subsec:additive_convolution} an expression in terms of additive convolution with Kummer sheaves, which reduces the computation to a local (monodromic) computation similar to that used for the classical component modeled on that of \cite{F-S-Y18}.

In Section \ref{sec:settings} we provide Hodge-theoretic properties for the moments of the generalized Airy equation of order $n\geq2$ (the case $n=2$ being our main concern),
yielding the finite monodromic mixed Hodge structure
on the de~Rham cohomologies of their symmetric powers $\Sym^k\Ai_n$, and an estimation of their weights. This is a variant of the techniques used in \cite{F-S-Y18}.

Once all these preliminaries are settled, the proof of Theorem \ref{thm:Hodge_numbers} is achieved in Section~\ref{sec:proof_Hodge_numbers}. As in \loccit, we play with two expressions of the de~Rham cohomology, one computed in dimension one like $\coH^1(\Afu,\Sym^k\Ai)$ and the other one using the geometry of the function $s^3\sum_{i=1}^k(\frac13x_i^3-x_i)$. The former provides in a simple way a basis of the de~Rham cohomology and an associated filtration which satisfies the conclusion of the theorem, and the latter relates the computation of Hodge numbers to that of the corresponding \emph{irregular Hodge filtration} on a suitable blow-up of $\PP^{k+1}$. While the geometry is simple enough to enable us to conclude the proof if $k$ is odd by showing that the basis filtration is contained in the irregular Hodge filtration, we need supplementary arguments to conclude in the case of even $k$, which occupies Sections \ref{subsec:synopsis}--\ref{subsec:HodgeAitilde}. Although the idea is similar to that of \cite{F-S-Y18}, we need here to make use of a Hodge-theoretic inverse stationary phase formula, as we already did in \cite{S-Y18}, to recover enough information on the Hodge filtration.

\subsubsection*{Acknowledgements}
We thank Javier Fresán for useful conversations at the beginning of this project. We thank the referee for useful comments and remarks on the first version of this article.

\section{Finite monodromic mixed Hodge structures}
\label{sec:mhsmu}

We review and adapt results of \cite[p.\,661]{S-S85}. Let $\mf = \varprojlim_m \ZZ/m\ZZ$
be the profinite completion of the abelian group $\ZZ$.
Let $\Vect(\QQ)$ be the category of finite-dimensional $\QQ$-vector spaces and let $\Vect^\mf(\QQ)$ denote the category
consisting of a finite-dimensional $\QQ$-vector space $H_\QQ$
together with an action of $\mf$ through a finite quotient $\ZZ/m\ZZ$ for some $m$.
Equivalently,
by setting $T$ to be the action of the topological generator $1\in\mf$,
an object of $\Vect^\mf(\QQ)$ is a pair $(H_\QQ,T)$
consisting of the space $H_\QQ$ together with an automorphism~$T$ of finite order.
Each group ring $\QQ[\ZZ/m\ZZ]$ has the orthogonal decomposition
\[ \QQ[\ZZ/m\ZZ] = e_1\QQ[\ZZ/m\ZZ]\oplus e_{\neq 1}\QQ[\ZZ/m\ZZ] \]
where $e_1 = \frac{1}{m}\sum_{g\in\ZZ/m\ZZ} g$
is the projector into the invariant part and $e_{\neq 1} = 1-e_1$,
and the decomposition is compatible with the quotient
$\ZZ/m'\ZZ \to \ZZ/m\ZZ$ for $m\mid m'$.
Correspondingly for a given object $(H_\QQ,T)$ of $\Vect^\mf(\QQ)$,
there is a canonical decomposition in $\Vect^\mf(\QQ)$
\[
(H_\QQ,T)=(H_{\QQ,1},\id)\oplus (H_{\QQ,\neq1},T),
\]
where $T$ has no non-trivial invariant in $H_{\QQ,\neq 1}$.
Let $\mu_m \subset\CC^\times$ be the cyclic subgroup of order $m$
and $\mu_m^* = \mu_m\setminus\{1\}$.
In terms of the eigenvalue decomposition
$H_\CC = \bigoplus_{\zeta\in\mu_m}H_\zeta$
where $H_\CC = H_\QQ\otimes\CC$
and $T = \zeta\id$ on $H_\zeta$,
one has
$H_{\QQ,1}\otimes\CC = H_1$
and $H_{\QQ,\neq 1}\otimes\CC = \bigoplus_{\zeta\in\mu_m^*}H_\zeta$.
Notice that $\ol{H_\zeta} = H_{1/\zeta}$
under complex conjugation. Finally the category $\Vect^\mf(\QQ)$
is equipped with its natural tensor structure.

Let $\MHS$ be the category of (graded polarizable) mixed Hodge structures $(H_\QQ,F^\cbbullet H_\CC,W_\bbullet H_\QQ)$.
We define the category $\MHS(\mf)$ of
\textit{finite monodromic mixed Hodge structures}
as the category consisting of objects of $\MHS$
equipped with an action of
$\mf$ through the quotient $\ZZ/m\ZZ$ for some $m$
(equivalently, with an automorphism $T$ of finite order as above).
Morphisms are morphisms of mixed Hodge structures compatible with the group action. In order to express in a simple way the tensor product in this category,
we consider the category $\MHS^\mf$ of $\mf$-mixed Hodge structures.
The objects of $\MHS^\mf$ take the form $((H_\QQ,T),F^\cbbullet_\mf H_\CC,W_\bbullet^\mf H_\QQ)$, where
\begin{itemize}
\item
$(H_\QQ,T)$ is an object of $\Vect^\mf(\QQ)$,
so that $H_\QQ=H_{\QQ,1}\oplus H_{\QQ,\neq1}$
and $H_\CC=\bigoplus_{\zeta\in\mu_m}H_\zeta$;
\item
$F^\sfp_\mf H_\CC=\bigoplus_{\zeta}F^\sfp_\mf H_\zeta$ is an exhaustive decreasing filtration indexed by $\sfp\in\QQ$, and for each $\zeta=\exp(-2\pii a)$ with $a\in(-1,0]\cap\QQ$, the filtration $F^\sfp_\mf H_\zeta$ jumps at most at $\ZZ-a$;
\item
$W_\bbullet^\mf H_\QQ= W_\bbullet^\mf H_1\oplus W_\bbullet^\mf H_{\neq1}$ is an exhaustive increasing filtration indexed by $\ZZ$,
\end{itemize}
satisfying the following property: setting
\begin{equation}\label{eq:mfH}
\begin{split}
F^p H_\CC&=\bigoplus_{\zeta\in\mu_m}F^{p-a}_\mf H_\zeta,\quad p\in\ZZ,\\
W_\ell H_\QQ&=W_\ell^\mf H_1\oplus W^\mf_{\ell+1} H_{\neq1},
\end{split}
\end{equation}
we impose that $(H_\QQ,F^\cbbullet H_\CC,W_\bbullet H_\QQ)$ is a (graded polarizable) mixed Hodge structure. Morphisms are the natural ones. With the induced automorphism $T$,
the latter becomes an object of $\MHS(\mf)$.

\begin{lemma}
The natural functor $\MHS^\mf\to\MHS(\mf)$ defined by \eqref{eq:mfH} is an equivalence of categories. Any morphism in $\MHS^\mf$ is bi-strict with respect to $F^\cbbullet_\mf,W_\bbullet^\mf$.
\end{lemma}

\begin{proof}
Let $((H_\QQ,F^\cbbullet H_\CC,W_\bbullet H_\QQ),T)$
be an object of $\MHS(\mf)$.
We note that the Hodge filtration is compatible with the eigenvalue decomposition of $H_\CC$ with respect to $T$, being stable by $T$, and a similar property for the weight filtration, so that the correspondence given by
\begin{equation}\label{eq:HmfH}
\begin{aligned}
F_\mf^\sfp H_\zeta&=F^{\sfp+a}H_\CC\cap H_\zeta\quad(\zeta=\exp(-2\pii a),\;a\in(-1,0]\cap\QQ,\;\sfp\in\ZZ-a),\\
W_\ell^\mf H_{\QQ,1}&=W_\ell H_\QQ\cap H_{\QQ,1},\quad W^\mf_\ell H_{\QQ,\neq1}=W_{\ell-1}H_\QQ\cap H_{\QQ,\neq1}
\end{aligned}
\end{equation}
defines the desired quasi-inverse functor owing to the condition that \eqref{eq:mfH} is a mixed Hodge structure.
The last statement follows from the similar one for $\MHS(\mf)$.
\end{proof}

The integers~$\ell$ such that $\gr_\ell^{W^\mf}H_\QQ\neq0$ are called the $\mf$-weights of the object of $\MHS^\mf$.

\begin{lemma}[Hodge symmetry and Hodge decomposition in $\MHS^\mf$]\label{rem:Hodgesymmf}
Setting $(\gr_\ell^{W^\mf}H_\CC)^{\sfp,\ell-\sfp}\!=\!\gr^\sfp_{F_\mf}\gr_\ell^{W^\mf}H_\CC$ ($\sfp\in\QQ$), we have the Hodge decomposition
\begin{starequation}\label{eq:decompmf}
\gr^{W^\mf}H_\CC\simeq\bigoplus_{\substack{\sfp,\sfq\in\QQ\\\sfp+\sfq\in\ZZ}}(\gr_{\sfp+\sfq}^{W^\mf}H_\CC)^{\sfp,\sfq}.
\end{starequation}%
Furthermore,
via the orthogonal decomposition $e_1+e_{\neq 1}$,
each object $((H_\QQ,T),F^\cbbullet_\mf H_\CC,W_\bbullet^\mf H_\QQ)$ in $\MHS^\mf$ decomposes canonically as
\[
((H_\QQ,T),F^\cbbullet_\mf H_\CC,W_\bbullet^\mf H_\QQ)=((H_\QQ,T),F^\cbbullet_\mf H_\CC,W_\bbullet^\mf H_\QQ)_1\oplus((H_\QQ,T),F^\cbbullet_\mf H_\CC,W_\bbullet^\mf H_\QQ)_{\neq1},
\]
where the first term belongs to the category $\MHS$.
\end{lemma}

\begin{proof}
Hodge symmetry on each $(\gr_\ell^WH_\QQ,F^\cbbullet \gr_\ell^WH_\CC)$ translates to
\[
\dim\gr^\sfp_{F_\mf}\gr_\ell^{W^\mf}H_\zeta=\dim\gr^{\ell-\sfp}_{F_\mf}\gr_\ell^{W^\mf}H_{1/\zeta},\quad \forall\zeta\in\CC^*,\;\sfp\in\QQ,\;\ell\in\ZZ.
\]
The proof is then straightforward.
\end{proof}

\begin{notation}\label{nota:mhsmf}
We denote the corresponding decomposition of $\MHS^\mf$ as
\[
\MHS^\mf=\MHS\oplus\MHS^\mf_{\neq1},
\]
and the component of an object on $\MHS$ is also called its \emph{classical component}.
\end{notation}

\begin{remark}\mbox{}
\begin{enumerate}
\item(The spectrum)
If $T$ has order $m$, the Hodge polynomial
\[
\sum_{\sfp\in\frac1m\ZZ}(\dim\gr^\sfp_{F_\mf}H_\CC)\cdot t^\sfp
\]
was considered by Steenbrink \cite{Steenbrink77} for the vanishing cycles of an isolated singularity of a holomorphic function and, in the case of a quasi-homogeneous isolated singularity, was compared to the Poincaré polynomial attached to the Newton filtration of the Jacobian quotient of the singularity.
\item
(Relation with Anderson's arithmetic Hodge structures)
Anderson has defined in \cite[\S6.1]{Anderson86} the notion of arithmetic Hodge structure. We will omit his condition (6.1.5) for the moment, which is related to the rational structure of the de~Rham component of the arithmetic Hodge structure. Any such structure is the direct sum of pure structures of some integral weight. Then the notion of pure arithmetic Hodge structure (without Condition (6.1.5) of \loccit) is related to that of a $\mf$-pure Hodge structure. Indeed, assume that
a finite-dimensional $\QQ$-vector space $H_\QQ$
with the decomposition
\[
H =H_\QQ\otimes\CC
=\bigoplus_{\substack{\sfp,\sfq\in\QQ\\\sfp+\sfq=w}}H^{\sfp,\sfq}
\]
is a pure arithmetic Hodge structure of weight $w\in\ZZ$. There exists a minimal finite set $A\subset(-1,0]$ such that $H^{\sfp,\sfq}\neq0$ implies $\sfp\in\bigcup_{a\in A}(\ZZ-a)$. Let us set $A^*=A\cap(-1,0)$. Hodge symmetry implies $A^*=-1-A^*$. For $a\in A$ and $\zeta=\exp(-2\pii a)$, one can define~$H_\zeta$ by summing the $H^{\sfp,\sfq}$'s with $\sfp\in\ZZ-a$. Let $H_{\neq 1} = \bigoplus_{\zeta\in\mu_m^*}H_\zeta$. Then, an arithmetic Hodge structure (without Condition~(6.1.5)) is a pure $\mf$-Hodge structure if and only if the decomposition $H=H_1\oplus H_{\neq1}$ is defined over $\QQ$: indeed, one defines the automorphism $T$ as being equal to $\zeta\id$ on $H_\zeta$, and the decomposition being defined over $\QQ$ is equivalent to $T$ being defined over $\QQ$.
\end{enumerate}
\end{remark}

We recall the notation for the tensor product of filtrations (indexed by $\QQ$ and decreasing, say) $(H',G')$ and $(H'',G'')$,
\cf \cite{S-Z85}:
\[
(G'\star G'')^n(H'\otimes H'')=\sum_{n_1+n_2=n}G'^{n_1}(H')\otimes G''^{n_2}(H'').
\]

\begin{defi}\label{def:tensorWF}
The tensor product
\[
((H,T),F^\cbbullet_\mf H,W^\mf_\bbullet H)=((H',T'),F^\cbbullet_\mf H',W^\mf_\bbullet H')\otimes((H'',T''),F^\cbbullet_\mf H'',W^\mf_\bbullet H'')
\]
is defined as usual by the formulas
\begin{align*}
(H,T)&=(H'\otimes H'',T'\otimes T'')\quad\text{(action of $\ZZ/(m'm'')\ZZ$)},\\
F^\cbbullet_\mf H&=(F_\mf H'\star F_\mf H'')^\cbbullet,\\
W^\mf_\bbullet H&=(W^\mf H'\star W^\mf H'')_\bbullet.
\end{align*}
\end{defi}

\begin{prop}\label{prop:tensMHSmf}
The tensor product of two objects of $\MHS^\mf$ is an object of $\MHS^\mf$, making $\MHS^\mf$, and hence $\MHS(\mf)$, an abelian tensor category.
The forgetful functor
$\MHS^\mf \to \Vect^\mf$ is a tensor functor.
\end{prop}

\begin{proof}
To check the first assertion, it is enough to check that
the two filtrations $F$ and $W$ on the tensor product deduced from $F_\mf$ and $W^\mf$ correspond to Scherk-Steenbrink's definition of the join \cite[p.\,661]{S-S85}, and to use the results therein. We have
\begin{equation}\label{eq:joint}
H_{\zeta}=\bigoplus_{\zeta'\zeta''=\zeta}H'_{\zeta'}\otimes H''_{\zeta''}\qand W^\mf_kH_{\zeta}=\bigoplus_{\zeta'\zeta''=\zeta}\sum_{i+j=k}(W^\mf_iH'_{\zeta'}\otimes W^\mf_jH''_{\zeta''}),
\end{equation}
where the sum over $i,j$ is taken in $H'_{\zeta'}\otimes H''_{\zeta''}$.
\begin{itemize}
\item
If $\zeta=1$, \eqref{eq:joint} reads
\[
W_kH_1=\sum_{i+j=k}(W_iH'_1\otimes W_jH''_1)\oplus\bigoplus_{\zeta'\neq1}\sum_{i+j=k-2}(W_iH'_{\zeta'}\otimes W_jH''_{1/\zeta'}).
\]

\item
If $\zeta\neq1$, \eqref{eq:joint} reads
\[
W_{k-1}H_\zeta=\sum_{i+j=k-1}\bigl[(W_iH'_1\otimes W_jH''_\zeta)\oplus(W_iH'_\zeta\otimes W_jH''_1)\bigr]\oplus \bigoplus_{\substack{\zeta',\zeta''\neq1\\\zeta'\zeta''=\zeta}}\sum_{i+j=k-2}W_iH'_{\zeta'}\otimes W_jH''_{\zeta''}.
\]
\end{itemize}
Therefore, \eqref{eq:joint} reads
\[
W_kH=\bigoplus_{\zeta',\zeta''}\sum_{i,j}W_iH'_{\zeta'}\otimes W_jH''_{\zeta''},
\]
where the summation is taken over all $i,j$ such that
\[
i+j=\begin{cases}
k&\text{if $\zeta'=1$ or $\zeta''=1$,}\\
k-2&\text{if $\zeta''=1/\zeta'\neq1$,}\\
k-1&\text{if $\zeta'\neq1$, $\zeta''\neq1$, and $\zeta'\zeta''\neq1$.}
\end{cases}
\]
This corresponds to the formula of \loccit, up to an obvious typo there.\footnote{One should replace the condition $a=b=0$ with $a$ or $b=0$.} The proof for $F^\cbbullet$ is similar: the formula of \loccit\ is
\[
F^pH_\zeta=\bigoplus_{\substack{\zeta',\zeta''\\\zeta'\zeta''=\zeta}}\sum_{k,\ell}F^kH_{\zeta'}\otimes F^\ell H_{\zeta''},
\]
where the sum is taken over pairs $k,\ell$ such that, setting $\zeta'=\exp(-2\pii a')$ and $\zeta''=\exp(-2\pii a'')$ with $a',a''\in(-1,0]$,
\[
p=\begin{cases}
k+\ell&\text{if }a'+a''\in(-1,0],\\
k+\ell+1&\text{if }a'+a''\in(-2,-1].
\end{cases}
\]
Let us set $a\in(-1,0]$ satisfy
\[
\begin{cases}
a'+a''&\text{if }a'+a''\in(-1,0],\\
a'+a''+1&\text{if }a'+a''\in(-2,-1].
\end{cases}
\]
Then we have in any case $p-a=(k-a')+(\ell-a'')$, as wanted. One can also check that the decomposition \eqref{eq:decompmf} behaves well by tensor product.
\end{proof}

\section{Perverse sheaves on the affine line}\label{sec:perv0mf}

\begin{notation}\label{nota:ij}
We will often use the following diagram involving the affine line (whose coordinate may take various names):
\[
\{0\}\Hto{i}\Afu\Hfrom{j}\Gm.
\]
\end{notation}

\subsection{\texorpdfstring{$\kk$}{k}-Perverse sheaves}
We equip the affine line $\Afu$ (coordinate~$\theta$) with its analytic structure, which we denote by~$\Afuan$.
However, for the sake of simplicity,
we simply denote it by $\Afu_\theta$ or $\Afu$ in this section.

We refer to \cite[Chap.\,2]{Katz96} for details on the next results. Let $\Vect(\kk)$ be the category of $\kk$\nobreakdash-vector spaces for some fixed field $\kk$, and let $\Perv(\kk)$ be the category of perverse sheaves on $\Afu$ of $\kk$-vector spaces. The field $\kk$ being fixed, we simply write $\Vect,\Perv$. The additive convolution product $\star$ on $\Perv$ does not take values in $\Perv$. However, let $\Perv_0$ be the full subcategory of $\Perv$ whose objects consist of perverse sheaves with zero global cohomology. It~is equipped with a tensor structure given by $\star$ and there is a natural projector $\Pi:\Perv\to\Perv_0$ given by the additive convolution with $j_!\kk_{\Gm}$. There is a functorial morphism $\ccF\to\Pi(\ccF)$ in $\Perv$ whose kernel and cokernel are constant perverse sheaves.
For a perverse sheaf $\ccF$ on $\Afu_\theta$, we denote by $\psip_\theta\ccF$, \resp $\phip_\theta\ccF$,
the space of nearby, \resp vanishing, cycles at $\theta=0$, equipped with its monodromy operator. The functor $\phip_\theta$ decomposes with respect to eigenvalues of
the semi-simple part of the
monodromy as $\phip_{\theta,1}\oplus\phip_{\theta,\neq1}$
where the second subscripts indicate the range of the eigenvalues.
The two components behave differently when extended to mixed Hodge modules, and we have $\phip_{\theta,\neq1}=\psip_{\theta,\neq1}$.

The topological Fourier (or Laplace) transformations $\FT_\pm$, for which we refer to \cite[\S VI.2]{Malgrange91} and \cite[\S1.b]{Bibi05}, transform objects of $\Perv_0$ to perverse sheaves on~another copy of $\Afu$, with coordinate $\tau$, say, of the form $\bR j_*\ccL[1]$ for some local system $\ccL$ on $\Gm{}_\tau$. However, the datum of the perverse sheaf $\bR j_*\ccL_\pm[1]=\FT_\pm(\ccF)$ is in general not enough to recover $\ccF$, as Stokes data for $\bR j_*\ccL_\pm[1]$ at infinity on $\Afu_\tau$ are missing in this description. The case of monodromic perverse sheaves is simpler, as shown by Lemma \ref{lem:monodromicperverse} below, where $\FT$ is either the $+$ or the $-$ Fourier transformation, and $\FT^{-1}$ is respectively the $-$ or the $+$ Fourier transformation. For $\ccF_1,\ccF_2$ in $\Perv_0$, we have $j^{-1}\FT(\ccF_1\star\ccF_2)[-1]\simeq \ccL_1\otimes \ccL_2$. Taking the fiber at $1$ of $(\FT\ccF)[-1]$ is a fiber functor for the Tannakian category $\Perv_0$.

\subsection{Monodromic perverse sheaves}
Let $\Perv_0^\mon$ be the full subcategory of $\Perv_0$ consisting of perverse sheaves $\ccF$ in $\Perv_0$ whose shifted restriction $j^{-1}\ccF[-1]$ to $\Gm$ is a local system. It is a tensor subcategory of $\Perv_0$. The following is standard.

\begin{lemma}\label{lem:monodromicperverse}\mbox{}
\begin{enumerate}
\item\label{lem:monodromicperverse1}
The shifted restriction functor $j^{-1}(\cbbullet)[-1]$ induces an equivalence of categories
\[
\Perv_0^\mon\isom(\text{local systems on $\Gm$}),
\]
with $j_!(\cbbullet)[1]$ as a quasi-inverse.
\item\label{lem:monodromicperverse2}
The vanishing cycle functor $\phip_\theta$ at $\theta=0$ induces an equivalence between $\Perv_0^\mon$ and the category consisting of pairs $(V,T)$ of a finite-dimensional vector space equipped with an automorphism. This equivalence respects the tensor structures on both categories.
\item\label{lem:monodromicperverse3}
The functor $j^{-1}\circ\FT(\cbbullet)[-1]$ induces an equivalence
\[
\Perv_0^\mon\isom(\text{local systems on $\Gm$}),
\]
with quasi-inverse being given by $\FT^{-1}\circ\bR j_*(\cbbullet)[1]$. By this equivalence, the additive convolution product corresponds to the tensor product of local systems.
\end{enumerate}
\end{lemma}

\begin{proof}
Let us only indicate the proof of \eqref{lem:monodromicperverse2}. We can equivalently realize the category of pairs $(V,T)$ as the category of local systems $\ccL$ of finite-dimensional $\kk$-vector spaces on $\Gm$. Then we have a diagram of equivalences
\[
\xymatrix@C=1cm{
\Perv_0^\mon\ar@/_3pc/[rrr]^-{\phip_\theta}\ar@<.5ex>[rr]^-{j^{-1}(\cbbullet)[-1]}&&\ar@<.5ex>[ll]^(.55){j_!(\cbbullet)[1]}(\text{local systems on $\Gm$})&\ar[l]_-\sim(\text{pairs $(V,T)$}).
} \qedhere
\]
\end{proof}

One can identify non canonically the fiber at $1$ of $(\FT\ccF)[-1]$ with the vector space underlying $\phip_{\theta}\ccF$ by the following standard lemma (according to the definition of topological Fourier transformation).

\begin{lemma}
Let $\ccF$ be a perverse sheaf on the neighbourhood of a closed disc $\ov\Delta$ with no singularity on $\partial\Delta$, and let $\Delta^{>0}$ be the union of the open disc $\Delta$ and a nonempty open interval in its boundary. Then $\sum \dim\phip_{\theta-\theta_i}\ccF=\dim \coH^1_\rc(\Delta^{>0},\ccF)$, where the sum is taken over all singular points $\theta_i$ of $\ccF$ in~$\Delta$.\qed
\end{lemma}

\begin{remark}\label{rem:phipsi}
For an object $\ccF$ of $\Perv_0^\mon$, the canonical morphism $\can:\psip_\theta\ccF\to\phip_\theta\ccF$
from the space of nearby cycles
is an isomorphism (because of \ref{lem:monodromicperverse}\eqref{lem:monodromicperverse1}), so we can replace $\phip_\theta$ with $\psip_\theta$. However, in order to have compatibility with the tensor structures in \ref{lem:monodromicperverse}\eqref{lem:monodromicperverse2}, the functor $\phip_\theta$ is more convenient.
\end{remark}

\subsection{Constant perverse sheaves}
The full subcategory $\Perv_0^\cst$ of $\Perv_0^\mon$ consists of objects of $\Perv_0$ that are constant on $\Gm$. These are the perverse sheaves isomorphic to $j_!\kk_{\Gm}^r[1]$ for some $r\geq0$. The restricted Fourier transformation $j^{-1}\circ\FT$ sends $j_!\kk_{\Gm}^r[1]$ to $\kk_{\Gm}^r[1]$.

\begin{lemma}\label{lem:constperv}
The functors
\[
\phip_\theta=\phip_{\theta,1}:\Perv_0^\cst\mto\Vect\qand\Pi\circ i_!:\Vect\mto\Perv_0^\cst
\]
are mutually inverse equivalences of tensor categories.\qed
\end{lemma}

\subsection{Finite monodromic perverse sheaves}\label{subsec:pervmf}
Let $\Perv_0^\mf$ denote the full tensor subcategory of $\Perv_0^\mon$ consisting of objects $\ccF$ for which
there exists a cyclic covering
$[m]:\Afu_{\theta_m}\to\Afu_\theta$ of degree $m\geq1$, written in coordinates as $\theta_m\mto \theta=(\theta_m)^m$, such that
the pullback $[m]^{-1}\ccF$ is an object of $\Perv_0^\cst$.
An object in $\Perv_0^\mf$ is called a \emph{finite monodromic perverse sheaf}. By~specializing Lemma \ref{lem:monodromicperverse}, we obtain the following.

\begin{lemma}\label{lem:phitheta}
The functor
\[
\phip_\theta:\Perv_0^\mf\mto\Vect^\mf
\]
is an equivalence of tensor categories.
\end{lemma}

\begin{proof}
We give details on this proof, as a similar argument will be used for exponential mixed Hodge structures
(see Proposition \ref{prop:EMHSmfMHSmf}). By the equivalence of Lemma \ref{lem:monodromicperverse}\eqref{lem:monodromicperverse1}, we can replace $\Perv_0^\mf$ with the category of local systems $\ccL$ on $\Gm$ with finite monodromy, and by Remark \ref{rem:phipsi}, the functor to $\Vect^\mf$ is given, on a local system $\ccL$ such that $[m]^{-1}\ccL$ is constant, by $\ccL\mto(\Gamma(\Gm,[m]^{-1}\ccL),G)$, where the action of $G=\ZZ/m\ZZ$ is that of the Galois group of the covering $[m]:\Gm\to\Gm$.

Conversely, given $(V,G)$ in $\Vect^\mf$ where $G=\ZZ/m\ZZ$, we define a local system $\ccL$ with finite monodromy as follows. Let $(\ccV,G)$ be the constant local system on $\Gm$ with stalk $V$ and the induced $G$-action. The local system $[m]_*\ccV$ is equipped with the induced $G$-action and the Galois $G$-action. Then $\ccL$ is defined as the invariant subsheaf under the action of $\{(g^{-1},g)\mid g\in G\}$. One checks that both functors are quasi-inverse to each other by using the property that, given $(V,G)$ as above, and equipping $\kk^m$ with the action of $G$ induced by the cyclic permutation of the standard basis vectors, then if one equips $V\otimes\kk^m$ with its natural $G\times G$\nobreakdash-action, $(V,G)$ is isomorphic to $(V\otimes\kk^m)^\inv$ equipped with the induced $G$-action, where $\inv$ means invariants under the $\{(g^{-1},g)\mid g\in G\}$-action.

One can then express a quasi-inverse functor of the functor $\phip_\theta$ of the lemma as follows. Given $(V,G)$ in $\Vect^\mf$, one first associates with it the object of $\Perv_0^\cst$ with $G$-action defined by $(\Pi(i_!V),G)$, and then the object $\ccF$ of $\Perv_0^\mf$ defined as $([m]_*\Pi(i_!V))^\inv$, with the same meaning as above for $\inv$.
\end{proof}

\begin{remark}
Let $[m]:\Afu_{\theta_m}\to\Afu$ be as above. If $\ccF=j_!\ccL[1]$ belongs to $\Perv_0^\mon(\Afu)$, then
\[
[m]^{-1}\ccF=[m]^{-1}j_!\ccL[1]=j_{m!}([m]^{-1}\ccL)[1]
\]
also belongs to $\Perv_0^\mon(\Afu_{\theta_m})$.
For an object $\ccF$ in $\Perv_0^\mf$
such that $[m]^{-1}\ccF$ belongs to $\Perv_0^\cst$,
we can identify $\phip_\theta\ccF$
with the vector space $\phip_{\theta_m,1}[m]^{-1}\ccF$.
The monodromy on the latter space is the identity, but one recovers the action of the monodromy on $\phip_\theta\ccF$ by means of the action of the group of the covering. The drawback with the functor $\ccF\mto\phip_{\theta_m,1}([m]^{-1}\ccF)$ is that it is a priori not compatible with the tensor structures. On the other hand, the drawback with the functor~$\phip_\theta$ of Lemma \ref{lem:phitheta} is seen when extending it to mixed Hodge modules. Namely, the part $\phip_{\theta,\neq1}$ shifts the weights and makes $\phip_\theta$ not compatible with tensor product when considering weights. This is the reason for the definition of Section \ref{sec:mhsmu}.
\end{remark}

\section{Exponential mixed Hodge structures}\label{sec:EMHS}
In this section, we recall the construction due to Kontsevich and Soibelman \cite[\S4]{K-S10} of the exponential mixed Hodge structures $\coH^\cbbullet_\rc(U,f)$ and $\coH^\cbbullet(U,f)$ attached to a regular function $f:U\to\Afu$ on a smooth quasi-projective variety $U$. We emphasize the weight properties in Proposition \ref{prop:weightsHdUf}, together with the pure part $\coH^\cbbullet_\rmid(U,f)$.

\subsection{A review of exponential mixed Hodge structures}\label{subsec:reviewEMHS}
By the Riemann-Hilbert correspondence,
the definitions of $\Perv_0$ and of the functor $\Pi$
can be translated to regular holonomic $\cD_{\Afu}$-modules, and, according to \cite[\S4]{K-S10} both can be lifted in a compatible way to the category $\MHM(\Afu_\theta)$ of mixed Hodge modules on $\Afu_\theta$. This gives rise to the category $\EMHS$ of exponential mixed Hodge structures, which is a full subcategory of $\MHM(\Afu_\theta)$, and there is a projector $\Pi:\MHM(\Afu_\theta)\to\EMHS$, which is an exact functor, together with a functorial morphism $N^\rH\to\Pi(N^\rH)$ in $\MHM(\Afu_\theta)$ whose kernel and cokernel are constant mixed Hodge modules. Each object of $\EMHS$ has a perverse realization in $\Perv_0$ and the projector $\Pi:\MHM(\Afu_\theta)\to\EMHS$ realizes as $\Pi:\Perv\to\Perv_0$ of Section \ref{sec:perv0mf}. Each object of $\EMHS$ is equipped with a weight filtration in $\EMHS$, defined from that in $\MHM(\Afu_\theta)$ by the formula
\[
W^\sEMHS_\ell\Pi(N)=\Pi(W_\ell N)
\]
for any mixed Hodge module $N^\rH$ on~$\Afu_\theta$. The category $\EMHS$ is equipped with a tensor product, induced by the additive convolution on $\Afu_\theta$. This product is strictly compatible with the weight filtration~$W^\sEMHS_\bbullet$. (See \loccit, or \cite[\S A.3]{F-S-Y18} for more details.)

The unit in $\EMHS$ is $\Pi\circ \Hm i_!(\QQ^\rH)$, which is pure of weight zero as an object of $\EMHS$. Note that, since $\pQQ^\rH_{\Gmtheta}:=\QQ^\rH_{\Gmtheta}[1]$ corresponds to the constant variation of Hodge structure of weight zero on $\Gmtheta$, it has weight one as an object of $\MHM(\Gmtheta)$ and therefore
$\Pi\circ \Hm i_!(\QQ^\rH) = \Hm j_!\pQQ^\rH_{\Gmtheta}$
is mixed of weights $\leq1$ in $\MHM(\Afu_\theta)$.

Let $N^\rH$ be an object of $\MHM(\Afu_\theta)$ such that its underlying $\cD_{\Afu_\theta}$-module $N$ has semi-simple monodromy at the origin. The mixed Hodge structure $\Hm\phi_\theta(N^\rH)$, as defined by Saito \cite{MSaito86,MSaito87}, has a weight filtration~$W_\bbullet$ and a Hodge filtration $F^\cbbullet$ obtained from that of $N^\rH$ as follows (for~the weight filtration, we use that the monodromy is semi-simple). We recall that both $\phi_{\theta,1}N$ and $\psi_{\theta,\neq1}N$ are obtained by grading $N$ with respect to its Kashiwara-Malgrange filtration, so that it is meaningful to speak of a filtration (weight or Hodge) on these vector spaces \emph{induced} by a filtration (weight or Hodge) on $N$.
\begin{itemize}
\item
$F^\cbbullet \psi_{\theta,\neq1}N$ is induced by $F^\cbbullet N$ and $W_\bbullet(\psi_{\theta,\neq1}N)=\psi_{\theta,\neq1}(W_{\bbullet+1}N)=\psi_{\theta,\neq1}(W_\bbullet[-1]N)$,
\item
$F^\cbbullet \phi_{\theta,1}N$ is induced by $F^{\cbbullet-1}N=F^\cbbullet[-1]N$ and $W_\bbullet(\phi_{\theta,1}N)=\phi_{\theta,1}(W_\bbullet N)$.
\end{itemize}

\begin{remark}\label{rem:WWEMHS}
Even if $N^\rH=\Pi(N^\rH)$ belongs to $\EMHS$, the weight filtrations $W_\bbullet N$ and $W_\bbullet^\sEMHS N$ may differ, as already seen for $\Hm j_!\pQQ^\rH_{\Gmtheta}$. However, they induce the same filtration on $\phi_{\theta,1}N$ and $\psi_{\theta,\neq1}N$. Indeed, for any~$\ell$,
the kernel and cokernel of the natural morphism
$W_\ell N\to W_\ell^\sEMHS N$ are constant on $\Afu_\theta$,
hence their vanishing cycle space $\phip_\theta=\phip_{\theta,1}\oplus\psip_{\theta,\neq1}$ is zero. Therefore, in the definition above, we can replace $W_\bbullet$ with $W_\bbullet^\sEMHS$ and we regard $\Hm\phi_\theta$ as a functor $\EMHS\mto\MHS$.
\end{remark}

\subsection{Weight and irregular Hodge filtrations on the de~Rham fiber}\label{subsec:wfdr}
We recall here the general framework, for which we refer \eg to \cite[Chap.\,3]{S-Y18}. We denote by $\wh\theta$ the variable which is Fourier dual to $\theta$. Let $N^\rH$ be a mixed Hodge module on~$\Afu_\theta$, with $(N,F^\cbbullet N,W_\bbullet N)$ as the underlying bifiltered $\Cltheta$-module and $(\ccF,W_\bbullet\ccF)$ as the underlying filtered $\QQ$-perverse sheaf.

\subsubsection*{Weight filtration}
Fourier transformation sends these data to a filtered $\Clwtheta$-module $(\FT_\theta N,W_\bbullet\FT_\theta N)$, with $W_\bbullet\FT_\theta N:=\FT W_\bbullet N$, corresponding, via the de~Rham functor, to the filtered perverse sheaf $(\FT_\theta\ccF,W_\bbullet\FT\ccF)$ on~$\Afuan_{\wh\theta}$ already considered in Section \ref{sec:perv0mf}. Restricting $\FT_\theta N$, \resp $\FT_\theta\ccF$, to $\wh\theta=1$ yields the de~Rham fiber $\coH^1_\dR(\Afu_\theta,N\otimes E^\theta)$, \resp the Betti fiber $(\FT_\theta\ccF)_1$, which is a finite-dimensional vector space over $\CC$ \resp $\QQ$. Moreover, additive convolution $\star$ on $\Cltheta$-modules induces tensor product of de~Rham fibers.
By exactness of the de~Rham fiber functor,
which follows from exactness of $\FT_\theta$ and the $\CC[\wh\theta,\wh\theta^{-1}]$-freeness of $j^+\FT_\theta N$, the natural morphism
\[
\coH^1_\dR(\Afu_\theta,W_\bbullet N\otimes E^\theta)\to \coH^1_\dR(\Afu_\theta,N\otimes E^\theta)
\]
is injective and its image is a filtration $W_\bbullet\coH^1_\dR(\Afu_\theta,N\otimes E^\theta)$. Moreover, $W_\bbullet N^\rH$ and $W^\sEMHS_\bbullet N^\rH$ have the same de~Rham \resp Betti image (same argument as for $\phip_\theta$ in Remark \ref{rem:WWEMHS}).
We~thus have a tensor product behavior of the $W$-filtration analogous to that occurring in Definition~\ref{def:tensorWF}:
\[
W_\bbullet \coH^1_\dR(\Afu_\theta,(N'\star N'')\otimes E^\theta)\simeq W_\bbullet \coH^1_\dR(\Afu_\theta,N'\otimes E^\theta)\otimes W_\bbullet \coH^1_\dR(\Afu_\theta,N''\otimes E^\theta).
\]

\subsubsection*{Irregular Hodge filtration}
On the other hand, the Hodge filtration $F^\cbbullet N$ gives rise to the Deligne filtration $F_\Del^\cbbullet(N\otimes E^\theta)$ (\cf\cite[\S6]{Bibi08}), which is indexed by $\ZZ-A$ for some finite subset $A\subset[0,1)\cap\QQ$. The natural morphism
\[
\coH^1\bigl(\Afu_\theta,F_\Del^\cbbullet\DR(N\otimes E^\theta)\bigr)\to\coH^1_\dR(\Afu_\theta,N\otimes E^\theta)
\]
is injective (\cf\loccit). Its image is denoted by $F_\irr^\cbbullet\coH^1_\dR(\Afu_\theta,N\otimes E^\theta)$. The filtration $F_\irr^\cbbullet\coH^1_\dR(\Afu_\theta,N\otimes E^\theta)$ behaves like in Definition~\ref{def:tensorWF} by additive convolution (\cf\cite[Th.\,3.39]{Bibi15}).

\begin{prop}[{\cite[Prop.\,A.10]{F-S-Y18}}]\label{prop:bifilt}
The functor $N^\rH\mto\bigl(\coH^1_\dR(\Afu_\theta,N\otimes E^\theta),F_\irr^\cbbullet,W_\bbullet)$ from $\MHM(\Afu_\theta)$ to bifiltered vector spaces factors through $\Pi_\theta$ and any morphism in $\MHM(\Afu_\theta)$ (or~in $\EMHS$) yields a strictly bifiltered morphism.\qed
\end{prop}

\subsection{Object of \texorpdfstring{$\EMHS$}{EMHS} associated with a regular function}\label{subsec:objregfun}

We start by considering a~geometric situation, that we will reduce to a question on mixed Hodge modules on $\Afu$ by~taking a suitable pushforward. We recall the notation of \cite[\S A.5]{F-S-Y18}.

Let $f: U\to\Afu_\theta$ be a regular function on a smooth complex quasi-projective variety $U$ of dimension $d+1$ and let $\pQQ^\rH_U$ denote the constant pure Hodge module of rank one and weight~$d+1$ on~$U$, with associated $\cD_U$-module $(\cO_U,\rd)$ and associated perverse sheaf $\pQQ_U=\QQ_{U^\an}[d+1]$. For each $r\in\ZZ$, we consider the mixed Hodge modules in $\MHM(\Afu_\theta)$
\[
(\pQQ^\rH_U)^r_*:=\cH^{r-d-1}\Hm f_*\pQQ^\rH_U,\quad
(\pQQ^\rH_U)^r_!:=\cH^{r-d-1}\Hm f_!\pQQ^\rH_U.
\]
They define objects
\begin{equation}\label{eq:HUf}
\coH^r(U,f)=\Pi_\theta((\pQQ^\rH_U)^r_*)\qand\coH^r_\rc(U,f)=\Pi_\theta((\pQQ^\rH_U)^r_!)
\end{equation}
of $\EMHS$, and we set
\[
\coH^r_\rmid(U,f)=\image[\coH^r_\rc(U,f)\ra\coH^r(U,f)]=\image\bigl[\Pi_\theta((\pQQ^\rH_U)^r_!)\ra\Pi_\theta((\pQQ^\rH_U)^r_*)\bigr].
\]

More generally, let $N_U^\rH$ be a mixed Hodge module on $U$. We~define similarly
\begin{equation}\label{eq:NHUj}
(N^\rH_U)^r_*:=\cH^{r-d-1}\Hm f_*N^\rH_U,\quad
(N^\rH_U)^r_!:=\cH^{r-d-1}\Hm f_!N^\rH_U,
\end{equation}
and
\begin{equation}\label{eq:HUNf}
\coH^r(U,N_U^\rH,f)=\Pi_\theta((N^\rH_U)^r_*)\qand\coH^r_\rc(U,N_U^\rH,f)_\rc=\Pi_\theta((N^\rH_U)^r_!).
\end{equation}
Then we set
\[
\coH^r_\rmid(U,N_U^\rH,f)=\image[\coH^r_\rc(U,f)\ra\coH^r(U,N_U^\rH,f)]=\image\bigl[\Pi_\theta((N^\rH_U)^r_!)\ra\Pi_\theta((N^\rH_U)^r_*)\bigr].
\]

\begin{remark}\label{rem:EMHSMHStg}
By definition, the de~Rham fiber $\coH^r_{\dR,\rc}(U,N_U^\rH,f)$, \resp $\coH^r_{\dR}(U,N_U^\rH,f)$, is equal to $\coH^1_{\dR,\rc}(\Afu_\theta,(N_U)_!^r\otimes E^\theta)$, \resp $\coH^1_\dR(\Afu_\theta,(N_U)_*^r\otimes E^\theta)$, and $\coH^r_{\dR,\rmid}(U,N_U^\rH,f)$ is the image of the former to the latter. Under the isomorphism
\[
\coH^r_{\dR,?}(U,N_U^\rH,f)\simeq\coH^r_{\dR,?}(U,N_U\otimes E^{f})\quad(?=\emptyset,\rc,\rmid),
\]
the irregular Hodge filtration on the de~Rham fiber of $\coH^r_?(U,N_U^\rH,f)\in\EMHS$ is identified with the irregular Hodge filtration $F_\irr^\bbullet\coH^r_{\dR,?}(U,N_U\otimes E^{f})$ of the right-hand side, computed by means of a compactification of $f$ as a map $\ov U\to\PP^1$, and the structure morphism of $\ov U$ (\cf\cite[Th.\,1.3(4)]{S-Y14} and \cite[(A.16) \& (A.17)]{F-S-Y18}).
\end{remark}

The following weight properties are obtained in a way similar to that of \cite[Prop.\,A.19]{F-S-Y18}.

\begin{prop}\label{prop:weightsHdUf}
Assume that $N_U^\rH$ is pure of weight $w+1$. Then the exponential mixed Hodge structure $\coH_\rc^{d+1}(U,N_U^\rH,f)$ is mixed of weights $\leq w+1$, $\coH^{d+1}(U,N_U^\rH,f)$ is mixed of weights $\geq w+1$, and $\coH_\rmid^{d+1}(U,N_U^\rH,f)$ is pure of weight $w+1$. Moreover, if $N_U^\rH$ is self-dual, \ie $\bD N_U^\rH(-w-1)\simeq N_U^\rH$, the following properties are equivalent:
\begin{enumerate}
\item\label{prop:weightsHdUf1}
the natural morphism $\gr_{w+1}^W\coH_\rc^{d+1}(U,N_U^\rH,f)\to\gr_{w+1}^W\coH^{d+1}(U,N_U^\rH,f)$ is an isomorphism,
\item\label{prop:weightsHdUf2} the equality
$\coH_\rmid^{d+1}(U,N_U^\rH,f)=W_{w+1}\coH^{d+1}(U,N_U^\rH,f)$ holds.\qed
\end{enumerate}
\end{prop}

\section{Finite monodromic exponential mixed Hodge structures}\label{sec:fmEMHS}

This section continues Section \ref{sec:EMHS} in the case where $U$ takes the form $\Afu_s\times V$ and \hbox{$f=s^mg$} for some $m\geq1$ and some regular function $g:V\to\Afu$. We first define the category $\EMHS^\mf$ of finite monodromic exponential mixed Hodge structure. Then, in this particular case, $\coH^\cbbullet_?(U,f)$ \hbox{($?=\rc,\emptyset,\rmid$)} is an object of $\EMHS^\mf$. We make explicit in geometric terms its classical and non-classical components (Proposition \ref{prop:HrUMf}). The case where $g$ is moreover assumed to be tame is considered in Section \ref{subsec:tame}, where the main objective is to provide tools for computing the irregular Hodge filtration on the corresponding de~Rham cohomologies: for the classical component, Corollary \ref{cor:midWcl} enables us to directly use results of \cite[App.]{F-S-Y18}, while for the non-classical component, the main tool is provided by Corollary \ref{cor:grFconvolution}.

\subsection{Constant exponential mixed Hodge structures}
Let $\EMHS^{\cst}$ be the full subcategory of $\EMHS$ consisting of objects whose associated $\QQ$\nobreakdash-perverse sheaf belongs to $\Perv_0^\cst$. Since constant mixed Hodge modules on $\Gm$ are exactly those mixed Hodge modules whose associated perverse sheaf is constant (\cf \cite[Th.\,4.20]{S-Z85}), we can as well define $\EMHS^{\cst}$ as the full subcategory of $\EMHS$ consisting of objects whose restriction to $\Gm$ is a constant mixed Hodge module.

\begin{lemma}
The functors
\[
\Hm\phi_{\theta,1}:\EMHS^{\cst}\mto\MHS\qand\Pi\circ \Hm i_!:\MHS\mto\EMHS^{\cst}
\]
are mutually quasi-inverse equivalences of tensor categories strictly compatible with weight filtrations.
\end{lemma}

\begin{proof}
It follows from Lemma \ref{lem:constperv} that $\EMHS^{\cst}$ is the essential image of $\Pi\circ\, \Hm i_!$. The remaining statements were already observed in \cite[Lem.\,A.12]{F-S-Y18}.
\end{proof}

\subsection{The pure Hodge module \texorpdfstring{$\go{m}^\rH$}{LqtauH}}\label{subsec:KrH}
Before introducing the category $\EMHS^\mf$, let us consider an example. We keep the notation of Section \ref{subsec:pervmf}. Let $m$ be an integer $\geq2$ and let \hbox{$[m]:\Afu_{\theta_m}\to\Afu_\theta$} denote the cyclic ramification of order $m$ defined by $\theta_m\mto\theta=\theta_m^m$. Let $\pQQ_{\Afu_{\theta_m}}^\rH$ be the pure Hodge module of weight $1$ on $\Afu_{\theta_m}$ (equivalently, the constant polarized variation of Hodge structure of weight $0$ on $\Afu_{\theta_m}$). The pushforward $\Hm[m]_*\pQQ_{\Afu_{\theta_m}}^\rH$ decomposes as the direct sum
\[
\Hm[m]_*\pQQ_{\Afu_{\theta_m}}^\rH=\pQQ_{\Afu_\theta}^\rH\oplus \go{m}^\rH
\]
of two pure Hodge modules of weight $1$ on $\Afu$. We have $\go{m}^\rH=\Hm j_*\,\Hm j^*\go{m}^\rH$.
Moreover, the underlying $\cD$-module $j^*\go{m}$ is $\cO_{\Gm}$-free of rank $m-1$. The Hodge filtration satisfies $\gr^p_F\go{m}=0$ for $p\neq0$. The $\QQ$-local system $[m]_*\QQ_{\Afu_{\theta_m}}$ decomposes as $\QQ_{\Afu_\theta}\oplus\ccE_m$. We identify the space of multi-valued global sections of $\ccE_m$ with the hyperplane of
$\QQ^m = [m]_*\QQ_{\Afu_{\theta_m}}|_{\{1\}}$
defined by ``sum of entries equal to zero'', on which the monodromy $\theta\mto e^{2\pii\theta}$ acts as the automorphism induced by the cyclic permutation of the basis vectors of $\QQ^m$.

Since no eigenvalue of the monodromy on $\ccE_m$ is equal to $1$, the natural morphisms
\[
\Hm j_!\,\Hm j^*\go{m}^\rH\to\Hm j_{!*}\,\Hm j^*\go{m}^\rH\to\Hm j_*\,\Hm j^*\go{m}^\rH
\]
are isomorphisms. Lastly, the space of nearby cycles $\Hm\psi_{\theta,\neq1}\go{m}^\rH$ is a pure Hodge structure of rank $m-1$ and of weight $0$, with $\gr^p_F\Hm\psi_{\theta,\neq1} \go{m}^\rH=0$ if $p\neq0$.

The pure Hodge module $\go{m}^\rH$ decomposes as a direct sum of pure complex Hodge modules~$K_\zeta^\rH$ of rank one, indexed by the $m$-th roots of unity $\zeta$ distinct from $1$,
where the monodromy acts by multiplication by $\zeta$.

On the other hand, we regard $\go{m}^\rH$ as $\Pi(\Hm[m]_*\,\pQQ_{\Afu_{\theta_m}}^\rH)$ since $\Pi(\pQQ_{\Afu_\theta}^\rH)=0$, and this object is exponentially pure of weight one.

\subsection{Finite monodromic exponential mixed Hodge structures}\label{subsec:finitemonoexp}
We denote by $\EMHS^\mf$ the full subcategory of $\EMHS$ consisting of objects $N^\rH$ such that $\Pi\bigl(\Hm[m]^*N^\rH\bigr)$ belongs to $\EMHS^{\cst}$ for some finite ramified cyclic covering
$[m]:\Afu_{\theta_m}\to\Afu_\theta$. Equivalently, the pullback by $[m]:\Gmm\to\Gmtheta$ of the flat bundle $j^*N$ is constant on $\Gmm$.

On $\EMHS^\mf$, the functor $\Hm\psi_{\theta,\neq1}$ takes values in $\MHS(\mf)$ by its very definition (\cf \cite{MSaito87}), where the notation $\MHS(\mf)$ is explained in Section \ref{sec:mhsmu}. Composing $\Hm\psi_{\theta,\neq1}$ with the equivalence defined by \eqref{eq:HmfH} gives rise to a functor
\[
\Hm\psi_{\theta,\neq1}^\mf:\EMHS^\mf\to\MHS^\mf_{\neq1}.
\]
For example, for $m\geq2$, $\go{m}^\rH$ has pure (exponential) weight one and $\Hm\psi_{\theta,\neq1}^\mf \go{m}^\rH$ has also pure $\mf$-weight equal to one.

The vanishing cycle functor $\Hm\phi_{\theta,1}$ is not affected by this modification in the presence of the $\mf$-action, and the sum of both is denoted by $\Hm\phi_\theta^\mf$.

\begin{prop}\label{prop:EMHSmfMHSmf}
The functor
\[
\Hm\phi_\theta^\mf=(\Hm\phi_{\theta,1}\oplus\Hm\psi_{\theta,\neq1}^\mf):\EMHS^\mf\mto\MHS^\mf
\]
is an equivalence of tensor categories which is strictly compatible with weights.
\end{prop}

As a consequence, there is a decomposition $\EMHS^\mf=\EMHS^\cst\oplus\EMHS^\mf_{\neq1}$ corresponding to the decomposition of Notation \ref{nota:mhsmf}.

\begin{defi}[Classical component]
For an object $N^\rH$ of $\EMHS^\mf$, its \emph{classical component} $N^\rH_\cl\in\EMHS^\cst$ corresponds to the mixed Hodge structure $\Hm\phi_{\theta,1}N^\rH$. It is a direct summand of~$N^\rH$ in $\EMHS^\mf$.
\end{defi}

\begin{proof}[Proof of Proposition \ref{prop:EMHSmfMHSmf}]
We mimic the definition in the proof of Lemma \ref{lem:phitheta} to construct a quasi-inverse functor.
Let $(V^\rH,G)$ be an object of $\MHS(\mf)$, with $G=\ZZ/m\ZZ$. We associate with it the object
\[
\bigl(\Hm[m]_*(\Pi(\Hm i_!V^\rH))\bigr)^\inv,
\]
where $\inv$ means the invariant submodule with respect to the action of $G^{-1}\times G'$. The proof of the equivalence property is then similar to that of \loccit\ That $\Hm\phi_\theta^\mf$ is compatible with the tensor product of each category is similar to \cite[Th.\,(8.11)]{S-S85},\footnote{We can apply \loccit\ since the monodromy is finite, hence semi-simple, \cf\cite[Rem.\,6.6(ii)]{B-H18}.} and can also be deduced from the more general result \hbox{\cite[Th.\,1.2]{M-S-S16}}.
\end{proof}

\begin{remark}\label{rem:weightsEMHSmf}
For an object $N^\rH$ of $\EMHS^\mf_{\neq1}$ (\ie having a classical component equal to zero), we have $W_\bbullet^\sEMHS N^\rH=W_\bbullet N^\rH$ since the underlying $\Cltheta$-module $N$ does not have any constant sub-quotient. Together with Remark \ref{rem:WWEMHS}, we conclude that the same property holds for objects of $\EMHS^\mf$. From Propositions \ref{prop:tensMHSmf} and \ref{prop:EMHSmfMHSmf} we conclude that the weight filtration satisfies
\[
W_\ell(N^{\prime\rH}\star N^{\prime\prime\rH})=\sum_i W_iN^{\prime\rH}\star W_{\ell-i}N^{\prime\prime\rH},\quad N^{\prime\rH}, N^{\prime\prime\rH}\in\EMHS^\mf.
\]
In particular, if $N^{\prime\rH},N^{\prime\prime\rH}$ are pure of respective weights $w',w''$, then their convolution product is pure of weight $w'+w''$.
\end{remark}

\begin{prop}\label{prop:MHSmfirr}
Let $N^\rH$ be an object of $\EMHS^\mf$ and let $\bigl(\phi_{\theta}^\mf(N),F^\cbbullet_\mf,W^\mf_\bbullet\bigr)$ denote the bifiltered vector space underlying the associated $\mf$-mixed Hodge structure $\Hm\phi_{\theta}^\mf(N^\rH)$. Then there exists a functorial bifiltered isomorphism
\[
\bigl(\phi_{\theta}^\mf(N),F^\cbbullet_\mf,W^\mf_\bbullet\bigr)\simeq \bigl(\coH^1_\dR(\Afu_\theta,N\otimes E^\theta),F^\cbbullet_\irr,W_\bbullet\bigr)
\]
compatible with tensor products.
\end{prop}

\begin{proof}
Compatibility with tensor products follows from Propositions \ref{prop:bifilt} and \ref{prop:EMHSmfMHSmf}. In order to prove the existence of a bifiltered isomorphism, we can decompose $N^\rH$ as $N^\rH_\cl\oplus N^\rH_{\neq1}$. The case of $N^\rH_\cl$ has been explained in \hbox{\cite[Prop.\,B.5]{F-S-Y18}}. We only consider the case of $N^\rH_{\neq1}$. By grading with respect to $W_\bbullet$, we can assume it is pure and we are left with comparing the Hodge filtrations. We apply the results explained in \cite[\S5]{S-Y18}. Since the monodromy of $N_{\neq1}$ around $\theta=0$ does not have eigenvalue one, and since $N_{\neq1}$ is monodromic, the same property holds at $\theta=\infty$, and Formula (7) in \loccit\ shows that, for $\alpha\in(0,1)$ and $\zeta'=\exp(-2\pii\alpha)$, we have
\[
\dim\gr_{F_\irr}^{p+\alpha}\coH^1_\dR(\Afu_\theta,N\otimes E^\theta)=\dim\gr_F^p\psi_{1/\theta,\zeta'}N.
\]
Let us set $\zeta=\zeta^{\prime-1}=\exp(-2\pii a)$ with $a=-\alpha\in(-1,0)$. Since $N_{\neq1}^\rH$ is monodromic, we have
\[
\dim\gr_F^p\psi_{1/\theta,\zeta'}N=\dim\gr_F^p\psi_{\theta,\zeta}N,
\]
according to the computation of \cite[Th.\,2.2]{TSaito20}. The right-hand side equals $\dim\gr_{F_\mf}^{p-a}\psi_{\theta,\zeta}N$, according to \eqref{eq:mfH}, that is, $\dim\gr_{F_\mf}^{p+\alpha}\psi_{\theta,\zeta}N$, as was to be proved.
\end{proof}

\begin{remark}[Hodge symmetry in $\EMHS^\mf$]\label{rem:HodgesymEmf}
According to Remark \ref{rem:Hodgesymmf} and Proposition \ref{prop:MHSmfirr}, we have, for $N^\rH\in\EMHS^\mf$,
\begin{gather*}
\gr^W\coH^1_\dR(\Afu_\theta,N\otimes E^\theta)\simeq\bigoplus_{\substack{\sfp,\sfq\in\QQ\\\sfp+\sfq\in\ZZ}}[\gr^W\coH^1_\dR(\Afu_\theta,N\otimes E^\theta)]^{\sfp,\sfq},
\\
\dim[\gr^W\coH^1_\dR(\Afu_\theta,N\otimes E^\theta)]^{\sfp,\sfq}=\dim[\gr^W\coH^1_\dR(\Afu_\theta,N\otimes E^\theta)]^{\sfq,\sfp},\quad \sfp,\sfq\in\QQ,\;\sfp+\sfq\in\ZZ.
\end{gather*}
\end{remark}

\subsection{Object of \texorpdfstring{$\EMHS^\mf$}{EMHS} associated with the function \texorpdfstring{$s^mg$}{srg}}\label{subsec:product}

In this section, we assume $U=\Afu_s\times V$ with $\dim V=d$ and $f=s^mg$ for some regular function $g: V\to\Afu_\tau$ and some $m\geq1$. We furthermore assume that $N_U^\rH=\pQQ^\rH_{\Afu_s}\boxtimes M_V^\rH$ for some mixed Hodge module $M_V^\rH$ on~$V$.
We prove a property analogous to that of \cite[Th.\,A.24]{F-S-Y18}.

\begin{prop}\label{prop:sqg}
Under these assumptions, for $?=\emptyset,\rc,\rmid$ and for every $r$, the exponential mixed Hodge structure $\coH^r_?(U,N_U^\rH,s^mg)$ as defined by \eqref{eq:HUNf} is an object of $\EMHS^\mf$.
\end{prop}

\begin{proof}
It is a matter of proving that $(N^\rH_U)^r_*$ and $(N^\rH_U)^r_!$ of \eqref{eq:NHUj}, when restricted to $\Gmtheta$, have no singular point and have finite monodromy. By factoring $f$ as $(s^m\tau)\circ(\id\times g)$ and taking first pushforward by $\id\times g$, we are reduced to the case where $V=\Afu_\tau$, $g=\id$ and we replace~$N^\rH_U$ with a mixed Hodge module \hbox{$N^\rH=\pQQ^\rH_{\Afu_s}\boxtimes M^\rH$} for some mixed Hodge module $M^\rH$ on $\Afu_\tau$. In~such a way, we are reduced to proving that the underlying $\cD_{\Afu_\theta}$-modules $N^r_*$ and $N^r_!$ have no singular point and finite monodromy on~$\Gmtheta$. A~duality argument also shows that it is enough to consider $N^r_*$. Moreover, by factoring through $s\mto t=s^m$, we are reduced to considering the pushforward by $f=t\tau$ of $K_\zeta\boxtimes M$, where $K_\zeta$ is the rank-one $\Clt$-module corresponding to the Kummer sheaf with eigenvalue~$\zeta$ satisfying $\zeta^m=1$, and $M$ is any regular holonomic $\Clt$-module on $\Afu_\tau$. Note that, if~$M$ is supported at $\tau=0$, the pushforward is supported at $\theta=0$ and the assertion is trivially satisfied. We can thus assume that $M=j_+j^+M$.

If $\zeta =1$, so that $K_\zeta=\cO_{\Afu_t}$, the assertion has been obtained in the proof of \cite[Th.\,A.24(1)]{F-S-Y18}. If~$\zeta\neq1$, we are reduced to proving that the multiplicative convolution of $j^+K_\zeta$ with $j^+M$ has no singular point on $\Gmtheta$ and has finite monodromy. We identify $\Gmt\times\Gmtau$ with $\Gmtheta\times\Gmtau$ by the change of variable $\theta=t\tau$, so that the map $f$ is the first projection. Then $j^+K_\zeta\boxtimes j^+M$ is identified with $j^+K_\zeta\boxtimes(j^+K_\zeta^\vee\otimes j^+M)$ and its $r$-th pushforward by the first projection is equal to $(K_\zeta)^{d_r}$, with $d_r=\dim\coH^r_\dR(\Gmtau,j^+K_\zeta^\vee\otimes j^+M)$.\end{proof}

\begin{remark}\label{rem:sqg}
If one replaces $\Afu_s$ with $\Gms$ in Proposition \ref{prop:sqg}, \ie $U=\Gms\times V$, then, by~\hbox{\cite[Th.\,A.24(2)]{F-S-Y18}} for $\zeta=1$ and the same argument for $\zeta\neq1$, one obtains that \hbox{$\coH_?^j(\Gms\times V,N_U^\rH,f)$} belongs to $\EMHS^\mf$.
\end{remark}

According to the decomposition $\EMHS^\mf=\EMHS^\cst\oplus\EMHS^\mf_{\neq1}$, we have a decomposition, for $?=\emptyset,\rc,\rmid$ and for every $r\in\NN$,
\[
\coH^r_?(U,N_U^\rH,s^mg)=\coH^r_?(U,N_U^\rH,s^mg)_\cl\oplus\coH^r_?(U,N_U^\rH,s^mg)_{\neq1}.
\]
We will make more explicit the terms of this decomposition. We start with a preliminary result. We set (there should be no confusion with Notation \ref{nota:ij})
\[
\cA=g^{-1}(0),\quad i:\cA\hto V\qand j:V^*=V\moins \cA\hto V
\]
and we denote similarly the corresponding inclusions from $D:=\Afu_s\times\cA$ to $U = \Afu_s\times V$.

\begin{lemma}\label{lem:HrUMf}
Assume that $M_V$ is supported on $\cA$. Then, for all $r$,
\[
\coH^r_?(U,N_U^\rH,s^mg)=\coH^r_?(U,N_U^\rH,s^mg)_\cl\simeq\coH^r_?(\Afu_t\times V,\pQQ_{\Afu_t}\boxtimes M_V^\rH)\quad\text{in }\MHS\simeq\EMHS^\cst.
\]
In particular,
these objects are independent of $m$.
\end{lemma}

\begin{proof}
The assumption implies that $M_V^\rH=\Hm i_*M^\rH_{\cA}$ for some object $M^\rH_{\cA}$ of $\MHM(\cA)$, according to the equivalence \cite[(4.2.4)]{MSaito87}. Then, denoting by $a$ the structure morphism, we have, for $?=*,!$,
\[
\Hm f_?N_U^\rH\simeq \Hm f_?\,\Hm i_*(\pQQ^\rH_{\Afu_s}\boxtimes M^\rH_{\cA})\simeq \Hm a_{\Afu_t\times \cA,?}(\pQQ^\rH_{\Afu_t}\boxtimes M^\rH_{\cA})\simeq \Hm a_{\Afu_t\times V,?}(\pQQ^\rH_{\Afu_t}\boxtimes M^\rH_V),
\]
in $\MHM(\Afu_\theta)$, and the result follows.
\end{proof}

In order to handle the general case, let us define $V^*_m$ so as to make Cartesian the following diagram:
\[
\xymatrix{
V^*_m\ar[d]_{g_m}\ar[r]^-{[m]_V}&V^*\ar[d]^g\\
\Gmsigma\ar[r]^-{[m]}&\Gmtau
}
\]
where the lower horizontal arrow is defined by $[m](\sigma)=\sigma^m$. Then $[m]_V:V^*_m\to V^*$ is a covering and $V^*_m$ is smooth. Let us set $M_{V^*}^\rH=\Hm j^*M_V^\rH$ and let $M_{V^*_m}^\rH$ denote the pullback $\Hm[m]_V^*M_{V^*}^\rH$. We recover $M_{V^*}^\rH$ as the $\mu_m$-invariant subobject $M_{V^*}^\rH=(\Hm[m]_{V*}M_{V^*_m}^\rH)^{\mu_m}$ in $\MHM(V^*)$.

\begin{prop}\label{prop:HrUMf}
With the above notation and assumptions, for $?=\emptyset,\rc,\rmid$ and every $r$, we have the identifications
\begin{starequation}\label{eq:HrUMf}
\coH^r_?(U,N_U,s^mg)_{\neq1}\simeq\Bigl[\coH^1(\Afu,[m])\otimes \coH^{r-1}_?(V^*_m,M_{V^*_m}^\rH)\Bigr]^{\mu_m}\quad\in\EMHS^\mf_{\neq1},
\end{starequation}%
where $\mu_m$ acts diagonally.
If moreover $M_V$ has no submodule, \resp quotient, supported on~$\cA$, then
\begin{starstarequation}\label{eq:HrUMfsup}
\begin{splitcases}
\coH^r(U,N_U,s^mg)_\cl&\simeq \coH^{r-1}\bigl(U,\cH^1(\Hm i_*\,\Hm i^!N_U^\rH)\bigr),\\
\resp \coH^r_\rc(U,N_U,s^mg)_\cl&\simeq\coH^{r+1}_\rc\bigl(U,\cH^{-1}(\Hm i_*\,\Hm i^*N_U^\rH)\bigr),
\end{splitcases}
\quad\text{in }\MHS\simeq\EMHS^\cst,
\end{starstarequation}%
and so
\begin{equation}
\tag{$\ref{prop:HrUMf}\,{*}{*}{*}$}\label{eq:HrUMfsupindepr}
\coH^r_?(U,N_U,s^mg)_\cl\simeq \coH_?^r(U,N_U,sg),
\end{equation}
which is independent of $m$.
\end{prop}

\begin{proof}
By definition, for $x'\!\in\!V^*_m$, we have $g_m(x')^m\!=\!g([m]_V(x'))$. We consider the isomorphism
\[
\varphi:\Afu_{s'}\times V^*_m\isom \Afu_s\times V^*_m,\quad (s',x')\mto(s'/g_m(x'),x').
\]
Let us set $f_m=f\circ(\id\times[m]_V):\Afu_s\times V^*_m\to\Afu$. We have
\[
f_m(s,x')=s^m\cdot g([m]_V(x'))=(s\cdot g_m(x'))^m\qand f_m\circ\varphi(s',x')=s^{\prime m}.
\]
As a consequence, we find in $\MHM(\Afu_\theta)$, for each $r$:
\begin{align*}
\cH^{r-d-1}\Hm f_*(\pQQ^\rH_{\Afu_s}\boxtimes j_*M^\rH_{V^*})&\simeq\Bigl[\cH^{r-d-1}\bigl(\Hm f_{m*}(\pQQ^\rH_{\Afu_s}\boxtimes M^\rH_{V^*_m})\bigr)\Bigr]^{\mu_m}\\
&\simeq\Bigl[\cH^{r-d-1}\bigl(\Hm (f_m\circ\varphi)_*(\pQQ^\rH_{\Afu_{s'}}\boxtimes M^\rH_{V^*_m})\bigr)\Bigr]^{\mu_m}\\
&\simeq\Bigl[\cH^{r-d-1}\bigl(\Hm (s^{\prime m})_*(\pQQ^\rH_{\Afu_{s'}}\boxtimes M^\rH_{V^*_m})\bigr)\Bigr]^{\mu_m}\\
&\simeq\Bigl[\cH^{r-d-1}\bigl(\Hm (s^{\prime m})_*(\pQQ^\rH_{\Afu_{s'}})\otimes a_{m*}(\Hm M^\rH_{V^*_m})\bigr)\Bigr]^{\mu_m},
\end{align*}
where $a_m$ is the structure morphism of $V^*_m$. Note that, on the first line, the action of $\mu_m$ is due to the ramification $[m]_V$ only while, on the next lines, owing to $\varphi$, the action is due to the ramification $(s',x')\mto(s^{\prime m},[m]_V(x'))$. Since $\Hm (s^{\prime m})_*(\pQQ^\rH_{\Afu_{s'}})=\cH^0\Hm (s^{\prime m})_*(\pQQ^\rH_{\Afu_{s'}})$, we deduce
\[
\cH^{r-d-1}\Hm f_*(\pQQ^\rH_{\Afu_s}\boxtimes j_*M^\rH_{V^*})\simeq\Bigl[\cH^0\Hm (s^{\prime m})_*(\pQQ^\rH_{\Afu_{s'}})\otimes \coH^{r-1}(V^*_m, M^\rH_{V^*_m})\Bigr]^{\mu_m}.
\]
After applying $\Pi_\theta$, since $\Pi_\theta\bigl(\cH^0\Hm (s^{\prime m})_*(\pQQ^\rH_{\Afu_{s'}})\bigr)\!=\!\coH^1(\Afu,[m])\!=\!\coH^1(\Afu,[m])_{\neq1}$, \hbox{we~obtain~that}
\begin{equation}\label{eq:HrUMf*}
\begin{aligned}
\coH^r\bigl(U,(\pQQ^\rH_{\Afu_s}\boxtimes j_*M^\rH_{V^*}),s^mg\bigr)&=\coH^r\bigl(U,(\pQQ^\rH_{\Afu_s}\boxtimes j_*M^\rH_{V^*}),s^mg\bigr)_{\neq1}\\
&\simeq\Bigl[\coH^1(\Afu,[m])\otimes \coH^{r-1}(V^*_m,M_{V^*_m}^\rH)\Bigr]^{\mu_m}.
\end{aligned}
\end{equation}

Arguing similarly with $\Hm j_!M^\rH_{V^*}$, we find
\begin{equation}\label{eq:HrUMf!}
\begin{aligned}
\coH^r\bigl(U,(\pQQ^\rH_{\Afu_s}\boxtimes j_!M^\rH_{V^*}),s^mg\bigr)&=\coH^r\bigl(U,(\pQQ^\rH_{\Afu_s}\boxtimes j_!M^\rH_{V^*}),s^mg\bigr)_{\neq1}\\
&\simeq\Bigl[\coH^1(\Afu,[m])\otimes \coH^{r-1}_\rc(V^*_m,M_{V^*_m}^\rH)\Bigr]^{\mu_m}.
\end{aligned}
\end{equation}

Owing to the exact sequences in $\MHM(V)$:
\begin{gather*}
0\to \cH^0(\Hm i_*\,\Hm i^!M_V^\rH)\to M_V^\rH\to \Hm j_*M_{V^*}^\rH\to \cH^1(\Hm i_*\,\Hm i^!M_V^\rH)\to0,\\
0\to \cH^{-1}(\Hm i_*\,\Hm i^*M_V^\rH)\to \Hm j_!M_{V^*}^\rH\to M_V^\rH\to \cH^0(\Hm i_*\,\Hm i^*M_V^\rH)\to0,
\end{gather*}
\eqref{eq:HrUMf} follows from Lemma \ref{lem:HrUMf} together with \eqref{eq:HrUMf*} \resp \eqref{eq:HrUMf!}.

The assumption of the second part of the proposition implies that the above exact sequences are respectively short exact sequences, and we obtain \eqref{eq:HrUMfsup} according to Lemma \ref{lem:HrUMf} and to the first equality in \eqref{eq:HrUMf*} \resp \eqref{eq:HrUMf!}. The last statement \eqref{eq:HrUMfsupindepr} only expresses the independence of $m$ in the right-hand side of \eqref{eq:HrUMfsup}.
\end{proof}

Assume for example that $M_V^\rH=\pQQ_V^\rH$, so that $N_U=\pQQ_U^\rH$. Since $\bD(\pQQ_U^\rH)\simeq\pQQ_U^\rH(d+1)$ and $\bD \Hm i_*\,\Hm i^*\simeq\Hm i_*\,\Hm i^!\bD$, we obtain the identification in $\MHS$:
\[
\coH^r(U,s^mg)_\cl\simeq\coH_\rc^{2d+2-r}(U,\Hm i_*\,\Hm i^*\QQ^\rH_U)^\vee(-d-1)=(\coH_\rc^{2d+2-r}(D,\QQ))^\vee(-d-1).
\]

Let $\MHM(\QQ(e^{2\pii/m}))$ be the category of $\QQ(e^{2\pii/m})$-mixed Hodge modules, that is, where we extend the coefficients of the perverse sheaf components to the field $\QQ(e^{2\pii/m})$. Let $\EMHS^\mf(\QQ(e^{2\pii/m}))$ be the corresponding category of $\QQ(e^{2\pii/m})$-exponential mixed Hodge structures. In $\EMHS^\mf(\QQ(e^{2\pii/m}))$, we can decompose $\go{m}^\rH$ into rank-one objects (Kummer sheaves):
\[
\go{m}^\rH\simeq \bigoplus_{\zeta\in\mu_m\moins\{1\}}K_\zeta^\rH.
\]
Then we obtain an identification in $\EMHS^\mf(\QQ(e^{2\pii/m}))$:
\begin{equation}\label{eq:nonclassical}
\coH^r_?(U,s^mg)_{\neq1}\simeq\bigoplus_{\zeta\in\mu_m\moins\{1\}}\Bigl(K_{\zeta,\theta}^\rH\otimes\coH^{r-1}_?\bigl(V^*,\Hm g^*K_{\zeta^{-1},\tau}^\rH\bigr)\Bigr).
\end{equation}
\begin{cor}\label{cor:midW}
For any $m\geq1$, we have, in $\EMHS^\mf_{\neq1}\simeq\MHS^\mf_{\neq1}$,
\[
\coH^{d+1}_\rmid(U,s^mg)_{\neq1}=W_{d+1}\coH^{d+1}(U,s^mg)_{\neq1}.
\]
\end{cor}

\begin{proof}
By using Formula \eqref{eq:HrUMf} with $M_V^\rH=\pQQ_V^\rH$ which has pure weight $d$, the assertion reduces to
\[
\coH^d_\rmid(V^*_m,\QQ)=W_d\coH^d(V^*_m,\QQ),
\]
which amounts to the property that $\gr_d^W\coH^d_\rc(V^*_m,\QQ)\to\gr_d^W\coH^d(V^*_m,\QQ)$ is bijective. Let $X$ be a good compactification of $V^*_m$. Since the mixed Hodge structure $\coH^d(V^*_m,\QQ)$ has weights $\geq d$ and $W_d\coH^d(V^*_m,\QQ)=\image[\coH^d(X,\QQ)\to \coH^d(V^*_m,\QQ)]$ (\cf \cite[Prop.\,4.20]{P-S08}), it follows that $\coH^d(X,\QQ)\to \gr_d^W\coH^d(V^*_m,\QQ)$ is an isomorphism. Dually, $ \gr_d^W\coH^d_\rc(V^*_m,\QQ)\to\coH^d(X,\QQ)$ is an isomorphism. Since the canonical map $\coH^d_\rc(V^*_m,\QQ)\to\coH^d(V^*_m,\QQ)$ factors through $\coH^d(X,\QQ)$, the assertion follows.
\end{proof}

\subsection{The case of a tame function \texorpdfstring{$g$}{g}}\label{subsec:tame}
In order to extend Corollary \ref{cor:midW} to $\EMHS^\cst$ (\cf Corollary \ref{cor:midWcl}), we assume that $V$ is the affine space $\Aff{d}$ (so~that $U=\Aff{d+1}$) and that $g$ is a cohomologically tame function on $V$ in the sense of \cite{Bibi96bb} (\eg $g$ is a convenient non-degenerate polynomial in the sense of Kouchnirenko \cite{Kouchnirenko76}).

On the one hand, a consequence of the assumption that $V=\Aff{d}$ is that, if $d\geq2$, when considering the exact sequences in $\EMHS^{\mf}$ (equivalently, in $\MHS^{\mf}$) with middle commutative square (\cf\cite[Ex.\,A.20]{F-S-Y18})
\begin{equation}\label{eq:locsg}
\begin{array}{c}
\xymatrix@C=.7cm{
\coH_\rc^{d+1}(\Aff{d})
&\ar[l] \ar[d]\coH_\rc^{d+1}(\Aff{d+1}, s^mg)
&\ar[l]_-{(*)}\ar[d]\coH_\rc^{d+1}(\Gms\times\Aff{d}, s^mg)
&\ar[l] \coH_\rc^d(\Aff{d})\\
\coH^{d-1}(\Aff{d})(-1)
\ar[r]& \coH^{d+1}(\Aff{d+1}, s^mg)
\ar[r]^-{(*)}& \coH^{d+1}(\Gms\times\Aff{d}, s^mg)
\ar[r]& \coH^d(\Aff{d})(-1).
}
\end{array}
\end{equation}
the maps $(*)$ are isomorphisms. We can thus regard $\coH^{d+1}_\rmid(\Aff{d+1},s^mg)$ as the image of the right vertical arrow.

On the other hand, the tameness assumption implies that the cohomology modules $\cH^rg_\dag\cO_V$ and $\cH^rg_+\cO_V$ are constant $\cD_{\Afu_\tau}$-modules for $r\neq0$, as well as the kernel and cokernel of the natural morphism $\cH^0g_\dag\cO_V\to\cH^0g_+\cO_V$ (\cf\cite{Bibi96bb}).

We consider the pushforward mixed Hodge modules $\cH^0\Hm g_! \pQQ^\rH_{\Aff{d}}$ and $\cH^0\Hm g_* \pQQ^\rH_{\Aff{d}}$. We thus have
\[
\Pi_\tau(\cH^0\Hm g_! \pQQ^\rH_{\Aff{d}})\isom\Pi_\tau(\cH^0\Hm g_* \pQQ^\rH_{\Aff{d}}).
\]
Let us set $\cH^0\Hm g_{!*} \pQQ^\rH_{\Aff{d}}=\image\bigl[\cH^0\Hm g_! \pQQ^\rH_{\Aff{d}}\to\cH^0\Hm g_* \pQQ^\rH_{\Aff{d}}\bigr]$ in $\MHM(\Afu_\tau)$. Since $\cH^0\Hm g_! \pQQ^\rH_{\Aff{d}}$ has weights $\leq d$ and $\cH^0\Hm g_* \pQQ^\rH_{\Aff{d}}$ has weights $\geq d$, we conclude that $\cH^0\Hm g_{!*} \pQQ^\rH_{\Aff{d}}$ is pure of weight $d$. Furthermore, from the above isomorphism, we deduce that
\[
\Pi_\tau(\cH^0\Hm g_{!*} \pQQ^\rH_{\Aff{d}})=\Pi_\tau(\cH^0\Hm g_* \pQQ^\rH_{\Aff{d}}).
\]
We set
\begin{equation}\label{eq:MH}
M^\rH=\image\bigl[\cH^0\Hm g_{!*} \pQQ^\rH_{\Aff{d}}\ra\Pi_\tau(\cH^0\Hm g_* \pQQ^\rH_{\Aff{d}})\bigr].
\end{equation}
This is a pure Hodge module on $\Afu_\tau$ of weight $d$ with no nonzero constant submodule. Furthermore, $\Pi_\tau(M^\rH)=\Pi_\tau(\cH^0\Hm g_* \pQQ^\rH_{\Aff{d}})$.

\begin{prop}\label{prop:HMH}
Under the previous assumptions and if $d\geq2$, there are isomorphisms for $?=\emptyset,\rc,\rmid$,
\begin{align*}
\coH^{d+1}_?(\Aff{d+1},s^mg)_\cl&\simeq\coH^2_?(\Gmt\times\Afu_\tau,\pQQ^\rH_{\Gmt}\boxtimes M^\rH,t\tau),\\
\coH^{d+1}_?(\Aff{d+1},s^mg)_{\neq1}&\simeq\coH^2_?(\Gmt\times\Afu_\tau,\Hm j^*\go{m,t}^\rH\boxtimes M^\rH,t\tau).
\end{align*}
\end{prop}

\begin{lemma}\label{lem:Hcstzero}
Let $\ccM$ be a constant holonomic $\cD_{\Afu_\tau}$-module. Then, for any $r\in\ZZ$ and $m\geq1$,
\[
\coH^r_{\dR,?}(\Gms\times\Afu_\tau,\cO_{\Gms}\boxtimes \ccM,s^m\tau)=0\quad\text{for $?=\emptyset,\rc,\rmid$}.
\]
\end{lemma}

\begin{proof}
Let us first consider the classical component $\coH^r_\dR(\Gmt\times\Afu_\tau,\cO_{\Gmt}\boxtimes \ccM,t\tau)$.
Recall that, for a holonomic $\cD_{\Afu_\tau}$-module $\ccM $, we have
\[
\coH^r_{\dR,?}(\Gmt\times\Afu_\tau,\cO_{\Gmt}\boxtimes \ccM,t\tau)\simeq\coH^{r-1}_{\dR,?}(\Gmt,j^+\FT_\tau(\ccM))\quad\forall r,\ ?=\rc,\emptyset.
\]
Furthermore, if~$\ccM$ is constant, then $\coH^r_{\dR,\rc}$ and $\coH^r_\dR$ vanish for all $r$ since $\FT_\tau \ccM$ is supported at the origin.

For the non-classical component, we can assume $m\geq2$. By taking pushforward by $s\mto t=s^m$, we have an identification
\begin{align*}
\coH^r_{\dR,?}(\Gms\times\Afu_\tau,\cO_{\Gms}\boxtimes \ccM,s^m\tau)_{\neq1}&\simeq\coH^r_{\dR,?}(\Gmt\times\Afu_\tau,j^+\go{m,t}\boxtimes \ccM,t\tau)\\
&\simeq\coH^{r-1}_{\dR,?}(\Gmt,j^+\go{m,t}\otimes j^+\FT_\tau(\ccM)),
\end{align*}
and these cohomologies vanish if $\ccM$ is constant.
\end{proof}

\begin{proof}[Proof of Proposition \ref{prop:HMH}]
According to the isomorphisms $(*)$ in \eqref{eq:locsg}, we may replace $\coH^{d+1}_?(\Aff{d+1},s^mg)$ with $\coH^{d+1}_?(\Gms\times\Aff{d},s^mg)$. Then, applying the pushforward by $g$, the complexes $\Hm g_! \pQQ^\rH_{\Aff{d}}$ and $\Hm g_* \pQQ^\rH_{\Aff{d}}$ reduce to the single mixed Hodge module $M^\rH$ up to complexes in $\catD^\rb(\MHM(\Afu_\tau))$ having constant Hodge modules as cohomologies. Therefore, by Lemma~\ref{lem:Hcstzero}, the hypercohomologies of these complexes reduce to $\coH^2_?(\Gms\times\Afu_\tau,\pQQ^\rH_{\Gms}\boxtimes M^\rH,s^m\tau)$. By~ap\-ply\-ing then the pushforward by $s\mto t=s^m$ we obtain the desired isomorphisms.
\end{proof}

Let us focus on the classical component and set $N^\rH=\pQQ^\rH_{\Afu_t}\boxtimes M^\rH$, so that $\Hm j^*N^\rH=\pQQ^\rH_{\Gmt}\boxtimes M^\rH$. The following lemma is mainly \cite[Cor.\,A.31]{F-S-Y18} plus an argument developed within the proof of \hbox{\cite[Th.\,3.2]{F-S-Y18}}. We give details for the sake of completeness.

\begin{prop}[\cite{F-S-Y18}]\label{prop:weightssimple}
Assume that $M^\rH$ is a pure object of $\MHM(\Afu_\tau)$ of weight $w$. Then
\[
\coH^2_\rmid(\Gmt\times\Afu_\tau,\Hm j^*N^\rH,t\tau)=W_{w+1}\coH^2(\Gmt\times\Afu_\tau,\Hm j^*N^\rH,t\tau).
\]
\end{prop}

\begin{proof}
Since $M^\rH$ is pure, it can be decomposed as the direct sum $M_0^\rH\oplus M_{\cst}^\rH\oplus M_1^\rH$, where $M_0^\rH$ is supported at $\tau=0$, $M_{\cst}^\rH$ is constant, and $M_1^\rH$ is such that $M_1$ has no submodule nor quotient module supported at $\tau=0$ or equal to a constant module.
We can thus consider separately $M_0^\rH$, $M_{\cst}^\rH$ and $M_1^\rH$, the case of $M_{\cst}^\rH$ being solved trivially since both terms in the lemma are zero in that case, according to Lemma \ref{lem:Hcstzero}.

Let us consider a diagram similar to~\eqref{eq:locsg}:
\begin{equation}\label{eq:locstau}
\begin{array}{c}
\xymatrix@C=.4cm{
\coH^2_\rc(\Afu_\tau,M^\rH)&\coH^2_\rc(\Aff{2},N^\rH,t\tau)\ar[l]\ar[d]&\coH^2_\rc(\Gmt\times\Afu_\tau,\Hm j^*N^\rH,t\tau)\ar[l]\ar[d]&\coH^1_\rc(\Afu_\tau,M^\rH)\ar[l]\\
\coH^0(\Afu_\tau,M^\rH)\ar[r]&\coH^2(\Aff{2},N^\rH,t\tau)\ar[r]&\coH^2(\Gmt\times\Afu_\tau,\Hm j^*N^\rH,t\tau)\ar[r]&\coH^1(\Afu_\tau,M^\rH)(-1)
}
\end{array}
\end{equation}
The upper and lower leftmost terms are zero for $M^\rH=M_0^\rH$ or $M_1^\rH$, the upper, \resp lower, middle terms have weight $\leq w+1$, \resp $\geq w+1$ (Proposition \ref{prop:weightsHdUf}), the upper rightmost term has weights $\leq w$ and the lower rightmost term has weights $\geq w+2$. In these cases, the assertion of the lemma is thus equivalent to the similar one on $\Afu_t$:
\begin{equation}\label{eq:H1midA}
\coH^2_\rmid(\Afu_t\times\Afu_\tau,N^\rH,t\tau)=W_{w+1}\coH^2(\Afu_t\times\Afu_\tau,N^\rH,t\tau).
\end{equation}

\subsubsection*{The case $M^\rH=M_0^\rH$}
Denoting by $i$ the inclusion $\Afu_t\times\{0\}\hto\Afu_t\times\Afu_\tau$, there exists a pure Hodge structure $V^\rH$ of weight $w$ such that $N^\rH=\Hm i_*(\pQQ_{\Afu_t}^\rH\otimes V^\rH)$. With the notation of \eqref{eq:NHUj} (with $d=1$ here), we find that $(N^\rH)_*^j=0$ for $j\neq1$, hence $\coH^2(\Aff{2},N^\rH,t\tau)=0$ in that case, so that both terms are zero in \eqref{eq:H1midA}.

\subsubsection*{The case $M^\rH=M_1^\rH$}
It has been treated in the proof of \cite[Th.\,3.2]{F-S-Y18}. Let us just sketch it. It follows from \cite[Cor.\,A.31(1)]{F-S-Y18} that
$\coH^2(\Afu_t\times\Afu_\tau,N^\rH,t\tau)$ is isomorphic to the mixed Hodge structure
\[
\coker[\rN\colon \psi_{\tau,1}M^\rH\ra\psi_{\tau,1}M^\rH(-1)].
\]
Purity of the mixed Hodge structure $\coH^2_\rmid(\Afu_t\times\Afu_\tau,N^\rH,t\tau)$ and inclusion in the subspace $W_{w+1}\coH^2(\Afu_t\times\Afu_\tau,N^\rH,t\tau)$ are clear from the weight estimates of Proposition \ref{prop:weightsHdUf}. We are reduced to proving equality of the dimensions of the de~Rham fibers, that is,
\[
\dim\coH^1_{\dR,\rmid}(\Afu_t,\FT_\tau M)=\dim W_{w+1}\coH^1_\dR(\Afu_t,\FT_\tau M).
\]
This is obtained in \loccit\ by proving the equality
\[
\dim\coH^1_{\dR,\rmid}(\Afu_t,\FT_\tau M)=\dim \rP_0,
\]
where $\rP_0$ is the primitive part of weight $w-1$ in $\gr^W\psi_{\tau,1} M^\rH$.
\end{proof}

From Propositions \ref{prop:HMH} and \ref{prop:weightssimple} we deduce immediately:

\begin{cor}\label{cor:midWcl}
Under the previous assumptions on $g$ and if $d\geq2$, we have in $\EMHS^\cst\simeq\MHS$:
\[
\coH^{d+1}_\rmid(\Aff{d+1},s^mg)_\cl=W_{d+1}\coH^{d+1}(\Aff{d+1},s^mg)_\cl,
\]
that is, according to \eqref{eq:HrUMfsupindepr},
\[
\coH^{d+1}_\rmid(\Aff{d+1},tg)=W_{d+1}\coH^{d+1}(\Aff{d+1},tg).\eqno\qed
\]
\end{cor}

\subsection{Computation of the non-classical part by additive convolution}\label{subsec:additive_convolution}
We keep the setting of Section \ref{subsec:tame}. We revisit \eqref{eq:nonclassical} from the point of view of additive convolution. We~consider $\go{m,t}^\rH$
as a pure Hodge module of weight $1$ on $\Afu_t$. Regarded as a complex Hodge module by extending the scalars, $\go{m,t}^\rH$ is the direct sum of Kummer Hodge modules $K_{\zeta,t}^\rH$ ($\zeta^m=1$, $\zeta\neq1$), which are pure complex Hodge modules of weight $1$ on $\Afu_t$ and Hodge filtration jumping at $0$ only. We continue using Notation \ref{nota:ij} for $\Afu$ with various coordinates.

\subsubsection*{A review of additive and middle additive convolutions}
We refer to \cite[\S1.1]{D-S12} for the following results. Let $M$ be a holonomic module on $\Afu_\tau$. We consider the sum map
\begin{align}
\Afu_\tau\times\Afu_\tau&\To{\Sum}\Afu_u\\
(\tau_1,\tau_2)&\Mto{\hphantom{\Sum}} u=\tau_1+\tau_2,
\end{align}
and define the additive convolutions $\go{m,\tau}\star_!M$ and $\go{m,\tau}\star M$ as the respective pushforwards $\Sum_\dag(\go{m,\tau}\boxtimes M)$ and $\Sum_+(\go{m,\tau}\boxtimes M)$. These objects of $\catD^\rb(\cD_{\Afu})$ have cohomology in degree zero only, and the image of the natural morphism $\go{m,\tau}\star_!M\to \go{m,\tau}\star M$ is denoted by $\go{m,\tau}\star_\rmid M$. We have
\begin{equation}\label{eq:523}
\FT(\go{m,\tau}\star M)\simeq j_+j^+\FT(\go{m,\tau}\star M)\quad\text{and}\quad\FT(\go{m,\tau}\star_\rmid M)\simeq j_{\dag+}j^+\FT(\go{m,\tau}\star M).
\end{equation}
In particular, $\go{m,\tau}\star M$ does not have any nontrivial constant submodule.

\subsubsection*{Weight properties of additive convolutions}
We assume that $M^\rH$ is a pure Hodge module of weight $w$ on $\Afu_\tau$. Then the various convolutions $\go{m,\tau}^\rH\star_?M^\rH$ ($?=!,\emptyset,\rmid$) are mixed Hodge modules.

\begin{lemma}
The mixed Hodge module $\go{m,\tau}^\rH\star M^\rH$ has weights $\geq w+1$, and we have
\[
\go{m,\tau}^\rH\star_\rmid M^\rH=W_{w+1}(\go{m,\tau}^\rH\star M^\rH).
\]
\end{lemma}

\begin{proof}
Since $\go{m,\tau}^\rH\boxtimes M^\rH$ is pure of weight $w+1$, it follows that $\go{m,\tau}^\rH\star_! M^\rH$ has weights $\leq w+1$, $\go{m,\tau}^\rH\star M^\rH$ has weights $\geq w+1$, and so $\go{m,\tau}^\rH\star_\rmid M^\rH$ is pure of weight $w+1$. From \eqref{eq:523} one sees that the quotient $\go{m,\tau}\star M/\go{m,\tau}\star_\rmid M$ is constant, hence so is the quotient $W_{w+1}\go{m,\tau}\star M/\go{m,\tau}\star_\rmid M$, which underlies a pure Hodge module of weight $w+1$, hence is a direct summand of \hbox{$W_{w+1}(\go{m,\tau}\star M)$}. Since $\go{m,\tau}\star M$ does not have any nontrivial constant submodule, the last assertion follows.
\end{proof}

\begin{lemma}\label{lem:Wi!}
The mixed Hodge structure $\cH^0\Hm i^!(\go{m,\tau}^\rH\star M^\rH)$ has weights $\geq w+1$ and
\[
W_{w+1}\bigl[\cH^0\Hm i^!(\go{m,\tau}^\rH\star M^\rH)\bigr]\simeq\rP_0\psi_{\tau,1}(\go{m,\tau}^\rH\star_\rmid M^\rH)(-1).
\]
\end{lemma}

\begin{proof}
Since $\Hm i^!$ increases weights, it follows that
\[
W_{w+1}[\cH^0\Hm i^!(\go{m,\tau}^\rH\star M^\rH)]\simeq W_{w+1}[\cH^0\Hm i^!(\go{m,\tau}^\rH\star_\rmid M^\rH)].
\]
Recall that $\cH^0\Hm i^!$ can be computed as $\coker[\rN_\tau:\psi_{\tau,1}\to\psi_{\tau,1}(-1)]$ on pure Hodge modules (\cf\eg\cite[Ex.\,A.2]{F-S-Y18}) and that the graded component $\gr_{w+1+\ell}^W$ ($\ell\geq0$) is isomorphic to the primitive part $\rP_\ell(-1)$. The conclusion follows.
\end{proof}

\subsubsection*{A comparison result}
Let $M^\rH$ be a mixed Hodge module on $\Afu_\tau$. We aim at computing the irregular Hodge filtration associated with the object of $\EMHS^\mf$ defined as
\begin{equation}\label{eq:526}
\coH^2(\Afu_t\times\Afu_\tau,\go{m,t}^\rH\boxtimes M^\rH,t\tau)=\coH^2(\Gmt\times\Afu_\tau,\Hm j^*\go{m,t}^\rH\boxtimes M^\rH,t\tau).
\end{equation}
Since we are only interested in the Hodge filtration, we work in the context of complex Hodge modules and we consider the objects $\coH^2(\Gmt\times\Afu_\tau,\Hm j^*K_{\zeta,t}^\rH\boxtimes M^\rH,t\tau)$ with $\zeta^m=1$ and $\zeta\neq1$.

\begin{prop}\label{prop:comparisonstar}
Let $M^\rH$ be a mixed Hodge module on $\Afu_\tau$. Then, for $\epsilon\in\{1,\dots,m-1\}$ and $\zeta=\exp(-2\pii\epsilon/m)$, we have
\begin{starequation}\label{eq:comparisonstar}
\dim\gr^{p-\epsilon/m}_{F_\irr}\coH^r_\dR(\Afu_t\times\Afu_\tau,K_{\zeta,t}^\rH\boxtimes M^\rH,t\tau)=\dim\gr^p_F\bigl[\cH^{r-2}\Hm i^!(K_{\zeta^{-1},\tau}^\rH\star M^\rH)\bigr],
\end{starequation}%
and both terms are zero for $r\neq1,2$.
\end{prop}

The above formula is stated in the setting of complex Hodge modules in order to make precise the shift by $-\epsilon/m$. Once this shift is understood, the proof takes place within the category of mixed Hodge modules and yields the formula for the dimension of the Hodge filtration of \eqref{eq:526}.

\begin{proof}
We first explain the identification of the underlying de~Rham cohomology spaces. If~$M$ is supported at the origin, \ie $M=i_+V$, we have
\[
\coH^r_\dR(\Gmt\times\Afu_\tau,j^+K_{\zeta,t}\boxtimes i_+V,t\tau)\simeq\coH^{r-1}_\dR(\Gmt, j^+K_{\zeta,t}\otimes_\CC V),
\]
which is zero for all $r$ since $\zeta\neq1$. We can thus assume that $M=j_+j^+M$, so that
\begin{align*}
\coH^r_\dR(\Gmt\times\Afu_\tau,j^+K_{\zeta,t}\boxtimes M^\rH,t\tau)&\simeq\coH^r_\dR(\Gmt\times\Gmtau,j^+K_{\zeta,t}\boxtimes j^+M,t\tau)\\
&\simeq\phi_\theta\bigl[\cH^{r-2}(t\tau)_+(j^+K_{\zeta,t}\boxtimes j^+M)\bigr]\quad\text{(Lemma \ref{lem:phitheta})}.
\end{align*}
On the one hand, we have seen that $\cH^{r-2}(t\tau)_+(j^+K_{\zeta,t}\boxtimes j^+M)$ is a vector bundle with connection having monodromy equal to $\zeta\id$ and fiber isomorphic to \hbox{$\coH^{r-1}_\dR(\Gm,j^+K_{\zeta^{-1}}\otimes j^+M)$}, where~$\Gm$ is regarded as the torus $\{t\tau=1\}$ (end of proof of Proposition \ref{prop:sqg}). Therefore,
\[
\begin{aligned}
\dim\coH^r_\dR(\Gmt\times\Gmtau,j^+K_{\zeta,t}\boxtimes j^+M,t\tau)
&=\dim\phi_\theta\bigl[\cH^{r-2}(t\tau)_+(j^+K_{\zeta,t}\boxtimes j^+M)\bigr]\\
&=\dim\psi_\theta\bigl[\cH^{r-2}(t\tau)_+(j^+K_{\zeta,t}\boxtimes j^+M)\bigr]\\
&=\rk\cH^{r-2}(t\tau)_+(j^+K_{\zeta,t}\boxtimes j^+M)\\
&=\dim\coH^{r-1}_\dR(\Gm,j^+K_{\zeta^{-1}}\otimes j^+M).
\end{aligned}
\]
It follows that the left-hand side can be nonzero only if $r=1,2$.

On the other hand, let $\delta$ be the diagonal embedding $\Afu_\tau\hto\Afu_\tau\times\Afu_\tau$, let $\iota$ denote the involution $\tau\mto-\tau$ on the first factor and set $\gamma=\iota\circ\delta:\tau\mto(-\tau,\tau)$. Note that $\iota^+K_{\zeta^{-1},\tau}\simeq K_{\zeta^{-1},\tau}$. The base change formula \cite[VI, 8.4]{Borel87b} in the present setting reads
\[
i^+\Sum_+(K_{\zeta^{-1},\tau}\boxtimes M)\simeq a_+\bigl(\gamma^+(K_{\zeta^{-1},\tau}\boxtimes M)\bigr)\simeq a_+\bigl(\delta^+(K_{\zeta^{-1},\tau}\boxtimes M)\bigr),
\]
where $a$ is the structure morphism. By definition, the $\cO_{\Afu}$-module underlying $\delta^+(K_{\zeta^{-1},\tau}\boxtimes M)$ is $K_{\zeta^{-1},\tau}\otimes_{\cO_{\Afu}}^L M$. Since the $\cO_{\Afu}$-module underlying $K_{\zeta^{-1},\tau}$ is $\cO_{\Afu}(*0)$, it is $\cO_{\Afu}$-flat, from which we deduce that
\begin{equation}\label{eq:KotimesM}
\delta^+(K_{\zeta^{-1},\tau}\boxtimes M)\simeq K_{\zeta^{-1},\tau}\otimes_{\cO_{\Afu}}M\simeq j_+(j^+K_{\zeta^{-1},\tau}\otimes j^+M),
\end{equation}
and so
\[
\cH^{r-2}\bigl[i^+\Sum_+(K_{\zeta^{-1},\tau}\boxtimes M)\bigr]\simeq \coH^{r-1}_\dR(\Gm,j^+K_{\zeta^{-1}}\otimes j^+M).
\]
We conclude that
\[
\dim\coH^r_\dR(\Afu_t\times\Afu_\tau,K_{\zeta,t}\boxtimes M,t\tau)=\dim\cH^{r-2}\bigl[i^+(K_{\zeta^{-1}}\star M)\bigr],\quad r=1,2.
\]

Let us now extend the previous computation by taking into account the Hodge filtrations. For the left-hand side of \eqref{eq:comparisonstar}, we have
\begin{equation}\label{eq:comparisonstarLHS}
\begin{aligned}
\dim\gr^{p-\epsilon/m}_{F_\irr}\coH^r_\dR(&\Gmt\times\Gmtau,j^+K_{\zeta,t}\boxtimes j^+M,t\tau)\\
&=\dim\gr^{p-\epsilon/m}_{F_\mf}\psi_{\theta,\neq1}^\mf\bigl[\cH^{r-2}(t\tau)_+(j^+K_{\zeta,t}\boxtimes j^+M)\bigr]\quad\text{(Proposition \ref{prop:MHSmfirr})}\\
&=\dim\gr^p_F\psi_{\theta,\neq1}\bigl[\cH^{r-2}(t\tau)_+(j^+K_{\zeta,t}\boxtimes j^+M)\bigr]\quad\text{(\cf\eqref{eq:mfH})}\\
&=\dim\gr^p_F\psi_\theta\bigl[\cH^{r-2}(t\tau)_+(j^+K_{\zeta,t}\boxtimes j^+M)\bigr]\quad\text{(because $\zeta\neq1$)}\\
&=\rk\gr^p_F\cH^{r-2}(t\tau)_+(j^+K_{\zeta,t}\boxtimes j^+M),
\end{aligned}
\end{equation}
by using that, for a mixed Hodge module $\ccN^\rH$ on $\Afu_\theta$, we have $\dim\gr^p_F\psi_\theta\ccN^\rH=\rk\gr^p_F\ccN^\rH$. We~can compute the rank by restricting by the inclusion $i_1:\{\theta=1\}\hto\Gm$. Let $\delta'$ denote the diagonal embedding $\Gm\hto\Gm\times\Gm$ and let $\inv$ denote the involution $t\mto 1/t$ of $\Gm$. Then $\inv\circ\delta'$ is the inclusion $\gamma':\{t\tau=1\}\hto\Gm\times\Gm$. We note that $\Hm\inv^!K^\rH_{\zeta,t}\simeq K^\rH_{\zeta^{-1},t}$. Therefore, denoting by $a'$ the structure morphism of $\{t\tau=1\}$, we have
\begin{align*}
\Hm i_1^!\bigl[\Hm(t\tau)_*(\Hm j^*K^\rH_{\zeta,t}\boxtimes \Hm j^*M^\rH)\bigr]&\simeq \Hm a'_*\bigl[\Hm\gamma^{\prime!}(\Hm j^*K^\rH_{\zeta,t}\boxtimes \Hm j^*M^\rH)\bigr]\quad\text{(\cf \cite[(4.4.3)]{MSaito87})}\\
&\simeq\Hm a'_*\bigl[\Hm\delta^{\prime!}(\Hm j^*K^\rH_{\zeta^{-1},t}\boxtimes \Hm j^*M^\rH)\bigr].
\end{align*}
Thus, the left-hand side of \eqref{eq:comparisonstar} is given by $\dim\gr^p_F\cH^{r-2}\Hm a'_*\bigl[\Hm\delta^{\prime!}(\Hm j^*K^\rH_{\zeta^{-1},t}\boxtimes \Hm j^*M^\rH)\bigr]$.

For the right-hand side of \eqref{eq:comparisonstar}, taking the notation as above and noting that $\Hm\iota^!K^\rH_{\zeta^{-1},\tau}\simeq K^\rH_{\zeta^{-1},\tau}$, we obtain
\[
\Hm i^!\bigl[\HSum_*(K_{\zeta^{-1},\tau}^\rH\boxtimes M^\rH)\bigr]\simeq\Hm a_*\bigl[\Hm\delta^!(K^\rH_{\zeta^{-1},\tau}\boxtimes M^\rH)\bigr].
\]
It remains to be checked that the natural morphism
\[
\Hm\delta^!(K^\rH_{\zeta^{-1},\tau}\boxtimes M^\rH)\to\Hm j_*\,\Hm\delta^{\prime!}(\Hm j^*K^\rH_{\zeta^{-1},\tau}\boxtimes \Hm j^*M^\rH)
\]
is an isomorphism. It is enough to check that the morphism of the underlying $\cD$-modules is an isomorphism, a property which is provided by \eqref{eq:KotimesM}.
\end{proof}

\begin{cor}\label{cor:grFconvolution}
For $M^\rH$ defined by \eqref{eq:MH}, for $\epsilon\in\{1,\dots,m-1\}$ and $\zeta=\exp(-2\pii\epsilon/m)$, we have, for each $p\in\ZZ$,
\begin{align*}
\dim\gr^{p-\epsilon/m}_F\coH^{d+1}(\Aff{d+1},s^mg)_{\neq1}&=\dim\gr^p_F\bigl[\cH^0\Hm i^!(K_{\zeta^{-1},\tau}^\rH\star M^\rH),\bigr]\\
\dim\gr^{p-\epsilon/m}_F\coH^{d+1}_\rmid(\Aff{d+1},s^mg)_{\neq1}&=\dim\gr^p_F\bigl[\rP_0\psi_{\tau,1}(K_{\zeta^{-1},\tau}^\rH\star_\rmid M^\rH)(-1)\bigr].
\end{align*}
\end{cor}

\begin{proof}
The first equality is obtained by applying Proposition \ref{prop:comparisonstar} to $M^\rH$ together with Proposition \ref{prop:HMH}. For the second one, we apply Lemma \ref{lem:Wi!}.
\end{proof}

\section{The generalized Airy connection and its symmetric powers}\label{sec:settings}

\subsection{The generalized Airy differential equation \texorpdfstring{$\Ai_n$}{Ai}}

Let $n\geq2$ be an integer. The Airy differential operator of order $n$ on $\Afu_z$ is defined~as
\begin{equation}\label{eq:Airy-op}
\partial_z^n-z
\end{equation}
and the corresponding $\CC[z]\langle\partial_z\rangle$-module is denoted by $\Ai_n=\CC[z]\langle\partial_z\rangle/\CC[z]\langle\partial_z\rangle\cdot(\partial_z^n-z)$.
The classical Airy equation corresponds to $n=2$ and the corresponding $\CC[z]\langle\partial_z\rangle$-module is simply denoted by $\Ai$.
Let us set $E^{x^{n+1}/(n+1)}=\CC[x]\langle\partial_x\rangle/(\partial_x+x^n)$ considered as a holonomic $\cD$-module on the affine line~$\Afu_x$.

\begin{lemma}\label{lem:Ain}
The $\CC[z]\langle\partial_z\rangle$-module $\Ai_n$ is the (negative) Fourier transform of $E^{x^{n+1}/(n+1)}$.
In other words,
$\Ai_n$ can be obtained by means of the diagram
\[
\begin{array}{c}
\xymatrix@R=.5cm@C=.2cm{
&\Afu_x\times\Afu_z\ar[dl]^(.45){\!\!f}\ar[dr]_\pi&\\
\Afu&&\Afu_z
}
\end{array}
\qquad f(x,z)=\tfrac1{n+1}\,x^{n+1}-zx,
\]
as
\begin{equation}\tag{\ref{lem:Ain}$\,*$}\label{eq:FTxn}
\begin{split}
\Ai_n&=\cH^0\pi_+E^f\\
&\simeq\coker\Bigl[\CC[x,z]\To{(\partial_x+(x^n-z))\otimes\rd x}\CC[x,z]\otimes\rd x\Bigr]
\end{split}
\end{equation}%
with $\partial_z([x^\ell]\otimes\rd x)=-[x^{\ell+1}]\otimes\rd x$.
\end{lemma}

\begin{proof}
The negative Fourier transformation is induced by the isomorphism $\CC[x]\langle\partial_x\rangle\to\CC[z]\langle\partial_z\rangle$ defined by $x\mto\partial_z$, $\partial_x\mto-z$. That it can be obtained by \eqref{eq:FTxn} is proved in \cite[App.\,2, \S1]{Malgrange91}.
\end{proof}

The $\CC[z]\langle\partial_z\rangle$-module $\Ai_n$ is a free $\CC[z]$-module of rank $n$, with a connection $\nabla$ having singularity at $\infty$ only. Being the Fourier transform of a $\CC[x]\langle\partial_x\rangle$-module of rank one, thus irreducible, $\Ai_n$ is also irreducible. Furthermore, the sheaf of horizontal sections $\Ai_n^\nabla$ is the constant sheaf of rank $n$ on $\Afuan_z$. The differential Galois group of $\Ai_n$ is (\cf \cite[Th.\,4.2.7]{Katz87})
\[\tag*{(6.2)}
\begin{cases}
\mathrm{SL}_n(\CC)&\text{if $n$ is odd},\\
\mathrm{Sp}_n(\CC)&\text{if $n$ is even}.
\end{cases}
\]
It follows that, for $n$ even, $\Ai_n$ is isomorphic to its dual module. More precisely, we have the following behaviour with respect to duality. We denote by $\Ai_n^\vee$ the dual $\CC[z]\langle\partial_z\rangle$-module. It~is isomorphic to $\CC[z]\langle\partial_z\rangle/((-\partial_z)^n-z)$.

Let us denote by $\iota_n:\Afu\to\Afu$ the isomorphism
\begin{equation}\label{eq:iotan}
\iota_n:\begin{cases}\dpl
z\mto\exp(\pii/(n+1))\cdot z&\text{if $n$ is odd},\\\dpl
z\mto z&\text{if $n$ is even}.
\end{cases}
\end{equation}

\begin{lemma}
We have $\Ai_n^\vee\simeq\iota_n^+\Ai_n$.
\end{lemma}

\begin{proof}
If $n$ is even, we have $(-\partial_z)^n-z=\partial_z^n-z$. If $n$ is odd, $\Ai_n^\vee$ is defined by the operator $\partial_z^n+z$, while $\iota_n^+\Ai_n$ is defined by the operator
\[
\exp(-n\pii/(n+1))\partial_z^n-\exp(\pii/(n+1))z
=\exp(-n\pii/(n+1))(\partial_z^n+z).\qedhere
\]
\end{proof}

The formal stationary phase formula for the local Fourier transform
$\cF_-^{\infty,\infty}$ (\cf\cite[\S5.c]{Bibi07a}) applied to $E^{x^{n+1}/(n+1)}$ shows the following, setting $w=1/z$ and $\wh\Ai_n=\CC\lpr w\rpr\otimes\Ai_n$:
\begin{itemize}
\item
The formal connection $\wh\Ai_n$ is isomorphic to the elementary formal connection
\begin{equation}\label{eq:formal_Ai_n}
[n]_+\big( E^{-nt^{n+1}/(n+1)}\otimes L_{(-1)^{n+1}}\big)
\simeq \big([n]_+ E^{-nt^{n+1}/(n+1)}\big)\otimes L_{\sfi_n},
\end{equation}
where
\begin{itemize}
\item
$[n]$ is the finite morphism $t\mto z=t^n$,
\item
$L_{(-1)^{n+1}}$ is the rank-one formal regular connection \hbox{$(\CC\lpr t^{-1}\rpr,\rd+\frac{(n+1)}2\,\rd t/t)$}, corresponding to the rank-one local system on the punctured disc with monodromy~$(-1)^{n+1}$,
\item
we set $\sfi_n=\exp-\pii(n+1)/n$ (so that $\sfi_2=\sfi$) and $L_{\sfi_n}$ is the rank-one formal regular connection $(\CC\lpr w\rpr,\rd+\frac{(n+1)}{2n}\,\rd w/w)$, corresponding to the rank-one local system on the punctured disc with monodromy $\sfi_n$.
\end{itemize}
The isomorphism \eqref{eq:formal_Ai_n} is not canonical, as we can replace $\sfi_n$ with any $n$-th root of $(-1)^{n+1}$ and obtain another isomorphism.
\item
The irregularity at infinity of $\Ai_n$ is equal to $n+1$; its (pure) slope is $(n+1)/n$.
\end{itemize}

\subsection{Symmetric powers of \texorpdfstring{$\Ai_n$}{Ai}}
We now consider $\Sym^k\Ai_n$. The preliminary analysis of this object can be done as in \cite[\S2.1]{F-S-Y18} for the symmetric powers of the Kloosterman connection. Let us set $\ccP(n,k)=\{\bma \in \ZZ_{\geq 0}^n \mid \sum_{i=1}^n a_i =k\}$. Then $\Sym^k\Ai_n$ has rank $\#\ccP(n,k)=\binom{n-1+k}{k}$ and has an irregular singularity at $\infty$ only.

\begin{lemma}[Irreducibility]\label{lem:irred}
The $\Clz$-module $\Sym^k\Ai_n$ is irreducible.
\end{lemma}

\begin{proof}
Since the differential Galois group $G\subset\mathrm{GL}_n(\CC)$ of $\Ai_n$ equals $\mathrm{SL}_n(\CC)$ or $\mathrm{Sp}_n(\CC)$ and any symmetric power of the standard representation of such $G$ is irreducible, one obtains that $\Sym^k\Ai_n$ is irreducible.
\end{proof}

Set $\zeta_n = \exp(\sfrac{2\pii}{n})$ and
\begin{align*}
\ccS(n,k)&= \textstyle\bigl\{
\bma \in \ccP(n,k)\mid \sum_{i=1}^n a_i\zeta_n^i = 0 \bigr\},\quad S_{n,k}= \#\ccS(n,k),\\
\irr(n,k)&= \frac{n+1}{n}\cdot\#(\ccP(n,k)\moins\ccS(n,k))
=\frac{n+1}{n}\biggl[\binom{n-1+k}{k}-S_{n,k}\biggr].
\end{align*}

\begin{lemma}\label{lem:SymwhAi}
We have
\[
\Sym^k\wh\Ai_n\simeq
(\CC^{S_{n,k}}\otimes L_{\sfi_n}^{\otimes k})\oplus\Bigl[\bigoplus_{\bma\in\ccP(n,k)\moins\ccS(n,k)} E^{-n(\sum a_i\zeta_n^i) t^{n+1}/(n+1)}\Bigr]^{\mu_n}\otimes L_{\sfi_n}^{\otimes k}.\]
Furthermore, $\irr_\infty(\Sym^k\wh\Ai_n)=\irr(n,k)$.
\end{lemma}

\begin{proof}
We have
$\Sym^k\wh\Ai_n\simeq\Sym^k([n]_+ E^{-nt^{n+1}/(n+1)})\otimes L_{\sfi_n}^{\otimes k}$, and $\Sym^k([n]_+ E^{-nt^{n+1}/(n+1)})$ is the $\mu_n$\nobreakdash-invariant submodule (where the action is given by $t\mto\exp(2\pii/n)t$) of
\begin{align*}
[n]^+\Sym^k([n]_+ E^{-nt^{n+1}/(n+1)})&\simeq\Sym^k([n]^+[n]_+ E^{-nt^{n+1}/(n+1)})\\
&=\textstyle\Sym^k\bigl(\bigoplus_{i=1}^nE^{-n\zeta_n^i t^{n+1}/(n+1)}\bigr).
\end{align*}
The latter $\Cpttm$-module with connection decomposes as
\[
\bigoplus_{\bma\in\ccP(n,k)} E^{-n(\sum a_i\zeta_n^i) t^{n+1}/(n+1)}.
\]
The result is obtained by taking the $\mu_n$-invariant submodule.
\end{proof}

\begin{example}[The case of $n=2$]\label{ex:SymwhAi}
In the case of $\Ai=\Ai_2$, $\Sym^k\Ai$ has rank $k+1$ and we find:
\[
\Sym^k\wh\Ai\simeq
\begin{cases}
\bigoplus_{j=0}^{(k-1)/2}([2]_+ E^{2(2j-k)t^3/3})\otimes L_{\sfi}^{\otimes k}&\text{if $k$ is odd},\\[5pt]
L_{\sfi}^{\otimes k}\oplus\bigoplus_{j=0}^{k/2-1}\Bigl(([2]_+ E^{2(2j-k)t^3/3})\otimes L_{\sfi}^{\otimes k}\Bigr)&\text{if $k$ is even}.
\end{cases}
\]
In particular, $\irr_\infty(\Sym^k\Ai)=3\flr{(k+1)/2}$ and, if $k$ is odd, $\Sym^k\wh\Ai$ is purely irregular.
\end{example}

\begin{cor}\label{cor:dim_moments_ain}
One has
\begin{align*}
\dim\coH_{\dR,?}^1(\Afu,\Sym^k\Ai_n)
&= \frac{1}{n}\binom{k+n-1}{k} - \frac{n+1}{n}S_{n,k}
\quad\text{for $? = \emptyset, \cp$}, \\
\dim\coH_{\dR,\rmid}^1(\Afu,\Sym^k\Ai_n)
&= \dim\coH^1_\dR(\Afu,\Sym^k\Ai_n) - \begin{cases}
\delta_{n\ZZ}(k)S_{n,k}, &\text{$n$ odd}, \\
\delta_{2n\ZZ}(k)S_{n,k}, &\text{$n$ even}. \end{cases}
\end{align*}
\end{cor}

\begin{proof}
The analogue of the Grothendieck-Ogg-Shafarevich formula and Lemma \ref{lem:SymwhAi} give
\[
\chi_\dR(\Afu,\Sym^k\Ai_n)=\rk\Sym^k\Ai_n-\irr(n,k).
\]
By irreducibility of $\Sym^k\Ai_n$ and affinity of $\Afu$, the left-hand side is $-\dim\coH_\dR^1(\Afu,\Sym^k\Ai_n)$. This yields the first equality for $?=\emptyset$. A duality argument gives the case $?=\cp$. The regular part of $\Sym^k\Ai_n$ at infinity has rank $S(n,k)$ and monodromy $\sfi_n^k\id=\exp(-k(n+1)\pii/n)\id$. We thus have
\[
\dim\coH_{\dR,\rmid}^1(\Afu,\Sym^k\Ai_n)=\begin{cases}
\dim\coH^1(\Afu,\Sym^k\Ai_n)&\text{if }k(n+1)/n\notin2\ZZ,\\
\dim\coH^1(\Afu,\Sym^k\Ai_n)-S(n,k)&\text{if }k(n+1)/n\in2\ZZ,
\end{cases}
\]
which yields the second equality.
\end{proof}

Let $f_k:\Afu_z\times\Aff{k}_x\to\Afu$ be defined by
\[
f_k(z,x_1,\dots,x_k)=\sum_{i=1}^k(\tfrac1{n+1}\, x_i^{n+1}-zx_i)
\]
and let $\pi_k:\Afu_z\times\Aff{k}_x\to\Afu_z$ denote the projection.
The symmetric group $\symgp_k$ acts as automorphisms on $f_k$ and $\pi_k$ by permuting the variables $x_i$.

\begin{prop}\label{prop:Aikfk}
We have
\[
\Ai_n^{\otimes k}\simeq\pi_{k+}E^{f_k}\qand \Sym^k\Ai_n\simeq(\pi_{k+}E^{f_k})^{\symgp_k,\chi},
\]
where the latter term is the isotypic component of $\pi_{k+}E^{f_k}$ under the action of~$\symgp_k$ with respect to the sign character $\chi$ on $\symgp_k$.
\end{prop}

\begin{proof}[Sketch]
The second identification follows from the first one. We write
\[
\Ai_n^{\otimes k}=\bigotimes^k\FT_-(E^{x^{n+1}/(n+1)}),
\]
which we interpret as the Fourier transform $\FT_-$ of the $k$-fold additive convolution product of $E^{x^{n+1}/(n+1)}$ with itself. We note that $\FT_-(E^{x^{n+1}/(n+1)})$ and its dual are $\cO_{\Afu_z}$-flat, which makes the computation of the iterated convolution easy ($E^{x^{n+1}/(n+1)}$ belongs to the category~$\mathsf P$ of Katz \cite{Katz96}). The iterated convolution can also be computed as the pushforward by the sum map $(x_1,\dots,x_k)\mto x_1+\cdots+x_k$ of $E^{{\scriptscriptstyle\sum}_ix_i^{n+1}/(n+1)}$. Then the Fourier transform $\FT_-$ of this convolution can be expressed as $\pi_{k+}E^{f_k}$. In~particular, $\pi_{k+}E^{f_k}=\cH^0\pi_{k+}E^{f_k}$.
\end{proof}

\begin{cor}\label{cor:isofk}
For $?=\emptyset,\rc$, the de~Rham cohomologies $\coH^r_{\dR,?}(\Afu,\Sym^k\Ai_n)$ vanish for $r\neq1$ and we have
\begin{starequation}\label{eq:isofk}
\coH^1_{\dR,?}(\Afu,\Sym^k\Ai_n)\simeq\coH^{k+1}_{\dR,?}(\Aff{k+1},E^{f_k})^{\symgp_k,\chi}.
\end{starequation}%
\end{cor}

\begin{proof}
The first assertion for $\coH^r(\Afu,\Sym^k\Ai_n)$ follows from the irreducibility of $\Sym^k\Ai_n$ and $\coH^r(\Afu,\Sym^k\Ai_n)=0$ for $r\geq2$ as $\Afu$ is affine. Furthermore, Poincaré duality for $\cD_{\Afu}$-modules implies (\cf \eqref{eq:iotan} for $\iota_n$)
\begin{align*}
\coH^r_{\dR,\rc}(\Afu,\Sym^k\Ai_n)&\simeq\coH^{2-r}_\dR(\Afu,\Sym^k\Ai_n^\vee)^\vee\\
&\simeq\coH^{2-r}_\dR(\Afu,\iota_n^+\Sym^k\Ai_n)^\vee\simeq\coH^{2-r}_\dR(\Afu,\Sym^k\Ai_n)^\vee,
\end{align*}
hence the first assertion holds for $\coH^r_{\dR,\rc}(\Afu,\Sym^k\Ai_n)$.

The case $?=\emptyset$ in \eqref{eq:isofk} is a consequence of the proposition and of the isomorphism $a_{\Afu+}\circ\pi_{k+}\simeq a_{\Aff{k+1}+}$, if $a$ denotes the structure morphism. For the case $?=\rc$, we argue by duality, using the isomorphism above for the left-hand side. For the right-hand side, we~note that $\coH^{k+1}_{\dR,\rc}(\Aff{k+1},E^{f_k})$ is dual to $\coH^{k+1}_\dR(\Aff{k+1},E^{-f_k})$, and the latter space is isomorphic to $\coH^{k+1}_\dR(\Aff{k+1},E^{f_k})$ as is seen by using the change of variables (compatible with the action of~$\symgp_k$)
\begin{equation}\label{eq:variablechange}
\begin{cases}\dpl
x_i\mto-x_i,\;z\mto z&\text{if $n$ is even},\\\dpl
x_i\mto\exp(\pii/(n+1))x_i,\;z\mto \exp(-\pii/(n+1))z&\text{if $n$ is odd},
\end{cases}
\end{equation}
which changes $f_k$ to $-f_k$.
\end{proof}

\Subsection{The \texorpdfstring{$\mf$-exponential mixed Hodge structure on the de~Rham cohomology of $\Sym^k\Ai_n$ and $\Sym^k\wt\Ai_n$}{mf}}\label{subsec:FirrSymktildeAi}
In order to apply the results of Section \ref{sec:fmEMHS} to $\Sym^k\Ai_n$,
we first consider a ramification of order $n$ on the base affine line.
Consider the $n$-fold ramified cover $[n]:\Afu_s \to \Afu_z$ given by $s\mto z = s^n$. Let $\wt\pi:\Afu_s\times\Afu_x\to\Afu_s$ denote the first projection and let
\[
\wt{f} = \frac{1}{n+1}\,x^{n+1} - s^nx
\]
equipped with the $\mu_n$ action on the variable $s$.
We define
\begin{equation}\label{eq:wtAin}
\wt\Ai_n=\cH^0\wt\pi_+E^{\wt f}\simeq\wt\pi_+E^{\wt f}\simeq[n]^+\Ai_n,
\end{equation}
where the last isomorphism is compatible with the $\mu_n$ actions.
Then $\wt\Ai_n$ is smooth on $\Afu_s$ and its formal model at infinity is the pullback by $[n]$ of \eqref{eq:formal_Ai_n}, and we have
\[
\Ai_n=([n]_+\wt\Ai_n)^{\mu_n}.
\]
Setting
\[
\wt{f}_k = \sum_{i=1}^k\Bigl(\frac{1}{n+1}\,x_i^{n+1} - s^nx_i\Bigr),
\]
we obtain similarly
\[
\wt\Ai{}_n^{\otimes k}\!\simeq\!\wt\pi_{k+}E^{\wt f_k},\quad\Sym^k\wt\Ai_n\!\simeq\!(\wt\pi_{k+}E^{\wt f_k})^{\symgp_k,\chi},\quad\Sym^k\Ai_n\!\simeq\!(\Sym^k\wt\Ai_n)^{\mu_n}\!\simeq\!(\pi_{k+}E^{\wt f_k})^{\mu_n\times\symgp_k,\chi}.
\]

\begin{cor}\label{cor:isowtfk}
For $?=\emptyset,\rc$, the de~Rham cohomologies $\coH^r_{\dR,?}(\Afu,\Sym^k\wt\Ai_n)$ and $\coH^r_{\dR,?}(\Gm,j^+\Sym^k\wt\Ai_n)$ vanish for $r\neq1$ and we have
\[
\begin{aligned}
\coH^1_{\dR,?}(\Afu,\Sym^k\wt\Ai_n)&\simeq\coH^{k+1}_{\dR,?}(\Aff{k+1},E^{\wt f_k})^{\symgp_k,\chi},\\
\coH^1_{\dR,?}(\Gm,j^+\Sym^k\wt\Ai_n)&\simeq\coH^{k+1}_{\dR,?}(\Gm\times\Aff{k},E^{\wt f_k})^{\symgp_k,\chi}.
\end{aligned}
\]
\end{cor}

\begin{proof}
We first claim that $\Sym^k\wt\Ai_n$ does not have any non trivial constant submodule: other\-wise its pushforward by $[n]$ would give rise, by taking $\mu_n$-invariants, to a non trivial constant submodule of $\Sym^k\Ai_n$, a contradiction. Then we can argue as in the proof of Corollary~\ref{cor:isofk}.
\end{proof}

As a consequence, we obtain:
\[
\coH^1_{\dR,?}(\Afu,\Sym^k\Ai_n)\simeq\coH^{k+1}_{\dR,?}(\Aff{k+1},E^{\wt f_k})^{\mu_n\times\symgp_k,\chi}.
\]

The isomorphisms \eqref{eq:isofk} together with \eqref{eq:HUf} lead us to define, for $?=\emptyset,\rc,\rmid$, the exponential mixed Hodge structures $\coH_?^1(\Afu, \Ai_n^{\otimes k})$ and $\coH_?^1(\Afu, \wt\Ai{}_n^{\otimes k})$ as
\begin{equation}\label{eq:Aifk}
\coH_?^1(\Afu, \Ai_n^{\otimes k}):= \coH_?^{k+1}(\Aff{k+1}, f_k),\quad\coH_?^1(\Afu, \wt\Ai{}_n^{\otimes k}):= \coH_?^{k+1}(\Aff{k+1}, \wt f_k)
\end{equation}
and then the exponential mixed Hodge structures $\coH_?^1(\Afu, \Sym^k\Ai_n)$ and $\coH_?^1(\Afu, \Sym^k\wt\Ai_n)$ as
\begin{equation}\label{eq:Aitensk}
\coH_?^1(\Afu, \Sym^k\Ai_n):= \coH_?^1(\Afu, \Ai_n^{\otimes k})^{\symgp_k,\chi},\quad\coH_?^1(\Afu, \Sym^k\wt\Ai_n):= \coH_?^1(\Afu, \wt\Ai{}_n^{\otimes k})^{\symgp_k,\chi},
\end{equation}
so that
\begin{equation}\label{eq:Aiftildek}
\coH_?^1(\Afu, \Sym^k\Ai_n)\simeq\coH_?^1(\Afu, \Sym^k\wt\Ai_n)^{\mu_n}.
\end{equation}
We can similarly define $\coH_?^1(\Gm,j^+\Sym^k\wt\Ai_n)$ in $\EMHS$ by replacing $\Aff{k+1}$ with $\Gm\times\Aff{k}$. Recall that $\coH^1_{\rmid}$ is defined as the image in $\EMHS$ of $\coH^1_{\rc}\to\coH^1$.

\begin{thm}\label{th:weightsMHM}
The exponential mixed Hodge structures
\[
\coH^1(\Afu,\Sym^k\Ai_n), \quad \coH^1_{\rc}(\Afu,\Sym^k\Ai_n), \quad \text{and} \quad \coH^1_{\rmid}(\Afu,\Sym^k\Ai_n)
\]
are $\mf$-exponential mixed Hodge structures of weights $\geq k+1$, $\leq k+1$, and $k+1$ respectively, and the natural morphisms between them are morphisms of $\mf$-exponential mixed Hodge structures.

Moreover, the induced morphism
\begin{equation}\label{eqn:isom-graded}
\gr^W_{k+1}\coH^1_{\rc}(\Afu,\Sym^k\Ai_n) \to \gr^W_{k+1}\coH^1(\Afu,\Sym^k\Ai_n)
\end{equation}
is an isomorphism, and $\coH^1_{\rmid}(\Afu,\Sym^k\Ai_n)$ is equal to its image.
This $\mf$-exponential pure Hodge structure of weight $k+1$ is equipped with a $(-1)^{k+1}$-symmetric pairing
\begin{equation}\label{eq:pairing-motive}
\coH^1_{\rmid}(\Afu,\Sym^k\Ai_n) \otimes \coH^1_{\rmid}(\Afu,\Sym^k\Ai_n) \to \QQ(-k-1).
\end{equation}
Similar properties hold for $\Sym^k\wt\Ai_n$.
\end{thm}

\begin{proof}
The first statement (weight estimates) follows from Proposition \ref{prop:weightsHdUf}. For the second statement, we first prove the results for $\Sym^k\wt\Ai_n$ by applying Corollaries \ref{cor:midW} and \ref{cor:midWcl} as follows. For $?=\emptyset,\rc,\rmid$, we note that
\[ \coH_?^r(\Aff{k}, \textstyle\frac{1}{n+1}\sum_{i=1}^k x_i^{n+1})
= \begin{cases}
\coH^1(\Afu, \frac{1}{n+1}\,x^{n+1})^{\otimes k} & \text{if $r=k$}, \\
0 & \text{otherwise},
\end{cases}
\]
by the argument used in \cite[Ex.\,A.22]{F-S-Y18}. Furthermore, the exponential mixed Hodge structure $\coH^1(\Afu, \frac{1}{n+1}\,x^{n+1})$ is pure of weight $1$ and of rank $n$ according to Section \ref{subsec:KrH}, and belongs to $\EMHS^\mf$. On the other hand, arguing in a way similar to \cite[Ex.\,A.20]{F-S-Y18}, we have the short exact sequences in $\EMHS$ with a commutative square
\begin{equation}\label{eq:sq_coh_wtf}
\begin{array}{c}
\xymatrix@C=.55cm{
0&\ar[l]\coH_\rc^{k+1}(\Aff{k+1}, \wt{f}_k)
\ar[d]&\ar[l] \coH_\rc^{k+1}(\Gm\times\Aff{k}, \wt{f}_k)
\ar[d]&\ar[l]\coH_\rc^k(\Aff{k}, \textstyle\frac{1}{n+1}\sum_{i=1}^k x_i^{n+1})& \ar[l]0\\
0\ar[r]&\coH^{k+1}(\Aff{k+1}, \wt{f}_k)
\ar[r]& \coH^{k+1}(\Gm\times\Aff{k}, \wt{f}_k)
\ar[r]& \coH^k(\Aff{k}, \textstyle\frac{1}{n+1}\sum_{i=1}^k x_i^{n+1})(-1)
\ar[r]& 0
}
\end{array}
\end{equation}
and thus isomorphisms for $?=\emptyset,\rc$:
\[
\gr_{k+1}^W\coH^{k+1}_?(\Aff{k+1}, \wt{f}_k)\simeq\gr_{k+1}^W\coH^{k+1}_?(\Gm\times\Aff{k}, \wt{f}_k),
\]
since the upper \resp lower rightmost non-trivial term in the diagram is pure of weight $k$ \resp $k+2$. Furthermore,
\[
\coH^{k+1}_\rmid(\Aff{k+1}, \wt{f}_k)\simeq\coH^{k+1}_\rmid(\Gm\times\Aff{k}, \wt{f}_k).
\]
Let us also set
\[
g(y)=\frac{1}{n+1}\, y^{n+1} - y,
\]
and consider the convenient non-degenerate polynomial
\begin{equation}\label{eq:defgk}
g_k(y_1,\dots,y_k) = \sum_{i=1}^k g(y_i).
\end{equation}
Under the change of variables $x_i = sy_i$ on $\Gm\times\Aff{k}$, and arguing as for \eqref{eq:locsg} one has, for $?=\emptyset,\rc,\rmid$,
\begin{equation}\label{eq:coh_wtf_gk}
\coH^{k+1}_?(\Gm\times\Aff{k}, \wt{f}_k)
= \coH^{k+1}_?(\Gm\times\Aff{k}, s^{n+1}g_k),
\end{equation}
and a similar equality between the $\chi$-isotypical components with respect to the action of $\symgp_k$ permuting $y_1,\dots,y_k$. We can thus apply Corollaries \ref{cor:midW} and \ref{cor:midWcl} to obtain the second point for $\Sym^k\wt\Ai_n$.

Lastly, the pairing in the case of $\Sym^k\wt\Ai_n$ is obtained from the pairing
\[
\coH^{k+1}_\rc(\Aff{k+1}, \wt{f}_k)\otimes \coH^{k+1}(\Aff{k+1}, -\wt{f}_k)\to\QQ(-k-1)
\]
together with a change of variables similar to \eqref{eq:variablechange}, and passing to the $\chi$-isotypical component.

In order to deduce the results for $\Sym^k\Ai_n$, one uses \eqref{eq:Aiftildek} for $?=\emptyset,\rc,\rmid$.
\end{proof}

\begin{cor}\label{cor:iotastrict}
We have the following identifications in $\EMHS^\mf$:
\begin{align*}
\coH^1(\Gm,j^+\Sym^k\wt\Ai_n)
&\simeq \coH^{k+1}(\Aff{k+1}, s^{n+1}g_k)^{\symgp_k,\chi},\\
\coH_\rmid^1(\Afu,\Sym^k\wt\Ai_n)
&\simeq\coH_\rmid^{k+1}(\Aff{k+1}, s^{n+1}g_k)^{\symgp_k,\chi}= W_{k+1}\coH^{k+1}(\Aff{k+1}, s^{n+1}g_k)^{\symgp_k,\chi}.
\end{align*}
\end{cor}

\begin{proof}
Notice that the localization sequences analogous to \eqref{eq:locsg} (along $s=0$) yield, for $k\geq2$, the canonical isomorphisms
\begin{align*}
\coH^{k+1}(\Aff{k+1},s^{n+1}g_k)
&\isom \coH^{k+1}(\Gm\times\Aff{k}, s^{n+1}g_k), \\
\coH_\cp^{k+1}(\Gm\times\Aff{k}, s^{n+1}g_k)
&\isom \coH_\cp^{k+1}(\Aff{k+1},s^{n+1}g_k).
\end{align*}
From Corollaries \ref{cor:midW} and \ref{cor:midWcl} we have canonical identifications
\[ \coH_\rmid^{k+1}(\Aff{k+1}, s^{n+1}g_k)
\isom W_{k+1}\coH^{k+1}(\Aff{k+1}, s^{n+1}g_k). \]
The desired identifications are obtained from the lower line of \eqref{eq:sq_coh_wtf} together with \eqref{eq:Aifk}, \eqref{eq:Aitensk} and \eqref{eq:Aiftildek}.
\end{proof}

In order to apply Proposition \ref{prop:HMH} we set, in a way similar to \eqref{eq:MH},
\begin{equation}\label{eq:MkH}
\wt M_k^\rH=\image\bigl[(\cH^0\Hm g_{k,!*} \pQQ^\rH_{\Aff{k}})^{\symgp_k,\chi}\ra\Pi_\tau((\cH^0\Hm g_{k*} \pQQ^\rH_{\Aff{k}})^{\symgp_k,\chi})\bigr]\in\MHM(\Afu_\tau),
\end{equation}
which is pure of weight $k$, and we have $\Pi_\tau(\wt M_k^\rH)=\Pi_\tau((\cH^0\Hm g_{k*} \pQQ^\rH_{\Aff{k}})^{\symgp_k,\chi})$ which has weights $\geq k$. Let us consider the exact sequence in $\MHM(\Afu_\tau)$:
\begin{equation}\label{eq:exccMPiM}
0\to\wt M_k^\rH\to\Pi_\tau(\wt M_k^\rH)\to\wt M_k^{\prime\rH}\to0.
\end{equation}

\begin{lemma}\label{lem:Mkimage}
We have $\wt M_k^\rH=W_k\Pi_\tau(\wt M_k^\rH)$ and $\wt M_k^{\prime\rH}$ is constant of weights $\geq k+1$.
\end{lemma}

\begin{proof}
The cokernel of $\wt M_k^\rH\to W_k\Pi_\tau(\wt M_k^\rH)$ is a constant pure Hodge module of weight~$k$, which is thus a direct summand, hence contained in $\Pi_\tau(\wt M_k^\rH)$. But the $\Cltau$-module underlying the latter mixed Hodge module does not contain any constant nonzero submodule, by definition of $\Pi_\tau$.
\end{proof}

\begin{prop}\label{prop:HodgewtAin}
We have, for $p\in\ZZ$, $\zeta=\exp(-2\pii\epsilon/(n+1))\neq1$ with $\epsilon\in\{1,\dots,n\}$,
\begin{align*}
\dim\gr^p_F\coH^1_\rmid(\Afu_s,\Sym^k\wt\Ai_n)_\cl&=\dim\gr^p_F[\rP_0\psi_{\tau,1}\wt M_k^\rH(-1)],\\
\dim\gr^{p-\epsilon/(n+1)}_{F_\irr}\coH^1_\rmid(\Afu_s,\Sym^k\wt\Ai_n)_{\neq1}&=\dim\gr^p_F\bigl[\rP_0\psi_{\tau,1}(K_{\zeta^{-1}}^\rH\star_\rmid \wt M_k^\rH)(-1)\bigr].
\end{align*}
\end{prop}

\begin{proof}
For the first equality, we identify $\coH^1_\rmid(\Afu_s,\Sym^k\wt\Ai_n)_\cl$ with $W_{k+1}\coH^{k+1}(\Aff{k+1}, tg_k)^{\symgp_k,\chi}$, according to the second line of Corollary \ref{cor:iotastrict} and \eqref{eq:HrUMfsupindepr}. Then we can argue as in the proof of \cite[Th.\,3.2]{F-S-Y18} by applying \cite[Cor.\,A.31(1)]{F-S-Y18} to $\wt M_k^\rH$.

The second equality follows from the second lines of Corollaries \ref{cor:grFconvolution} and \ref{cor:iotastrict}.
\end{proof}

Contrary to the case of the symmetric power of the Kloosterman connection $\Sym^k\wt{\mathrm{Kl}}_2$ considered in \cite{F-S-Y18}, $j^+\Sym^k\wt\Ai_n$ is not the restriction to $\Gm$ of the Fourier transform of a regular holonomic module like $\Pi_\tau(\wt M_k)$ because it is not of slope $1$ at infinity (as can be seen from Lemma \ref{lem:SymwhAi}). The Fourier transform of $\Pi_\tau(\wt M_k)$ will nevertheless prove useful in~\ref{subsec:Gg} when $n=2$. It also takes the form of a symmetric power $j_+\Sym^kG_g$, as~asserted by the following lemma. In other words, the rational slope of $\Sym^k\Ai_n$ at infinity leads to considering both expressions of $\Sym^k\wt\Ai_n$:
\[
[n]^+\Sym^k\Ai_n=\Sym^k\wt\Ai_n\qand j^+\Sym^k\wt\Ai_n\simeq[n+1]^+\Sym^kG_g.
\]

\begin{lemma}\label{lem:MkSymkGg}
Let $G_g=j^+\FT\wt M_1$ be the restriction to $\Gm$ of the Fourier transform of $\wt M_1$ and let $\FT\wt M_k$ be the Fourier transform of $\wt M_k$. We have
\[
\FT\Pi_\tau(\wt M_k)=j_+j^+\FT\wt M_k\simeq j_+\Sym^kG_g.
\]
Furthermore, $\FT\wt M_k\simeq j_{\dag+}\Sym^kG_g$.
\end{lemma}

\begin{proof}
We have $\Pi_\tau(\wt M_k)=\Pi_\tau(\cH^0 g_{k+} \cO_{\Aff{k}})^{\symgp_k,\chi}$, so
\[
j^+\FT\wt M_k=j^+\FT\Pi_\tau(\wt M_k)\simeq j^+\FT(\cH^0 g_{k+} \cO_{\Aff{k}})^{\symgp_k,\chi}\simeq\Sym^kG_g.
\]
For the second point, let us define $\wt\ccM_k$ such that $\FT\wt\ccM_k=j_{\dag+}\Sym^kG_g$. Since $\FT\wt\ccM_k$ is the minimal extension of $G_g$ at the origin, it follows that $\FT\wt \ccM_k\subset\FT \wt M_k$ (because $\FT\wt\ccM_k=\FT\wt\ccM_k\cap \FT \wt M_k$ by minimality) and thus $\wt\ccM_k\subset \wt M_k$. We wish to prove equality. Since $\FT \wt M_k/\FT\wt\ccM_k$ is supported at the origin of $\Afu_t$, the quotient $\wt M_k/\wt\ccM_k$ is constant, and equality would follow from the property that $\wt M_k$ does not have any nonzero constant quotient $\Cltau$-module. In order to prove this property, it is enough to show that such a quotient would underlie a quotient in the category $\MHM(\Afu_\tau)$. Indeed, $\wt M_k^\rH$ being pure, this quotient would be a direct summand, hence contained in $\wt M_k^\rH$, and $\wt M_k$ is known to have no nonzero constant holonomic submodule, being contained in $\Pi_\tau(\wt M_k)$. Dually, it is enough to prove that the maximal constant $\Cltau$-submodule of a pure Hodge module on~$\Afu_\tau$ underlies the maximal constant pure Hodge module. This follows from the theorem of the fixed part. In such a way, we have identified $\wt M_k$ with $\wt\ccM_k$.
\end{proof}

\section{Proof of Theorem \ref{thm:Hodge_numbers}}\label{sec:proof_Hodge_numbers}
We consider from now on the classical Airy equation $\partial_z^2-z$ (the case $n=2$ in Section \ref{sec:settings}), and we set $\Ai=\Ai_2$ and $f(x,z)=\frac13 x^3-zx$. The strategy of the proof of Theorem \ref{thm:Hodge_numbers} is very similar to that developed in \cite{F-S-Y18} for the Kloosterman case. It consists of
\begin{enumerate}
\item
exhibiting a natural basis of $\coH^1_\dR(\Afu,\Sym^k\Ai)$ (Proposition \ref{prop:basisG} and Corollary \ref{cor:basis});
\item
showing that this basis is adapted to the irregular Hodge filtration by lifting it to a suitable compactification of $\Aff{k+1}$ and by computing order of poles of representative differential forms along components of the divisor at infinity;
\item
showing that this basis induces a basis of each graded piece of the irregular Hodge filtration by means of a duality argument.
\end{enumerate}

\begin{notation}\label{nota:k'}
We use the following notation and convention:
\begin{itemize}
\item
For integers $k\geq1$ and $m\in\ZZ$, we set
\[
k':=\flr{\psfrac{k-1}{2}}\text{ (\ie $k=2k'+1$ or $k=2(k'+1)$)},\quad m!\text{ and }m!!=1\text{ if $m\leq0$}.
\]
\item
Since the cases where $k$ is even and where $\kfour$
play a special role, we use the simplified common notation:
\[
[1,k'\rcr=
\begin{cases}
\{1,\dots,k'+1\}&\text{if $k$ is odd},\\
\{1,\dots,k'\}&\text{if $k$ is even},
\end{cases}
\]
\end{itemize}
\end{notation}

\subsection{De Rham cohomology of \texorpdfstring{$\Sym^k\Ai$ and $\Sym^k\wt\Ai$}{SkAi}}
One checks that $\Ai$ is the free $\CC[z]$\nobreakdash-mod\-ule generated by the classes $v_0,v_1$ of $\rd x,-x\rd x$ (\cf \eqref{eq:FTxn}). Moreover, $\partial_zv_0=-x\rd x=v_1$ and $\partial_z^2v_0=x^2\rd x\equiv\nobreak zv_0$, so that we recover the Airy equation $(\partial_z^2-z)v_0=0$.
Corollary \ref{cor:dim_moments_ain} in this case reads
\begin{equation}\label{eq:dimH1}
\begin{aligned}
\dim \coH^1_{\dR,?}(\Afu,\Sym^k\Ai)&=\begin{cases}
k'+1&\text{if $k$ is odd},\\
k'&\text{if $k$ is even},
\end{cases}
\quad \text{for $? = \emptyset, \cp$}, \\
\dim \coH^1_{\dR,\rmid}(\Afu,\Sym^k\Ai)&=\begin{cases}
\dim \coH^1_\dR(\Afu,\Sym^k\Ai) & \text{if $\knotfour$},\\
\dim \coH^1_\dR(\Afu,\Sym^k\Ai)-1 & \text{if $\kfour$}.
\end{cases}
\end{aligned}
\end{equation}

Let $[2]:s\mto s^2=z$ be the ramified covering of $\Afu_z$ of order $2$ given by $s\mto s^2=z$, and set $\wt\Ai=[2]^+\Ai$. Let $L=L_{-1} = (\cO_{\Gm}, \de + \frac{1}{2}\frac{\de z}{z})$. Recall that $j:\Gm\hto\Afu$ denotes the inclusion.
We will set $L\otimes\Sym^k\Ai:=j_+(L\otimes j^+\Sym^k\Ai)$.
Then
\[ \coH_\dR^i(\Afu,L\otimes\Sym^k\Ai)
= \coH_\dR^i(\Gm,L\otimes j^+\Sym^k\Ai), \]
which has dimension $3(k'+1)$ if $i=1$
and is zero otherwise.
We have, on $\Afu$,
\[
[2]_+\Sym^k\wt\Ai\simeq\Sym^k\Ai\oplus(L\otimes\Sym^k\Ai),
\]
and the decomposition
\begin{equation}\label{eq:symktildeAi}
\coH_\dR^1(\Afu, \Sym^k\wt\Ai)
= \coH_\dR^1(\Afu, \Sym^k\Ai)
\oplus \coH_\dR^1(\Afu, L\otimes\Sym^k\Ai)
\end{equation}
into $\mu_2$-character spaces.
More explicitly,
we have the isomorphism
\[ [2]^*L = \Bigl(\cO_{\Gm}, \de + \frac{\de s}{s}\Bigr)
\isom (\cO_{\Gm}, \de) \]
via multiplication by $s$.
Therefore, for an element of $\coH_\dR^1(\Gm,L\otimes j^+\Sym^k\Ai)$
represented by the global section
$\eta \in \Gamma(\Gm,L\otimes j^+\Sym^k\Ai\otimes\Omega^1)$,
we regard it as the class $s[2]^*\eta$ belonging to the image of $\coH_\dR^1(\Afu, \Sym^k\wt\Ai)$ in $\coH_\dR^1(\Gm, j^+\Sym^k\wt\Ai)$.

We use the decomposition \eqref{eq:symktildeAi}
to obtain a basis.
The connection on $\Ai$ is given by the matrix form
\[
\partial_z(v_0,v_1)=(v_0,v_1)\cdot\begin{pmatrix}0&z\\1&0\end{pmatrix}.
\]
To treat the symmetric power moments and the twists together
in this subsection,
for $k\geq1$ and $\rho = 0,1/2$,
we let $V_\rho$ be the connection on $\Gm$
\[ V_\rho = \begin{cases}
j^+\Sym^k\Ai, & \rho=0, \\
L\otimes j^+\Sym^k\Ai, & \rho=1/2.
\end{cases}
\]
Let $\Lambda_\rho = \Gamma(\Gm,V_\rho)$
be the differential module of global sections.
We fix the basis $\bmu:=(u_a)_{0\leq a\leq k}$ of $\Lambda_\rho$
where $u_a=v_0^{k-a}v_1^a$,
in which the differential structure reads
\begin{equation}\label{eq:definingua}
z\partial_z u_a = \rho u_a + (k-a)zu_{a+1}+az^2u_{a-1},
\quad0\leq a\leq k,
\end{equation}
with the convention $u_{-1} = u_{k+1} = 0$.
Then the de Rham cohomology of~$V_\rho$
is given by the cohomology of the two term complex
$z\partial_z:\Lambda_\rho\to \Lambda_\rho$
\[ \coH_\dR^i(\Gm,V_\rho)
= \HH^i(\Lambda_\rho \xrightarrow{z\partial_z} \Lambda_\rho). \]

\pagebreak[2]
\begin{prop}[{\cf\cite[Proof of Prop.\,4.14]{F-S-Y18}}]\label{prop:basisG}\mbox{}
\begin{enumerate}
\item\label{prop:basisG1}
Put $\Lambda_\rho^+ = \bigoplus_{a=0}^k\CC[z]u_a$.
Then the $\CC[z]$-module $\Lambda_\rho^+$ is stable under $z\partial_z$
and the inclusion of
$(\Lambda_\rho^+,z\partial_z)$ into $(\Lambda_\rho,z\partial_z)$
is a quasi-isomorphism.
\item\label{prop:basisG2}
The cokernel $\HH^1(\Lambda_\rho^+,z\partial_z)$
has the basis
\begin{align*}
z^{k'+1}u_0, \cdots, zu_0, u_0, u_1, \cdots, u_k
& \quad \text{if $k$ is odd}, \\
z^{k'}u_0, \cdots, zu_0, u_0, u_1, \cdots, u_k
& \quad \text{if $k$ is even}.
\end{align*}
\end{enumerate}
\end{prop}

\begin{remark}
When analyzing the symmetric products $\Sym^k\Kl_3$ of Kloosterman connections of rank three, Y.\,Qin \cite[\S3.2.1]{Qin22} has developed a slightly different method to obtain an analogous result.
\end{remark}

\begin{proof}\mbox{}\par
\eqref{prop:basisG1}
In fact, for any $r\geq 0$,
the lattice $z^{-r}\Lambda_\rho^+$
is stable under $z\partial_z$
and the induced map
\[ z\partial_z: z^{-r-1}\Lambda_\rho^+/z^{-r}\Lambda_\rho^+
\to z^{-r-1}\Lambda_\rho^+/z^{-r}\Lambda_\rho^+ \]
coincides with the multiplication by $\rho -r-1$
and is an isomorphism.

\eqref{prop:basisG2}
Define a degree map on $\Lambda_\rho^+$ by setting
\[ \deg z = \frac{2}{3},
\quad
\deg u_a = \frac{a}{3}. \]
Then $z\partial_z$ is inhomogeneous of degree one.
On the associated graded module $\ol{\Lambda_\rho^+}$,
the induced map $\ol{z\partial_z}$ is $\CC[z]$-linear
with
\[ \ol{z\partial_z} \bar{u}_a = (k-a)z\bar{u}_{a+1} + az^2\bar{u}_{a-1},
\quad
0\leq a\leq k, \]
where $\bar{u}_a$ denotes the image of $u_a$.

Assume $k$ is even.
It is clear that $\ker\ol{z\partial_z} = \CC[z]\ol u$
where
\[ \ol u = \sum_{i=0}^{k'+1} (-1)^i\binom{k'+1}{i}z^{k'+1-i}\bar{u}_{2i}. \]
Inside the $\CC[z]$-module $\HH^1(\ol{\Lambda_\rho^+},\ol{z\partial_z})$,
the class $\bar{u}_0$ is torsion-free,
$\bar{u}_{2r+1}$ are $z$-torsion for $0\leq r\leq k'$,
and
\[ z\bar{u}_{2r} \equiv \prod_{i=1}^r\frac{1-2i}{k+1-2i}\, z^{2r}\bar{u}_0,
\quad
1\leq r\leq k'+1. \]
Furthermore one has (\cf\cite[Proof of Lem.\,4.17]{F-S-Y18} for a similar argument)
\[ \coker\Big[ z\partial_z: \HH^0(\ol{\Lambda_\rho^+},\ol{z\partial_z})
\ra \HH^1(\ol{\Lambda_\rho^+},\ol{z\partial_z}) \Big]
\isom \HH^1(\Lambda_\rho^+,z\partial_z) \]
and
\begin{align*}
z\partial_z(z^r\ol u)
&= \sum_{i=0}^{k'+1} (-1)^i\binom{k'+1}{i}(r+k'+1-i+\rho)z^{r+k'+1-i}\bar{u}_{2i} \\
&\equiv \begin{cases}
c_0z^{k'+1}\bar{u}_0 + (-1)^{k'+1}\rho \bar{u}_k & \text{if $r=0$}, \\
c_rz^{r+k'+1}\bar{u}_0 & \text{if $r>0$},
\end{cases}
\end{align*}
for certain $c_0,c_r \in \CC$.
Since $\dim \HH^1(\Lambda_\rho^+,z\partial_z) = k+k'+1$,
one must have $c_r \neq 0$ for all $r\geq 0$
and the claim follows.

The case of $k$ odd can be treated similarly and is easier.
Indeed in this case,
$\det \ol{z\partial_z} = (-1)^{k'+1}k!!z^{3k'+3}$
and hence
$\coH^1(\Lambda_\rho^+,z\partial_z) \cong \coH^1(\ol{\Lambda_\rho^+},\ol{z\partial_z}$).
We omit the details.
\end{proof}

\pagebreak[2]
\begin{cor}\mbox{}\label{cor:basis}
\begin{enumerate}
\item
The classes
\[
\omega_i = z^{i-1}u_0\de z,
\quad
\eta_j = u_j\frac{\de z}{z}
\qquad
(i \in [1,k'\rcr,\, 0\leq j\leq k)
\]
(\cf Notation \ref{nota:k'}) form a basis of $\coH_\dR^1(\Gm,j^+\Sym^k\Ai)$.
\item
The classes
\begin{starequation}\label{eq:basisL}
\omega^-_i = z^iu_0\frac{\de z}{z},
\quad
\eta^-_j = u_j\frac{\de z}{z}
\qquad
(i \in [1,k'\rcr,\; 0\leq j\leq k)
\end{starequation}%
form a basis of $\coH_\dR^1(\Gm, L\otimes j^+\Sym^k\Ai)$.
\item\label{cor:basis3}
The family $\Basis_k=\{\omega_i \mid i\in[1,k'\rcr\}$ is a basis of $\coH^1_\dR(\Afu,\Sym^k\Ai)$.
\end{enumerate}
\end{cor}

\begin{proof}
Let $(\Sym^k\Ai)_0 = \bigoplus_{j=0}^k\CC u_j$
be the fiber of $\Sym^k\Ai$ at $z=0$.
We have the exact sequence
\[ 0 \to \coH_\dR^1(\Afu,\Sym^k\Ai)
\to \coH_\dR^1(\Gm,j^+\Sym^k\Ai)
\To{\Res} (\Sym^k\Ai)_0\to 0 \]
where $\Res$ denotes the residue map.
Statement (3) follows by noticing that
$\Res(\omega_i)=0$
and
$\Res(\eta_j) = u_j$.
\end{proof}

\begin{remark}[{\cf\cite[Proof of Prop.\,4.21]{F-S-Y18}}]\label{rem:symkfk}
Under the isomorphism \eqref{eq:isofk} for $?=\emptyset$, the classes $\omega_i$ are mapped to
\[
w_i=[z^{i-1}\rd z\,\rd x_1\cdots\rd x_k].
\]
\end{remark}

\subsection{Middle de~Rham cohomology of \texorpdfstring{$\Sym^k\Ai$}{SkAi}}
By \eqref{eq:dimH1}, $\coH^1_{\dR,\rmid}(\Afu,\Sym^k\Ai)$ is equal to $\coH^1_{\dR}(\Afu,\Sym^k\Ai)$ if $\knotfour$ and is of codimension one in it if $\kfour$.
In order to compute a basis of $\coH^1_{\dR,\rmid}(\Afu,\Sym^k\Ai)$ in the latter case, we define a family of numbers $\gamma_{k,i}$ ($i\in\ZZ$) as follows. Consider the classical Airy functions $\Aif,\Bif$ (\cf \cite[\S10.4]{A-S72} and \cite{V-S04}) which are entire functions on $\CC_z$. Let $s = \frac{2}{3}z^{3/2}$.
As $z\to \infty$, we have the asymptotic expansions
\begin{equation}\label{eq:asympt-AB}
\begin{aligned}
\Aif(z) &\sim \frac{e^{-s}}{2\sqrt{\pi}z^{1/4}}
	\sum_{n=0}^\infty (-1)^n\frac{(n+\frac{1}{2})_{2k}}{54^nn!}\frac{1}{s^n},
&&
|\arg z| < \pi, \\
\Bif(z) &\sim \frac{e^{s}}{\sqrt{\pi}z^{1/4}}
	\sum_{n=0}^\infty \frac{(n+\frac{1}{2})_{2k}}{54^nn!}\frac{1}{s^n},
&&
|\arg z| < \frac{1}{3}\pi.
\end{aligned}
\end{equation}
Set $w = 1/z$. The formal asymptotic expansion of $2\pi(\Aif\Bif)$ at $z=\infty$ is an element in $\sqrt{w}\,(1+w^3\QQ\lcr w^3\rcr)$. It is the unique, up to scaling, formal power series solution in $\sqrt{w}$
to the second symmetric power of the Airy equation
and one checks directly that it has a positive coefficient
in each degree $\sqrt{w}\,w^{3i}$, $i\geq 0$. We define $\gamma_{k,i}$ by the formula
\[
(2\pi\Aif\Bif)^{k/2} \sim \sum_{i\gg -\infty}\gamma_{k,i}w^i.
\]
We then have
\begin{equation}\label{eq:gammaki2}
\gamma_{k,i}
\begin{cases}
= 0&\text{ for $i \not\in k/4+3\ZZ_{\geq 0}$},\\
= 1&\text{ for $i=k/4$},\\
> 0&\text{ for $i \in k/4+3\ZZ_{\geq 0}$}.
\end{cases}
\end{equation}

\begin{prop}[{\cf\cite[Cor.\,3.11]{F-S-Y20b}}]\label{prop:Bkmid}
If $\kfour$, the family
\[
\Basis_{k,\rmid}=\{\omega_i-\gamma_{k,i}\omega_{k/4} \mid i\in[1,k'\rcr,\,i\neq k/4\}
\]
is a basis of $\coH^1_{\dR,\rmid}(\Afu,\Sym^k\Ai)$, and $\omega_{k/4}$ induces a basis of
\[
\coH^1_{\dR}(\Afu,\Sym^k\Ai)/\coH^1_{\dR,\rmid}(\Afu,\Sym^k\Ai).
\]
\end{prop}

\begin{proof}
The second statement follows from the first one and Corollary \ref{cor:basis}\eqref{cor:basis3}. For the first part, let
\[
\iota_{\wh\infty}: \Sym^k\Ai \to (\Sym^k\Ai)_{\wh\infty}
\]
denote the formalization of $\Sym^k\Ai$ at infinity, and $\wh\nabla$ the induced connection. We can represent elements of $\coH^1_{\dR,\rc}(\Afu,\Sym^k\Ai)$ by pairs $(\wh m,\eta)$ as follows (\cf \cite[Cor.\,3.5]{F-S-Y20a}):
\begin{itemize}
\item
$\wh m$ is a formal germ in $(\Sym^k\Ai)_{\wh\infty}$, and
\item
$\eta$ belongs to $\Gamma(\Afu,\Omega^1_{\Afu}\otimes\Sym^k\Ai)$,
\end{itemize}
such that, denoting by $\wh\eta=\iota_{\wh\infty}\eta$ the formal germ of $\eta$ in $\bigl[\Omega^1_{\PP^1,\wh\infty}\otimes(\Sym^k\Ai)_{\wh\infty}\bigr]$, $\wh m$ and $\eta$ are related by $\wh\nabla\wh m=\wh\eta$.

We can regard $\coH^1_{\dR,\rmid}(\Afu,\Sym^k\Ai)$ as the image of the natural morphism
\[
\coH^1_{\dR,\rc}(\Afu,\Sym^k\Ai)\to \coH^1_{\dR}(\Afu,\Sym^k\Ai)
\]
sending a pair $(\wh m,\eta)$ to $\eta$. According to \cite[Rem.\,3.6]{F-S-Y20a}, there exists a basis of $\coH^1_{\dR,\rc}(\Afu,\Sym^k\Ai)$ consisting of
\begin{itemize}
\item
pairs $(\wh m_i,0)_i$ where $(\wh m_i)_i$ is a basis of $\ker\wh\nabla$
in $(\Sym^k\Ai)_{\wh\infty}$,
and
\item
a set of pairs $(\wh m_j,\eta_j)_j$,
of cardinality $\dim\coH^1_{\dR,\rmid}(\Afu,\Sym^k\Ai)$, related as above such that $(\eta_j)_j$ are linearly independent in $\coH^1_{\dR}(\Afu,\Sym^k\Ai)$.
\end{itemize}
Furthermore, such a family $(\eta_j)_j$ is a basis of~$\coH^1_{\dR,\rmid}(\Afu,\Sym^k\Ai)$. The proposition follows from Lemma \ref{lem:solinfty} below.
\end{proof}

\begin{lemma}\label{lem:solinfty}
Let us fix $i\geq1$. If $\kfour$, the subspace $\ker\wh\nabla\subset (\Sym^k\Ai)_{\wh\infty}$ has dimension one and the equation $\wh\nabla\wh m_i=\iota_{\wh\infty}(\omega_i-\gamma_{k,i}\omega_{k/4})$ has a solution (in fact a dimension-one affine space of solutions).
\end{lemma}

\begin{proof}
Assume $4\mid k$.
In this case,
the factor $L_\sfi^{\otimes k}$ in the (unique) formal decomposition of $(\Sym^k\Ai)_{\wh\infty}$ in Example \ref{ex:SymwhAi} is the trivial meromorphic connection $(\CC\lpr w\rpr,\rd)$, so $\ker\wh\nabla$ has dimension one.
For a $1$-form $\omega\in\Gamma(\Afu,\Sym^k\Ai\otimes\Omega^1_{\Afu})$,
we let
$(\omega)_\reg \in \CC\lpr w\rpr\rd w$
be the regular part of~$\iota_{\wh\infty}(\omega)$ in the decomposition of $(\Sym^k\Ai)_{\wh\infty}$ (twisted by $\rd w$). Then $\iota_{\wh\infty}(\omega)\in\image(\wh\nabla)$ if~and only if $\res_w(\omega)_\reg=0$. In order to compute $\res_w(\omega_i)_\reg$, we consider the horizontal basis $\{e_0,e_1\}$ of $\Ai^\nabla$
given by
\[
e_0 = \Aif(z)v_1-\Aif'(z)v_0,
\quad
e_1 = \Bif(z)v_1-\Bif'(z)v_0.
\]
Then $v_0 = \pi(\Bif e_0 - \Aif e_1)$. We note that, for any $i\geq1$,
\begin{align*}
(\omega_i)_\reg=(-w^{-i}u_0\rd w/w)_\reg
&= \pi^k\big([\Bif e_0 - \Aif e_1]^k\big)_\reg \cdot(-w^{-i}\rd w/w)\\
&= \pi^k\binom{k}{k/2}(-\Aif\Bif)^{k/2}(e_0e_1)^{k/2} \cdot(-w^{-i}\rd w/w)
\end{align*}
since $(\Aif\Bif)^{k/2}$ is the only product among $\Aif^i\Bif^{k-i}$
which has no exponential factor in its asymptotic expansion
by \eqref{eq:asympt-AB}. Therefore
\[ \res_w (\omega_i)_\reg
= -(-\pi/2)^{k/2}\binom{k}{k/2}\gamma_{k,i}(e_0e_1)^{k/2} \]
and consequently, for $i\neq k/4$, there exists $\wh{m}_i$ satisfying
$\wh\nabla\wh{m}_i
= \iota_{\wh\infty}(\omega_i-\gamma_{k,i}\omega_{k/4})$.
\end{proof}

\subsection{Proof of Theorem \ref{thm:Hodge_numbers} for \texorpdfstring{$k$}{k} odd}
According to Theorem \ref{th:weightsMHM}, the de~Rham cohomology $\coH^1_{\dR}(\Afu,\Sym^k\Ai)$ underlies a $\mf$-exponential mixed Hodge structure $\coH^1(\Afu,\Sym^k\Ai)$ of weights $\geq k+1$, which is pure of weight $k+1$ and equal to $\coH^1_\rmid(\Afu,\Sym^k\Ai)$ if $\knotfour$. In the case $k$ is odd, we will first relate the irregular Hodge filtration $F_\irr^\cbbullet\coH^1_\dR(\Afu,\Sym^k\Ai)$ of $\coH^1_\dR(\Afu,\Sym^k\Ai)$ with the basis $\Basis_k$.

\begin{lemma}\label{lem:Hodge_numbers_odd}
If $k$ is odd, we have
\[
\omega_i\in F_\irr^{(k+1)-(k+2i)/3}\coH^1_\dR(\Afu,\Sym^k\Ai)\quad(1\leq i\leq k'+1).
\]
\end{lemma}

\begin{proof}
From Corollary \ref{cor:iotastrict} we deduce an inclusion
\begin{equation}\label{eq:iota}
\iota: \coH^1(\Afu,\Sym^k\Ai) \hto \coH^{k+1}(\Aff{k+1}, s^3g_k)
\end{equation}
in $\EMHS^\mf$. This inclusion is strict for the irregular Hodge filtrations on each term, hence it is enough to check that $\iota(\omega_i)\in F_\irr^{(k+1)-(k+2i)/3}\coH^{k+1}(\Aff{k+1}, s^3g_k)$ for $i$ in the given range.

We begin with the compactification $\PP^k$ of $\Aff{k}$
with the boundary $\Aff{k-1}_\infty$.
In fact, $\PP^k$ is the toric variety
associated with the Newton polytope of the Laurent polynomial $g_k$ on $\Gm^k$
and it is non-degenerate
(in the sense of Mochizuki, see \cite[\S A.2]{F-S-Y18})
for the pair $(\Aff{k},g_k)$.
Let $X$ be the blowup of $\PP^1_s\times\PP^k$
along $\{0\}\times\Aff{k-1}_\infty$
(the intersection of the divisors
$\{0\}\times\PP^k$ where~$s^3g_k$ vanishes with order 3
and $\PP^1\times\Aff{k-1}_\infty$
where $s^3g_k$ has triple pole).
Let $P'$ be the proper transform of $\PP^1\times\Aff{k-1}_\infty$
in $X$.

Since $k$ is odd, it is readily checked that $X$ is a non-degenerate compactification
of $(\Aff{k+1}, s^3g_k)$ with boundary divisor $D=X\setminus \Aff{k+1}$.
On $X$,
the pole divisor $P$ of $s^3g_k$ equals
$3(\{\infty\}\times\PP^k + P')$. The cohomology $\coH_\dR^{k+1}(\Aff{k+1},s^3g_k)$ is computed as the hypercohomology on $X$ of the twisted $P$-meromorphic de~Rham complex
\[
(\Omega^\cbbullet_X(\log D)(*P),\rd+\rd(s^3g_k))
\]
and the irregular filtration $F^\cbbullet_\irr\coH_\dR^{k+1}(\Aff{k+1},s^3g_k)$ on the hypercohomology is induced by the filtration $F^\cbbullet_\Yu(\Omega^\cbbullet_X(\log D)(*P),\rd+\rd(s^3g_k))$ by subcomplexes defined by (\cf\cite[(6)]{Yu12})
\[
F^\lambda_\Yu\Omega^p_X(\log D)(*P)=\begin{cases}
0&\text{if }\lambda<0,\\
\Omega^p_X(\log D)(\flr{(p-\lambda) P})&\text{if }\lambda\geq0.
\end{cases}
\]
In particular, the image of the natural map
\[ \Gamma(X,\Omega^{k+1}(\log D)(\flr{\mu P}))
\to \coH_\dR^{k+1}(\Aff{k+1},s^3g_k) \]
lies in $F^{k+1-\mu}_\irr\coH_\dR^{k+1}(\Aff{k+1},s^3g_k)$. Let us consider the image by $\iota$ of the basis $\Basis_k$. We~note that
\[
z^{i-1}\rd z\wedge\rd x_1\wedge\cdots\wedge\rd x_k \in \Gamma(X,\Omega^{k+1}(\log D)(\flr{\tfrac13(k+2i)P})).
\]
It follows that $\iota(\omega_i)$ ($i=1,\dots,k'+1$), which is equal to the image of $w_i$ (\cf Remark \ref{rem:symkfk}) in $\coH_\dR^{k+1}(\Aff{k+1},s^3g_k)$, belongs to
\[
F_\irr^{(k+1)-(k+2i)/3}\coH_\dR^{k+1}(\Aff{k+1},s^3g_k).\qedhere
\]
\end{proof}

\begin{remark}
When $k$ is odd, the functions $f_k$ and $s^3g_k$ are non-degenerate as Laurent polynomials. Therefore to conclude the above Lemma, one can also use the toric method of \cite{A-S97,Yu12} via the inclusions \eqref{eq:isofk} and \eqref{eq:iota} respectively and restriction of the base space from $\Aff{k+1}$ to $\Gm^{k+1}$ as adapted in \cite[\S4.3.1]{F-S-Y18}. Instead of describing explicitly the Newton polytope of~$f_k$ or $s^3g_k$ needed in the toric approach, here we compute the Hodge filtration by exhibiting a non-degenerate compactification which will appear again in the case of even $k$ in Section~\ref{subsec:proofHodge_numbers_even}.
\end{remark}

It will be convenient to define the decreasing filtration $G^\sfp\coH^1_\dR(\Afu,\Sym^k\Ai)$ ($\sfp\in\QQ$) as
\[
G^{(k+1)-(k+2i)/3}\coH^1_\dR(\Afu,\Sym^k\Ai)=\bigl\langle\omega_j\in\Basis_k\mid j\leq i\bigr\rangle
\quad
(1\leq i\leq k'+1),
\]
and thus
\begin{equation}\label{eq:grpG}
\dim\gr^\sfp_G\coH^1_\dR(\Afu,\Sym^k\Ai)=\begin{cases}
1&\text{if $\sfp=(k+1)-(k+2i)/3$ for $1\leq i\leq k'+1$},\\
0&\text{otherwise}.
\end{cases}
\end{equation}
The statement of the lemma above amounts then to
\begin{equation}\label{eq:GpFp}
G^\sfp\coH^1_\dR(\Afu,\Sym^k\Ai)\subset F^\sfp_\irr\coH^1_\dR(\Afu,\Sym^k\Ai)\quad
\text{for all $\sfp$ and $k$ odd}.
\end{equation}

\begin{proof}[Proof of Theorem \ref{thm:Hodge_numbers} when $k$ is odd]
Let us set $d_\sfp=\dim\gr^\sfp_{F_\irr}$ and $\delta_\sfp=\dim\gr^\sfp_G$. Then, by~the above inclusion, $\sum_{\sfq\geq \sfp}\delta_\sfq\leq\sum_{\sfq\geq \sfp}d_\sfq$ for each $\sfp$ with equality for $\sfp$ small and for $\sfp$ large. Furthermore, by~Hodge symmetry (Remark \ref{rem:HodgesymEmf}) we have $d_{k+1-\sfq}=d_\sfq$ and, noticing that for~$k$ odd the symmetry $i\mto j=k'+2-i$ amounts to $k+1-(k+2i)/3\mto k+1-(k+2j)/3$,
we deduce from Formula \eqref{eq:grpG} that $\delta_{k+1-\sfq}=\delta_\sfq$. As a consequence, we also have $\sum_{\sfq\leq \sfp}\delta_\sfq\leq\sum_{\sfq\leq \sfp}d_\sfq$ for all $\sfp$, and it follows that $d_\sfp=\delta_\sfp$ for all $\sfp$. Since $\delta_\sfp=1$ for $\sfp$ as described in the theorem, and zero otherwise, this concludes the proof.
\end{proof}

\subsection{Synopsis of the proof of Theorem \ref{thm:Hodge_numbers} when \texorpdfstring{$k$}{k} is even}\label{subsec:synopsis}
For $k$ even, the formulas in the theorem are equivalent to $\dim\gr_F^\sfp\coH^1_\dR(\Afu,\Sym^k\Ai)=1$ for
\begin{equation}\label{eq:dimgrpF}
\sfp=
\begin{cases}
\begin{cases}
k/2-(2i-1)/3,\\
k/2+1+(2i-1)/3,
\end{cases}
& 1\!\leq\! i\!<\! k/4 \quad\text{if $\ktwofour$},
\\[12pt]
\begin{cases}
k/2-2i/3,\\
k/2+1+2i/3,
\end{cases}
& 1\leq i< k/4 \quad\text{if $\kfour$},
\\
\quad k/2+1 & \text{if $\kfour$},
\end{cases}
\end{equation}
and $\dim\gr_F^\sfp\coH^1_\dR(\Afu,\Sym^k\Ai)=0$ otherwise.

The argument used for $k$ odd cannot be extended to this case since we are not able to prove in a similar way that the inclusion \eqref{eq:GpFp} holds for \emph{every} $\sfp$ (\cf \eqref{eq:GpFpeven} below), and indeed $g_k$ is a~degenerate Laurent polynomial when $k$ is even.
We will thus develop another method, similar to that used in \cite{F-S-Y18} for the moments of the Kloosterman connection, relying on (ramified) Fourier transformation. This will need much more material, which we develop in Sections \ref{subsec:Gg}--\ref{subsec:HodgeAitilde}.
Let us emphasize a new tool with respect to \loccit, namely the inverse stationary phase formula with filtration, as expressed in \cite[(7)]{S-Y18} (\cf the proof of Proposition \ref{prop:rkFMkepsilon}). We explain below the main steps.

\subsubsection*{Step 1}
In this step (realized in Sections \ref{subsec:Gg}--\ref{subsec:HodgeAitilde}), we consider the pullback $\Sym^k\wt\Ai$ and the decomposition \eqref{eq:symktildeAi} and its mid analogue
\[
\coH^1_{\dR,\rmid}(\Afu_s,\Sym^k\wt\Ai)\simeq\coH^1_{\dR,\rmid}(\Afu_z,\Sym^k\Ai)\oplus\coH^1_\dR(\Afu_z,L\otimes\Sym^k\Ai).
\]
We directly compute the irregular Hodge numbers $\dim\gr^\sfp_{F_\irr}\coH^1_\dR(\Afu_s,\Sym^k\wt\Ai)$ for all $\sfp$, relative to the irregular Hodge filtration considered in Section \ref{subsec:FirrSymktildeAi}.

For $\epsilon\in\{0,1,2\}$, we consider the map
\begin{equation}\label{eq:rhoeps}
\begin{aligned}
\rho_\epsilon:\{j\mid0\leq j\leq k,\;k+j+\epsilon\not\equiv0\bmod3\}&\to\NN\\
j&\mto\flr{\sfrac{(k+j+\epsilon)}3}.
\end{aligned}
\end{equation}
Note that, for $p\in\NN$, we have $\#\rho_\epsilon^{-1}(p)\in\{0,1,2\}$.

\begin{prop}[{analogous to {\cite[Prop.\,4.20(2)]{F-S-Y18}}}]\label{prop:dimgrFtildeAi}
If $k$ is even, we have, for $\epsilon=0,1,2$ and $p\in\ZZ$,
\[
\dim\gr^{p-\epsilon/3}_{F_\irr}\coH_{\dR,\rmid}^1(\Afu_s,\Sym^k\wt\Ai)=\begin{cases}
1&\text{if $\epsilon=0$ and $p=k/2,k/2+1$},\\
1&\text{if $\epsilon\neq0$ and $p=k/2+1$},\\
\#\rho_\epsilon^{-1}(p-1)&\text{otherwise}.
\end{cases}
\]
\end{prop}

At this step, we use the techniques explained in \cite[App.]{F-S-Y18} together with Proposition \ref{prop:HodgewtAin}. Since the irregular Hodge filtration is compatible with the decomposition \eqref{eq:symktildeAi}, it also remains to understand the action of $\mu_2$ on the filtered vector space $\coH_{\dR,\rmid}^1(\Afu_s,\Sym^k\wt\Ai)$ and, if $\kfour$, to determine the irregular Hodge filtration on the dimension-one quotient space $\coH_\dR^1/\coH_{\dR,\rmid}^1$ of $\Sym^k\Ai$. We will more precisely relate the irregular Hodge filtration with the filtration $G^\cbbullet$ defined below by the basis constructed in Corollary \ref{cor:basis}.

\subsubsection*{Step 2}
We define the filtration $G^\sfp\coH_\dR^1(\Afu,\Sym^k\wt\Ai)$ as the filtration defined by the basis made of $\omega_i,\omega_i^-,\eta_j^-$ for $i\in[1,k']$ and $0\leq j\leq k$, such that
\[
\gr_G^{k+1-(k+2i)/3} = \CC\omega_i,
\quad
\gr_G^{k+1-(k+2i+1)/3} = \CC\omega_i^-,
\quad
\gr_G^{k+1-(k+j+1)/3} = \CC\eta_j^-.
\]

\begin{lemma}\label{lem:dimgrGp}
If $k$ is even, we have, for $\sfp>k/2+1$,
\[
G^\sfp\coH_\dR^1(\Afu_s,\Sym^k\wt\Ai)\subset\coH_{\dR,\rmid}^1(\Afu_s,\Sym^k\wt\Ai),
\]
and for $\epsilon=0,1,2$ and $k/2+1<p\in\ZZ$,
\[
\dim\gr^{p-\epsilon/3}_{G}\coH_\dR^1(\Afu_s,\Sym^k\wt\Ai)=\dim\gr^{p-\epsilon/3}_{G}\coH_{\dR,\rmid}^1(\Afu_s,\Sym^k\wt\Ai)=\#\rho_\epsilon^{-1}(p-1).
\]
\end{lemma}

\begin{proof}
The first point is clear if $\knotfour$ and follows, if $\kfour$, from Proposition \ref{prop:Bkmid} and from the equivalence
\[
k+1-(k+2i)/3>k/2+1\iff i<k/4.
\]
For the second point, a direct check shows that
\[
\dim\gr^\sfp_G\coH_{\dR,\rmid}^1(\Afu_s,\Sym^k\wt\Ai)=
\begin{cases}
2&\text{for }\sfp\in(k/2+1,2k/3+1/3]\cap\frac13\ZZ,\\
1&\text{for }\sfp=2k/3+2/3,\\
0&\text{for }\sfp>2k/3+2/3.
\end{cases}
\]
Assume for example $k\equiv0\bmod3$. Then one finds, for $p\in\ZZ$,
\begin{align*}
\rho_0^{-1}(p-1)&=
\begin{cases}
2&\text{for }p\in[k/3+1,2k/3]\cap\ZZ,\\
0&\text{otherwise},
\end{cases}\\
\rho_1^{-1}(p-1)&=
\begin{cases}
2&\text{for }p\in[k/3+1,2k/3]\cap\ZZ,\\
1&\text{for }p=2k/3+1,\\
0&\text{otherwise},
\end{cases}\\
\rho_2^{-1}(p-1)&=
\begin{cases}
2&\text{for }p\in[k/3+2,2k/3+1]\cap\ZZ,\\
1&\text{for }p=k/3+1,\\
0&\text{otherwise}.
\end{cases}
\end{align*}
The desired equality is easily checked in this case. The other cases $k\equiv1,2\bmod3$ are checked similarly.
\end{proof}

When $k$ is even, the analogue of Lemma \ref{lem:Hodge_numbers_odd} holds only partially.

\begin{lemma}\label{lem:Hodge_numbers_even}
If $k$ is even, we have
\begin{align*}
\omega_i&\in F_\irr^{(k+1)-(k+2i)/3}\coH^1_\dR(\Afu,\Sym^k\Ai)\quad\text{if }1\leq i\leq\flr{k/4},\\
\omega^-_i &\in F_\irr^{(k+1)-(k+2i+1)/3}\coH_\dR^1(\Gm, L\otimes j^+\Sym^k\Ai)\quad\text{if }1\leq i\leq k'/2,\\
\eta^-_j &\in F_\irr^{(k+1)-(k+j+1)/3}\coH_\dR^1(\Gm, L\otimes j^+\Sym^k\Ai)\quad\text{if }0\leq j\leq k'.
\end{align*}
\end{lemma}

The statement of this lemma, which is proved in Section \ref{subsec:proofHodge_numbers_even}, amounts then to
\begin{equation}\label{eq:GpFpeven}
\begin{aligned}
G^\sfp\coH^1_\dR(\Afu,\Sym^k\Ai)&\subset F^\sfp_\irr\coH^1_\dR(\Afu,\Sym^k\Ai)\\
G^\sfp\coH^1_\dR(\Afu,\Sym^k\wt\Ai)&\subset F^\sfp_\irr\coH^1_\dR(\Afu,\Sym^k\wt\Ai)
\end{aligned}
\quad
\text{for $\sfp\geq k/2+1$}.
\end{equation}

\subsubsection*{Step 3: End of the proof}

We conclude the proof of Theorem \ref{thm:Hodge_numbers} in the case of even $k$ as follows. Firstly, Proposition \ref{prop:dimgrFtildeAi} and Lemmas \ref{lem:dimgrGp} and \ref{lem:Hodge_numbers_even} imply the equalities
\[
F^\sfp_\irr\coH^1_{\dR,\rmid}(\Afu_s,\Sym^k\wt\Ai)= G^\sfp\coH^1_{\dR,\rmid}(\Afu_s,\Sym^k\wt\Ai) \quad(\sfp> k/2+1).
\]
Restricting to $\coH^1_{\dR,\rmid}(\Afu,\Sym^k\Ai)$ yields
\[
F^\sfp_\irr\coH^1_{\dR,\rmid}(\Afu,\Sym^k\Ai)= G^\sfp\coH^1_{\dR,\rmid}(\Afu,\Sym^k\Ai)\quad(\sfp> k/2+1).
\]
Since $\dim\gr_G^\sfp\coH^1_{\dR,\rmid}(\Afu,\Sym^k\Ai)=1$ for $\sfp\in(k/2+1,2k/3]$ and is equal to zero for $\sfp>2k/3$, we find that the same property holds for $\gr_{F_\irr}^\sfp\coH^1_{\dR,\rmid}(\Afu,\Sym^k\Ai)$.

If $\ktwofour$, $\coH^1(\Afu,\Sym^k\Ai)=\coH^1_\rmid(\Afu,\Sym^k\Ai)$ is pure of weight $k+1$, and Hodge symmetry implies that $\dim\gr_{F_\irr}^\sfp\coH^1_\dR(\Afu,\Sym^k\Ai)=1$ for $\sfp\in[k/3+1,k/2)$. The theorem is proved in this case.

If $\kfour$, Hodge symmetry yields the result for $\coH^1_\rmid(\Afu,\Sym^k\Ai)$ and it remains to obtain the last line in \eqref{eq:dimgrpF}. In~that case, the underlying vector space $\coH_\dR$ of the quotient
\[
\coH:=\coH^1(\Afu,\Sym^k\Ai)/W_{k+1}\coH^1(\Afu,\Sym^k\Ai)=\coH^1(\Afu,\Sym^k\Ai)/\coH^1_{\rmid}(\Afu,\Sym^k\Ai)
\]
has dimension one with basis $\omega_{k/4}$
(Proposition \ref{prop:Bkmid}).
Thus this $\mf$-mixed Hodge structure is a pure Hodge structure.
Since its weight is $\geq k+2$, the Hodge structure is of type $(p,p)$ with $p\geq k/2+1$.
The first line of Lemma \ref{lem:Hodge_numbers_even} shows that
$\omega_{k/4}\in F^{k/2+1}\coH_\dR$,
hence the previous inequality is in fact an equality, yielding the last line of \eqref{eq:dimgrpF}.

This completes the proof of Theorem \ref{thm:Hodge_numbers} when $k$ is even.\qed

\subsection{The differential module \texorpdfstring{$G_g$}{Gg} and its symmetric powers}\label{subsec:Gg}
We continue assuming $k$ even, although this is not important in this section. We consider the function $g:\Afu_y\to\Afu_\tau$, $g(y)=\frac{1}3\,y^3-y$ and define $g_k(y_1,\dots,y_k)=\sum_{i=1}^kg(y_i)$. The function $g$ has two simple critical points $y=\pm1$ with critical values $\tau=\pm2/3$,
upon which the vanishing cycle space is of dimension one and local monodromy equals $-\id$.

We take up the notation used in Lemma \ref{lem:MkSymkGg}. Let $\wt M_1=\cH^0g_+\cO_{\Afu_y}$, that we regard as a $\Cltau$-module. The localized Fourier transform~$G_g=j^+\FT\wt M_1$ of $\wt M_1$
on $\Gmt$ is the cokernel of the $\Ctm$-linear morphism
\[
\Ctm[y]\To{\partial_y+(y^2-1)t}\Ctm[y]
\]
and is equipped with the connection $\nabla_{\partial_t}$ induced by $\partial_t+g$. Then $G_g$ is $\Ctm$-free of rank two with basis $\{\wt v_0,\wt v_1\}$, where $\wt v_0$ is the class of $1$ and $\wt v_1$ that of $y$. The matrix of the connection in this basis is given by
\[
\nabla_{\partial_t}(\wt v_0,\wt v_1)=(\wt v_0,\wt v_1)\cdot\biggl[-t^{-1}\begin{pmatrix}1/3&0\\0&2/3\end{pmatrix}-\frac23\begin{pmatrix}0&1\\1&0\end{pmatrix}\biggr].
\]
One can check that it is irreducible (\eg argue as in \cite[Ex.\,8.19]{S-vdP01}). It has a regular singularity at the origin with semi-simple monodromy having eigenvalues $\exp(\pm2\pii/3)$. On the other hand, the formalized module $\wh G_g$ at infinity decomposes as $(E^{2t/3}\oplus E^{-2t/3})\otimes L_{-1}$, where~$L_{-1}$ is the rank-one $\CC\lpr1/t\rpr$-vector space with connection having monodromy equal to~$-\id$.

Since the determinant of $G_g$ is the trivial rank-one bundle with connection, the differential Galois group of $G_g$ is contained in $\mathrm{SL}_2(\CC)$. It is in fact equal to this group (argue as in \loccit). It follows that $\Sym^k G_g$ is also irreducible and its monodromy at the origin is semi-simple with eigenvalues $1$, $\exp2\pii/3$ and $\exp-2\pii/3$. Indeed, the eigenvalues are the numbers (counted with multiplicities) \hbox{$\exp(i+2j)2\pii/3$}, for $i,j\in[0,k]$ and $i+j=k$, that is, the numbers $\exp2(k+j)\pii/3$ with $j\in[0,k]$.

We will implicitly identify $\ZZ/3\ZZ$ with $\{0,1,2\}$. For $\epsilon\in\ZZ/3\ZZ$ and $\zeta=\exp(-2 \pii\epsilon/3)$, we denote by $K_{\zeta,t}$ the Kummer sheaf on $\Afu_t$ with monodromy $\zeta$ around $t=0$ (hence $\zeta^{-1}$ around $t=\infty$) and we also consider similarly $K_{\zeta^{-1},\tau}$. We set $\wt M_{k,0}=\wt M_k$ and $\wt M_{k,\epsilon}=K_{\zeta^{-1},\tau}\star_\rmid\wt M_k$ for $\epsilon\neq0$. We note that, according to \cite[Prop.\,1.18]{D-S12}, we have
\[
\FT\wt M_{k,\epsilon}\simeq j_{\dag+}(j^+K_{\zeta,t}\otimes\Sym^kG_g).
\]

From the above computation of the monodromy of $G_g$ we deduce, for $\epsilon,\epsilon'\in\ZZ/3\ZZ$,
\begin{equation}\label{eq:psit1}
\dim\psi_{t,\exp 2\epsilon'\pii/3}(j^+K_{\zeta,t}\otimes\Sym^k G_g)=
\#\{j\in[0,k]\mid k+j\equiv \epsilon'-\epsilon\bmod3\}.
\end{equation}
We note that, due to semi-simplicity of the monodromy at the origin, the minimal extension $j_{\dag+}(j^+K_{\zeta,t}\otimes\Sym^k G_g)$ at the origin of $j^+K_{\zeta,t}\otimes\Sym^k G_g$ satisfies
\begin{align*}
\phi_{t,1}\bigl[j_{\dag+}(j^+K_{\zeta,t}\otimes\Sym^k G_g)\bigr]&=0,\\
\phi_t\bigl[j_{\dag+}(j^+K_{\zeta,t}\otimes\Sym^k G_g)\bigr]&=\psi_{t,\neq1}\bigl[j_{\dag+}(j^+K_{\zeta,t}\otimes\Sym^k G_g)\bigr]
=\psi_{t,\neq1}(j^+K_{\zeta,t}\otimes\Sym^k G_g).
\end{align*}

At infinity, the formalized module satisfies
\[
(j^+K_{\zeta,t}\otimes \Sym^kG_g)^\wedge\simeq
\bigoplus_{j=0}^k\wh E^{2(2j-k)t/3}\otimes (\wh L_{-1}^{\otimes k}\otimes \wh K_{\zeta^{-1},1/t}).
\]
(Note that on the right-hand side, the monodromy of $\wh K$ is computed with center at infinity, while on the left-hand side it is computed with center at the origin, hence the change from $\zeta$ to $\zeta^{-1}$.) This formal module has a nonzero regular part at $t=\infty$ if and only if $k$ is even, and in that case the regular component has rank one and is isomorphic to $\wh K_{\zeta^{-1},1/t}$.

The monodromy at infinity (in the $\tau$ coordinate) of $\wt M_{k,\epsilon}$ is semi-simple and is isomorphic to the monodromy of \hbox{$\phi_t(j_{\dag+}(j^+K_{\zeta,t}\otimes\Sym^k G_g)$} (according to the inverse stationary phase formula, which is a $\cD$-module variant of \cite[Cor.\,5.20]{Bibi05b}), and the latter space is nothing but the space $\psi_{t,\neq1}(j^+K_{\zeta,t}\otimes\Sym^k G_g)$. In~particular,
\begin{equation}\label{eq:rkwMk}
\begin{split}
\rk \wt M_{k,\epsilon}&=\#\{j\in[0,k]\mid k+j\not\equiv-\epsilon\bmod3\}\\
&=\begin{cases}
\begin{cases}
2(\lfloor k/3\rfloor+1)&\text{if }k\not\equiv0\bmod3,\\
2\lfloor k/3\rfloor&\text{if }k\equiv 0\bmod3,
\end{cases}
&\text{if }\epsilon=0,\\[20pt]
\begin{cases}
2\lfloor k/3\rfloor+1&\text{if }k\not\equiv2\bmod3,\\
2(\lfloor k/3\rfloor+1)&\text{if }k\equiv 2\bmod3,
\end{cases}
&\text{if }\epsilon\neq0.
\end{cases}
\end{split}
\end{equation}
On the other hand, since $\FT(\Pi_\tau(\wt M_{k,\epsilon}))=j_+(j^+K_{\zeta,t}\otimes\Sym^k G_g)$, we have
\begin{align*}
\rk\Pi_\tau(\wt M_{k,\epsilon})&=\dim\psi_{1/\tau}\bigl[\Pi_\tau(\wt M_{k,\epsilon})\bigr]=\dim\phi_t\bigl[j_+(j^+K_{\zeta,t}\otimes\Sym^k G_g)\bigr]\\
&=\dim\psi_t\bigl[j_+(j^+K_{\zeta,t}\otimes\Sym^k G_g)\bigr]=\rk\Sym^k G_g=k+1.
\end{align*}
Let us set $\wt M_{k,\bbullet}=\bigoplus_{\epsilon\in\ZZ/3\ZZ}\wt M_{k,\epsilon}$. We then find
\[
\rk\wt M_{k,\bbullet}=2(k+1),\quad\rk\Pi_\tau(\wt M_{k,\epsilon})=3(k+1).
\]

In the exact sequence
\[
0\to j_{\dag+}(j^+K_{\zeta,t}\otimes\Sym^k G_g)\to j_+(j^+K_{\zeta,t}\otimes\Sym^k G_g)\to C_\epsilon\to0,
\]
the cokernel $C_\epsilon$ is supported at the origin and is isomorphic to $\CC[\partial_t]^{\nu_\epsilon}$, where
\[
\nu_\epsilon=\dim\psi_{t,1}(j^+K_{\zeta,t}\otimes\Sym^k G_g)=\dim\psi_{t,\exp{2\pii\epsilon/3}}\Sym^k G_g.
\]
By taking the inverse Fourier transform of the above exact sequence, we obtain the exact sequence
\begin{equation}\label{eq:exMPiM}
0\to\wt M_{k,\epsilon}\to\Pi_\tau(\wt M_{k,\epsilon})\to\wt M'_{k,\epsilon}\to0,
\end{equation}
where $\wt M'_{k,\epsilon}$ is constant of rank $\nu_\epsilon=\#\{j\in[0,\dots,k\}\mid k+j\equiv\epsilon\bmod3\}$. As a consequence, there is an exact sequence
\[
0\to\wt M_{k,\bbullet}\to\Pi_\tau(\wt M_{k, \bbullet})\to\wt M'_{k, \bbullet}\to0,
\]
where $\wt M'_{k, \bbullet}$ is constant of rank $k+1=\sum_{\epsilon\in\ZZ/3\ZZ}\nu_\epsilon$.

The structure of $\wt M_{k,\epsilon}$ as a $\Cltau$-module is described by the following proposition.

\pagebreak[2]
\begin{prop}\label{prop:Mkepsilon}
For $\epsilon\in\ZZ/3\ZZ$ the following properties hold.
\begin{enumerate}
\item\label{prop:Mkepsilon1}
The regular holonomic $\Cltau$-module $\wt M_{k,\epsilon}$ is irreducible of rank given by \eqref{eq:rkwMk}.
\item\label{prop:Mkepsilon2}
The monodromy of $\wt M_{k,\epsilon}$ at infinity is semi-simple with eigenvalues $\exp(\pm2\pii/3)$.
\item\label{prop:Mkepsilon3}
The singular points $\tau_j$ of $\wt M_{k,\epsilon}$ are the complex numbers $2(2j-k)/3$, $j=0,\dots,k$, and for each such $j$, $\phi_{\tau-\tau_j}\wt M_{k,\epsilon}$ has dimension one and monodromy $(-1)^k\exp(2\pii\epsilon/3)\id$. In particular, since $k$ is even, $0$ is a singular point of $\wt M_{k,\epsilon}$.
\item\label{prop:Mkepsilon4}
Furthermore,
\begin{itemize}
\item\label{prop:Mkepsilon4a}
if $\epsilon\neq0$, $\phi_{\tau,1}\wt M_{k,\epsilon}=0$ and $\psi_{\tau,1}\wt M_{k,\epsilon}$ has dimension equal to $\rk \wt M_{k,\epsilon}-1$ and monodromy equal to $\id$; on the other hand, $\psi_{\tau,\exp(2\pii\epsilon/3)}\wt M_{k,\epsilon}$ has dimension one;
\item\label{prop:Mkepsilon4b}
if $\epsilon=0$, $\phi_{\tau,1}\wt M_k$ has dimension one, $\psi_{\tau,1}\wt M_k$ has dimension equal to $\rk \wt M_k$ and the nilpotent part $\rN$ of the monodromy on $\psi_{\tau,1}\wt M_k$ satisfies $\rN^2=0$.
\end{itemize}
\end{enumerate}
\end{prop}

\begin{proof}
This follows from the stationary phase formula applied to $\wt M_{k,\epsilon}$ (\cf\eg \cite{Bibi07a}) and the formal decomposition of $(j^+K_{\zeta,t}\otimes\Sym^k G_g)^\wedge$ at $t=\infty$ computed above.
\end{proof}

\subsection{Weight properties of \texorpdfstring{$\wt M_{k,\epsilon}$}{Mk}}

Recall the exact sequence \eqref{eq:exccMPiM}, with $\wt M_k^\rH$ pure of weight~$k$.

\begin{prop}[{analogous to {\cite[Prop.\,2.21]{F-S-Y18}}}]\label{prop:rkWMkepsilon}\mbox{}
For $\epsilon\neq0$, the $\Cltau$-module $\wt M_{k,\epsilon}$ underlies a pure Hodge module on~$\Afu_\tau$ of weight $k+1$.
\end{prop}

\begin{proof}
Since $\wt M_{k,\epsilon}=K_{\zeta^{-1},\tau}\star_{\rmid}\wt M_k$ for $\epsilon\neq0$, so that we can write $\wt M_{k,1}\oplus\wt M_{k,2}$ as $\go{m,\tau}\star_{\rmid}\wt M_k$, we can enrich it with the structure of a pure Hodge module $\go{m,\tau}^\rH\star_{\rmid}\wt M_k^\rH$ of weight $k+1$ (the weights behave in an additive way under middle convolution). It follows that~$\wt M_{k,\epsilon}$ ($\epsilon\neq0$) can be enriched with the structure of a pure complex Hodge module $\wt M_{k,\epsilon}^\rH$ of weight $k+1$. Interpreting $\Pi_\tau$ as additive convolution with $j_\dag\cO_{\Gmtau}$, and noting that $\Hm j_!(\pQQ^\rH_{\Gmtau})\in\MHM(\Afu_\tau)$ has weights $0,1$, we conclude that
\[
\Pi_\tau(\wt M_{k,\epsilon}^\rH):=\Hm j_!(\pQQ^\rH_{\Gmtau})\star\wt M_{k,\epsilon}^\rH
\]
has weights $\geq k+1$.
Furthermore, due to the irreducibility of $\wt M_{k,\epsilon}$ (Proposition \ref{prop:Mkepsilon}\eqref{prop:Mkepsilon1}),
we can argue as in Lemma \ref{lem:Mkimage}
to conclude that $\wt M_{k,\epsilon}^\rH=W_{k+1}\Pi_\tau(\wt M_{k,\epsilon}^\rH)$.
\end{proof}

\subsection{Proof of Proposition \ref{prop:dimgrFtildeAi}}\label{subsec:HodgeAitilde}
Recall the map $\rho_\epsilon$ defined by \eqref{eq:rhoeps}.

\begin{prop}\label{prop:rkFMkepsilon}
The Hodge bundles $F^\cbbullet(\wt M_{k,\epsilon})$ satisfy, for $\epsilon\in\ZZ/3\ZZ$,
\[
\rk\gr^p_F\wt M_{k,\epsilon}=\#\rho_\epsilon^{-1}(p)\quad\forall p.
\]
\end{prop}

\begin{proof}
Since the monodromy of $\wt M_k$ at infinity is semi-simple, the limit mixed Hodge structure of $\wt M_k^\rH$ at infinity is pure, and the graded pieces of the limit Hodge filtration have the same dimension as the corresponding graded pieces at a generic point of~$\Afu_\tau$. In order to compute $\rk\gr_F^p\wt M_k$, we are thus led to computing the ranks for this limit Hodge filtration.

For that purpose, we use the comparison with the limit at $t=0$ of the irregular Hodge filtration of $\Sym^k G_g$, as proved in \cite{S-Y14}\footnote{We used this argument in the opposite direction in \cite[(7)]{S-Y18}.} in order to take advantage of the property that the irregular Hodge filtration on $\Sym^k G_g$ is easily computed by means of the Thom-Sebastiani formula from that of~$ G_g$, hence is more directly accessible than the Hodge filtration of $\wt M_k$.

We first compute the jumps of the irregular Hodge filtration $F^\cbbullet_\irr G_{g,1}$ of~$ G_g$ at $t=1$. Since $\Hm g_*\pQQ^\rH_{\Afu}$ is a pure Hodge module of weight $1$, it corresponds, away from the singularities, to a polarizable variation of pure Hodge structure of weight zero on $\Afu_\tau$. The eigenvalues of the monodromy at $\tau=\infty$ of the non constant part of $g_+\cO_{\Afu}$ (the constant part has rank one and is given by the trace) being $\exp(\pm2\pii/3)$, the irregular Hodge filtration $F^\cbbullet_\irr G_{g,1}$ jumps at $1/3$ and~$2/3$ (\cf \cite[(7)]{S-Y18}).

In order to compute the jumps of the irregular Hodge filtration of $j^+K_{\zeta,t}\otimes\Sym^k G_g$, we use that this filtration is obtained by tensor product from the irregular Hodge filtration of $K_{\zeta,t}$ and that of $\Sym^k G_g$ (Thom-Sebastiani).

On the one hand, by the same Thom-Sebastiani argument and the computation above, the jumps of $F_\irr^\cbbullet\Sym^k G_{g,1}$ occur exactly at $\lambda=(i+2j)/3$ where $i,j$ vary from~$0$ to~$k$ and $i+j=k$. We write $\lambda=k/3+j/3$ with $j$ varying from $0$ to $k$.

On the other hand, if $\epsilon\neq0$, we regard $K_{\zeta^{-1},\tau}$ on $\Afu_\tau$ as a complex Hodge module of rank one with Hodge filtration jumping at $p=0$ only. Then, by \cite[(7)]{S-Y18}, the irregular Hodge filtration of its localized Fourier transform $K_{\zeta,t}$ on $\Gmt$ jumps at $\alpha$ only, with $\alpha\in[0,1)$ and $\alpha\equiv\epsilon/3\bmod1$. By using once more the good behaviour of $F_\irr$ under tensor product, we~find similarly that the jumps of $F_\irr^\cbbullet(j^+K_{\zeta,t}\otimes\Sym^k G_g)$ occur exactly at $\lambda=k/3+j/3+\epsilon/3$, where~$j$ varies from $0$ to $k$ and $\epsilon/3\in(0,1)$.

Since $\rk\Sym^k G_g=k+1$, this implies that the jumps are all equal to one. For applying the inverse stationary phase formula, we do not care of the jumps at integers, and the one-dimensional jumps of $F_\irr^\cbbullet(j^+K_{\zeta,t}\otimes\Sym^k G_g)$ ($\epsilon\in\ZZ/3\ZZ$) at $p+1/3$, \resp $p+2/3$ for $p\in\ZZ$ occur respectively for
\begin{align}\label{eq:p+1/3}
p&=\frac13(k+j+\epsilon-1),\quad\text{for }j\in\{0,\dots,k\}\text{ such that }k+j+\epsilon\equiv1\bmod3,\\
\label{eq:p+2/3}
p&=\frac13(k+j+\epsilon-2),\quad\text{for }j\in\{0,\dots,k\}\text{ such that }k+j+\epsilon\equiv2\bmod3.
\end{align}
Using \cite[(7)]{S-Y18} once more in the other direction, we find that the limit Hodge filtration on $\psi_{1/\tau}\wt M_{k,\epsilon}$ satisfies, for $\epsilon\in\ZZ/3\ZZ$,
\begin{align*}
\dim \gr^p_F\psi_{1/\tau,\exp(2\pii/3)}\wt M_{k,\epsilon}&=
\begin{cases}
1&\text{in Case \eqref{eq:p+1/3}},\\
0&\text{otherwise},
\end{cases}\\
\dim \gr^p_F\psi_{1/\tau,\exp(-2\pii/3)}\wt M_{k,\epsilon}&=
\begin{cases}
1&\text{in Case \eqref{eq:p+2/3}},\\
0&\text{otherwise}.
\end{cases}\end{align*}
We thus find the desired formula.
\end{proof}

\begin{prop}[Hodge filtration on $\psi_{\tau,1}\wt M_{k,\epsilon}$]
\label{prop:rkpsiFMkepsilon}
Assume $k=2(k'+1)$ is even.
\begin{enumerate}
\item\label{prop:rkpsiFMkepsilon1}
If $\epsilon=0$, $\dim\gr^p_F\psi_{\tau,1}\wt M_k=\rk\gr^p_F\wt M_k=\#\rho_0^{-1}(p)$.
\item\label{prop:rkpsiFMkepsilon2}
If $\epsilon\neq0$, $\dim\gr^p_F\psi_{\tau,1}\wt M_{k,\epsilon}=\begin{cases}
\#\rho_\epsilon^{-1}(p)&\text{if $p\neq k/2$},\\
\#\rho_\epsilon^{-1}(k/2)-1&\text{if $p= k/2$}.
\end{cases}$
\end{enumerate}
\end{prop}

\begin{proof}\mbox{}
\begin{enumerate}
\item
In this case, $\psi_\tau\wt M_k=\psi_{\tau,1}\wt M_k$ (\cf Proposition \ref{prop:Mkepsilon}\eqref{prop:Mkepsilon4}), so the conclusion follows.
\item
In this case, $\psi_{\tau,1}\wt M_{k,\epsilon}$ has codimension one in $\psi_\tau\wt M_{k,\epsilon}$
(Proposition \ref{prop:Mkepsilon}\eqref{prop:Mkepsilon4})
and we have to determine which value of $p$ as given by Proposition \ref{prop:rkFMkepsilon} to put aside. We know that $\dim\phi_{\tau,1}\wt M_k=1$, and this space underlies a pure Hodge structure of weight~$k$, hence of type $(k/2,k/2)$, so that $\gr^p_F\phi_{\tau,1}\wt M_k=0$ for $p\neq k/2$. Recalling that $\wt M_{k,\epsilon}^\rH=\Hm j_*K^\rH_{\zeta^{-1},\tau}\star_\rmid\wt M_k^\rH$, a proof similar to that of \cite[Th.\,3.1.2(2)]{D-S12} using \cite[Th.\,1.2]{M-S-S16} implies that
\[
\dim\gr^p_F\psi_{\tau,\exp(2\pii\epsilon/3)}\wt M_{k,\epsilon}=\dim\gr^p_F\phi_{\tau,1}\wt M_k,
\]
and the conclusion follows.\qedhere
\end{enumerate}
\end{proof}

\begin{proof}[Proof of Proposition \ref{prop:dimgrFtildeAi}]
We continue assuming $k$ even and we recall that we have set $k=2(k'+1)$. We apply Proposition \ref{prop:HodgewtAin}.

\begin{enumerate}
\item
If $\epsilon=0$, the Lefschetz decomposition of $\gr^W\psi_{\tau,1}\wt M_k^\rH$ reads $\rP_1^\rH\oplus\rP_0^\rH\oplus\rN_\tau(\rP_1^\rH)$, and $\rP_1^\rH$ has weight $k/2=k'+1$ and dimension one. Therefore, the Hodge jump on $\rP_1^\rH$ is at $p=k'+1$ and that on $\rN_\tau(\rP_1^\rH)$ is at $k'$. The remaining jumps on $\rP_0^\rH$ are of size $1$ if $p=k',k'+1$ and of size $\#\rho^{-1}(p-1)$ for the other values of $p$, according to Proposition \ref{prop:rkpsiFMkepsilon}\eqref{prop:rkpsiFMkepsilon1}. We conclude with Proposition \ref{prop:HodgewtAin}:
\[
\dim\gr^p_{F_\irr}\coH_{\dR,\rmid}^1(\Afu_s,\Sym^k\wt\Ai)=
\begin{cases}
1&\text{if }p= k/2,k/2+1,\\
\#\rho_0^{-1}(p-1)&\text{if }p\neq k/2,k/2+1.
\end{cases}
\]
\item
If $\epsilon\neq0$, the monodromy of $\psi_{\tau,1}\wt M_{k,\epsilon}$ is equal to the identity, so that $\psi_{\tau,1}\wt M_{k,\epsilon}=\rP_0\psi_{\tau,1}\wt M_{k,\epsilon}$. Propositions \ref{prop:HodgewtAin} and \ref{prop:rkpsiFMkepsilon}\eqref{prop:rkpsiFMkepsilon2} give
\[
\dim\gr^{p-\epsilon/3}_{F_\irr}\coH_{\dR,\rmid}^1(\Afu_s,\Sym^k\wt\Ai)=\begin{cases}
1&\text{if $p=k/2+1$},\\
\#\rho_\epsilon^{-1}(p-1)&\text{otherwise}.
\end{cases}
\qedhere
\]
\end{enumerate}
\end{proof}

\subsection{Proof of Step 2}\label{subsec:proofHodge_numbers_even}
We prove Lemma \ref{lem:Hodge_numbers_even} (in a way similar to that of \cite[\S4.3.3]{F-S-Y18}).

We start with $\omega_i$. We take up the construction and the notation of Lemma \ref{lem:Hodge_numbers_odd}. The singular locus $S\subset \Aff{k}$ of $(g_k)\subset\Aff{k}$
consists of the $\binom{k}{k/2}$ points
$y_i = \epsilon_i$
with
$\epsilon_i = \pm 1$
and
$\sum_{i=1}^k\epsilon_i = 0$
(where $g_k=0$ has ordinary quadratic singularity).
Then $X$ is non-degenerate away from $\{\infty\}\times S$.
Let $\varpi:\wt{X}\to X$ be the 3-step blowup of $X$ constructed as follows.
Let~$X_1$ be the blowup of $X$ along $\{\infty\}\times S$.
Then on each component of the exceptional divisor~$E_1$,
the rational function $s^3g_k$
possesses an ordinary quadratic point.
Let~$X_2$ be the blowup of~$X_1$ along these points
with exceptional divisor $E_2$;
let $\wt{X}$ be the blowup of $X_2$,
with exceptional divisor $\wt{E}_3$,
along the intersection of $E_2$ and the proper transform of~$E_1$.
Then by a direct checking,
$\wt{X}$ is a non-degenerate compactification of $(\Aff{k+1},s^3g_k)$.
In $\wt{X}$,
let $\wt{D} = \wt{X}\setminus\Aff{k+1}$ be the boundary,
$\wt{P}$ the pole divisor of $s^3g_k$
and $\wt{E}_j$ the proper transform of~$E_j$
for $j=1,2$.
We~moreover have
\[ \ord_{\wt{E}_1} s^3g_k = -1,
\quad
\ord_{\wt{E}_2} s^3g_k = 1,
\quad
\ord_{\wt{E}_3} s^3g_k = 0. \]
Given a form
\[
\omega = z^{r-1}x^n\rd z\wedge\rd x_1\wedge\cdots\wedge\rd x_k,
\quad
r\geq 1,\
n = (n_i) \in\{0,1\}^k,\ \nu = \textstyle\sum_{i=1}^k n_i,
\]
one checks that the pullback satisfies
\begin{equation}\label{eq:varpi_omega}
\begin{split}
\varpi^*\omega &\in
\Gamma\big(\wt{X},\Omega^{k+1}(\log\wt{D})
	(\lfloor m_{r,\nu}\wt{P}\rfloor - (k-2r-\nu)\wt{E}_2 - (k-4r-2\nu)\wt{E}_3)\big),
\\
m_{r,\nu} &= \frac{k+2r+\nu}{3},
\end{split}
\end{equation}
provided that $k\geq 4r+2\nu$ (in order to have allowable pole order along $\wt{E}_1$).
In particular, if~$\mu\geq\frac13(k+2i)$ and $i\leq k/4$ (and since $k-2i\geq2$),
\[
\varpi^*(z^{i-1}\rd z\wedge\rd x_1\wedge\cdots\wedge\rd x_k) \in
\Gamma\big(\wt{X},\Omega^{k+1}(\log\wt{D})
	(\lfloor\mu\wt{P}\rfloor)\big),
\]
and thus $\iota(\omega_i)$
defines an element in
$F^{k+1-(k+2i)/3}\coH_\dR^{k+1}(\Aff{k+1},s^3g_k)$.

Let us now consider the basis $\omega^-_i,\eta^-_j$ of $\coH^1_{\dR}(\Afu,L\otimes\Sym^k\Ai)$. As in the proof of Lemma \ref{lem:Hodge_numbers_odd},
let $\iota: \coH^1(\Afu,L\otimes\Sym^k\Ai) \to \coH^{k+1}(\Aff{k+1}, s^3g_k)$
denote the inclusion in $\EMHS^\mf$.
In this setting, given a form
\[
\omega = z^rx^n\frac{\rd z}{z}\wedge\rd x_1\wedge\cdots\wedge\rd x_k,
\quad
r\geq 0,\
n = (n_i) \in\{0,1\}^k,\ \nu = \textstyle\sum_{i=1}^k n_i,
\]
one has
\begin{align*}
s\varpi^*\omega &\in
\Gamma\big(\wt{X},\Omega^{k+1}(\log\wt{D})
	(\lfloor m^-_{r,\nu}\wt{P}\rfloor - (k-2r-\nu-1)\wt{E}_2 - (k-4r-2\nu-2)\wt{E}_3)\big), \\
m^-_{r,\nu} &= \frac{k+2r+\nu+1}{3},
\end{align*}
provided that $k\geq 4r+2\nu+2$.
In particular,
\[ [s\varpi^*\omega]
\in F_\irr^{k+1-m^-_{r,\nu}}\coH_\dR^{k+1}(\Aff{k+1},s^3g_k)
\quad
\text{if $k\geq 4r+2\nu+2$}. \]
Now we have $\iota(\omega^-_i) = \varpi^*\omega$
with $r=i$, $\nu=0$,
while $\iota(\eta^-_j) = \varpi^*\omega^-$
where $\omega^-$ is the average
(i.e., the symmetric projection of any)
of the forms $\omega$
with $r=0$, $\nu=j$.
The claim follows.\qed

\backmatter
\providecommand{\sortnoop}[1]{}\providecommand{\eprint}[1]{\href{http://arxiv.org/abs/#1}{\texttt{arXiv\string:\allowbreak#1}}}

\end{document}